\documentclass[reqno, 12pt]{amsart}
 \usepackage{xfrac}
\usepackage[utf8]{inputenc}

\usepackage{tikz-cd}
 
\usepackage{amsfonts, amsthm, amsmath, amssymb}
\usepackage{hyperref}
\hypersetup{colorlinks=false}
\usepackage[all]{xy}

\usepackage{pifont}

\usepackage[margin=1in]{geometry}

\RequirePackage{mathrsfs} \let\mathcal\mathscr
  
\newcommand{\set}[2]{\left\lbrace #1 : #2 \right\rbrace}
\newcommand{\real}{\mathbb{R}}
\newcommand{\norm}[1]{\left\lVert#1\right\rVert}

\usepackage[colorinlistoftodos,french,bordercolor=black,backgroundcolor=red,linecolor=red,textsize=footnotesize,textwidth=0.75in,shadow]{todonotes}
 
\usepackage{hyperref}
\usepackage{dsfont}
\usepackage{upgreek}
\usepackage{mathabx,yfonts}
\usepackage{enumerate}
\usepackage{caption}
\numberwithin{equation}{section}

\newtheorem{theorem}{Theorem}[section] 
\newtheorem{lemma}[theorem]{Lemma}
\newtheorem{proposition}[theorem]{Proposition}
\newtheorem{corollary}[theorem]{Corollary}

\newenvironment{claim}[1]{\par\noindent{\bf Claim.}\space#1}{}
\newenvironment{claimproof}[1]{\par\noindent{\it Proof of Claim.}\space#1}{\leavevmode\unskip\penalty9999 \hbox{}\nobreak\hfill\quad\hbox{}}

\theoremstyle{definition}

\newtheorem{remark}[theorem]{Remark}
\newtheorem{definition}[theorem]{Definition}

\newtheorem{problem}[theorem]{Problem}

\begin{document}
\title{The Daniell Integral: Integration without measure}
\author{Adriaan de Clercq}
\maketitle

\begin{abstract}
    In his 1918 paper ``A General Form of Integral'',
    Percy John Daniell developed a theory of integration capable of 
    dealing with functions on arbitrary sets. Daniell’s method differs from the measure-theoretic
     notion of integration. Linear functionals 
     over vector lattices were considered as the fundamental 
     objects on which he built the theory, rather than measures over sets.  
     In this document, we explore Daniell’s concept of integration 
    and how his theory relates to the measure-theoretic notion of integration. 
    We paint a picture of the historical context surrounding Daniell’s ideas. 
    Furthermore, we present examples due to Norbert Wiener, where the Daniell 
    integral was employed on spaces too general for the standard integration 
    techniques of the time.

\end{abstract}

\tableofcontents
\newpage

\section{Introduction}
In this document, we review the theory of integration developed by Percy John Daniell (1889-1946).
His developments came in the period following
the success of Lebesgue's theory of integration based on measure. 
However, measure theory was still in its 
infancy, and an abstract theory of measure had not yet been developed. 
Consequently, as an attempt to generalise the theory of integration, 
Daniell looked at the integral, rather than measure, as the principle object of study. 
In his series 
of papers which began with the 1918 paper ``A General Form of Integral'' \cite{daniell_1}, Daniell provided a
method for extending the ideas of integration to sets of utmost generality. 

We start with, in Section \ref{Section:history}, a brief overview of the history of integration and the motivating 
works behind the important theorems in the 19th- and the early 20th- century 
mathematics. This serves to provide the 
historical context in which Daniell published his papers. 
In Section \ref{section:extension}, we give a full treatment of Daniell's integral as was done in his paper 
``A General Form of Integral''.
We explore, in Section \ref{section:measure}, the conditions under which Daniell's integral is equivalent 
to the modern theory of measure, with the main result being the Daniell-Stone Theorem (Theorem \ref{thm:DS}). 
Daniell was also interested in generalising the Stieltjies integral. We 
discuss such generalisation in Section \ref{section:Sintegral}.  
Furthermore, we prove that under suitable restriction, a linear functional can be decomposed as the difference 
of two positive linear functionals. 
Finally, in Section \ref{Section:Wiener}, we give a brief overview of work done by 
Norbert Wiener (1894-1964) 
which utilised the Daniell Integral.

\section{Historical Background} \label{Section:history}
A brief history of the theory of integration is given in this section. For a more 
extensive overview of the topic we refer the readers to \cite[Chapter~9]{history_analysis} and \cite{Hawkins}.

\subsection{Definition of a function}
The concept of integration and differentiation was understood 
on an intuitive level since the works of 
Gottfried Wilhelm Leibniz (1646-1716) and Isaac Newton (1642-1726/27). 
However, they did not use the term functions. 
Rather, they referred to quantities and the rates of changes of these quantities.
Furthermore, intuitive geometrical notions of limit procedures sufficed as a justification for 
the theory. 
Therefore, the notion of integral was synonymous with area, 
and the notion of derivative was synonymous with tangent. We refer the readers to \cite[Chapter~3]{history_analysis} for 
more details surrounding the development of calculus.  

It was not until Leonhard Euler (1707-1783), that the idea of a function was beginning to form. 
In his 1748 book {\it Introductio in analysin infinitorum}, he defined a function to be an analytical expression 
containing constants and variables. Later on, in his 1755 text 
{\it Institutiones calculi differentialis}, he defined a function to be 
a quantity which depends on another.
Augustin-Louis Cauchy (1789-1857) later used Euler's definition
when he started a new age of rigour with his  
textbook {\it Cours d'analyse de l'\'Ecole Royale Polytechnique} written in 1821.
A year later in 1822, 
Joseph Fourier (1768-1830) published his own definition of a function in his main work 
{\it Th\'eorie analytique de la Chaleur}. The definition was as follows:
\begin{quotation}
    ``In general, the function $f(x)$ represents a succession of values or ordinates each of
which is arbitrary. An infinity of values being given to the abscissa $x$, there are an
equal number of ordinates $f(x)$. All have actual numerical values, either positive
or negative or nul. We do not suppose these ordinates to be subject to a common
law; they succeed each other in any manner whatever, and each of them is given
as it were a single quantity." (Fourier, 1822.) 
\end{quotation}
In essence, Fourier defined a function in terms of its graph. Although this still had the 
geometrical flavour of the previous century, this definition is general enough 
as to include the stranger counterexamples that drove mathematicians of the time to begin treating the 
subject with the necessary care.

In his work, Fourier decomposed a function $f$ into the series
\begin{align}
    f(x)= \frac{1}{2} a_0 + \sum_{n=1}^\infty (a_n \sin (nx)+b_n \cos (nx) ) \label{eqn:fourier_series} 
\end{align}

where the coefficients are given by 
\[
    a_0= \frac{1}{\pi} \int_{-\pi}^\pi f(x) \, dx ,
\]
\[
    a_n= \frac{1}{\pi} \int_{-\pi}^\pi f(x) \sin (nx) \, dx, \text{ and }
    b_n= \frac{1}{\pi} \int_{-\pi}^\pi f(x) \cos (nx) \, dx, \text{ for } n \in \mathbb{N}.  
\]
In his derivations, he assumed that the interchange between an infinite series and integration 
is valid, i.e.that the following can be done to a series of functions
$\sum_{n=1}^\infty s_n(x)$:
\[
    \int_a^b \sum_{n=1}^\infty s_n(x) \, dx = \sum_{n=1}^\infty \int_a^b s_n(x) \, dx.
\]

This brought up two questions. Firstly, what is the meaning of $\int_a^b f(x) \, dx$ for 
`arbitrary' functions $f$ as given by Fourier's definition? Secondly, when can the limit 
of a sequence of functions be interchanged with the integral? 

\subsection{Integrating arbitrary functions}

Cauchy was the first to give an answer to the first question, at least in part. 
In his  {\it Cours d'analyse} he defined functions as variables depending 
on other variables, and he
gave a precursor to the definition of a continuous function as we know it today. In his 1823 work 
{\it R\'esum\'e des lecons donn\'ees \`a l'\'Ecole Royale Polytechnique sur le calcul infinit\'esimal},
Cauchy defined the integral of a continuous
function by first partitioning the interval $[a,b]$ into $n$ parts 
$$a=x_0 <x_1<x_2< \cdots <x_n = b.$$
Thereafter, he would consider the sums 
\begin{align}
    S = \sum_{k=1}^n f(x_{k-1})(x_k-x_{k-1}). \label{eqn:riemann_sum}
\end{align}
Today we would recognise these as `left Riemann sums' and he showed that, as a consequence 
of his definition of continuity, these sums tend to a definite limit\footnote{ Cauchy defined the limit in {\it Cours d'analyse} as: 
``When the values successively attributed to the same variable approach a fixed
value indefinitely, in such a way as to end up by differing from it as little as one
could wish, this last value is called the limit of all the others.''}
 as the partitions are refined. 
He defined the integral to be this definite limit. 

Another analyst who gave careful attention to the 
questions stemming from Fourier's work, was the French mathematician 
Peter Gustav Lejeune Dirichlet (1805-1859). 
He studied these questions with the same rigour as pioneered by Cauchy, and proved 
conditions for when a function can be represented by a 
Fourier series as given in \eqref{eqn:fourier_series}. He prudently restricted 
himself to functions with only finitely many discontinuities, as he was 
already aware of nowhere continuous functions, such as the following, which later 
took on the name `Dirichlet monster':
\[
    f(x)=  \begin{cases}
        0, &\text{ if } x \in \mathbb{Q},\\
        1, &\text{ if } x \in \mathbb{R} \backslash \mathbb{Q}.
    \end{cases}
\]
Dirichlet wanted to know how far the assumptions of continuity 
can be weakened and if the 
integral can be extended to include functions with infinitely 
many discontinuities. 


Georg Friedrich Bernhard Riemann (1826-1866) studied 
in Berlin from 1847 until 1849, finishing 
his doctoral dissertation in 1851. Riemann was 
heavily influenced by Dirichlet, who moved to Berlin 
in 1931 to take up a professorship position. 
Under the guidance of Dirichlet, 
Riemann came up with the definition of an integral, 
which later would come to be known as the `Riemann integral'.
In his 1854 essay\footnote{ This was work done in requirement for 
Riemann's habilitation and was later was published in 1867, 
after his death. A habilitation is a required 
qualification for teaching in some European universities. 
Part of the requirements for a habilitation is a thesis,
similar to a doctoral dissertation,
based on independent scholarship.
} 
{\it \"Uber die Darstellbarkeit einer Function durch eine trigonometrische Reihe},
he detailed this integral. 
He, like Cauchy, considered the interval $[a,b]$ and partitioned it by
\[
    a=x_0 <x_1<x_2< \cdots< x_n = b.
\]
Thereafter, for $k\in \lbrace 1,2,\cdots, n \rbrace$, 
defining $\delta_k=x_k-x_{k-1}$ and letting
$\varepsilon_k \in [0,1]$ be arbitrary, he considered the sums
\[
    S = \sum_{k=1}^n f(x_{k-1}+\varepsilon_k \cdot \delta_k)\cdot \delta_k.
\]
If  the sums tend to  
a definite limit regardless how one goes about choosing the $\varepsilon_k$'s, 
he defined that limit to be the integral of $f$. 

Also appearing in this work, is an example of an integrable function 
(in Riemann's sense) with infinitely many 
discontinuities (Figure \ref{fig:riemann}), 
thus satisfying Dirichlet's original desires of 
extending the concept of integration. 
At the time, this work was seen as being of the utmost generality. 
\begin{figure}
    \includegraphics[width=0.75\linewidth]{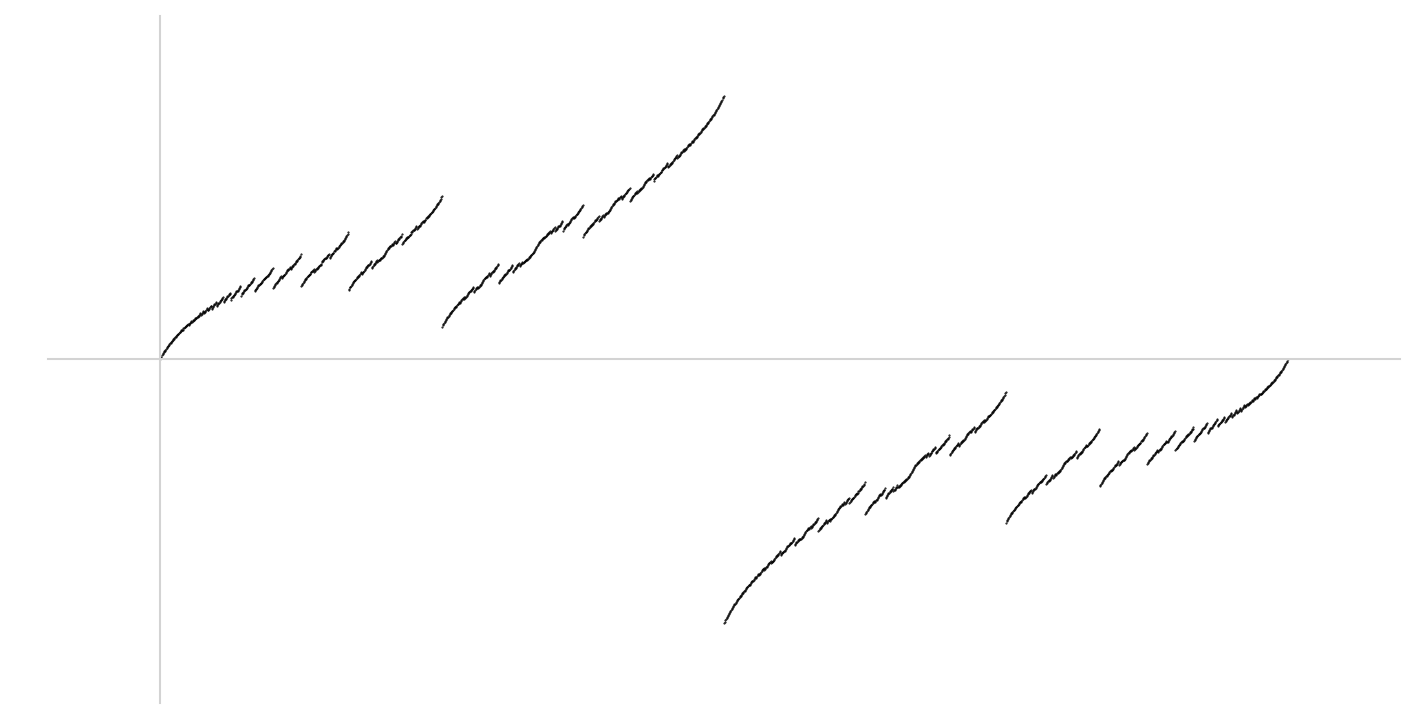}
    \caption{Riemann's example of an integrable function with a dense set of 
    discontinuities (we refer the readers to \cite[Section~9.3, p. 265-266]{history_analysis}
     for a detailed construction).}\label{fig:riemann}
\end{figure}

\subsection{Development of Riemann's ideas}
Enter the French mathematician 
Jean-Gaston Darboux (1842-1917). It is important to 
note that Darboux was a big advocate for the advancement of rigour in the 19th-century 
mathematics, producing numerous 
counterexamples designed to showcase the problem 
with blindly trusting intuitive notions in 
analysis \cite[Section~9.3, p. 269]{history_analysis}. 

In his 1875 work {\it M\'emoire sur les fonctions discontinues}, 
Darboux developed necessary and 
sufficient conditions for a function to be 
Riemann integrable (what we refer to today 
as Darboux integrable). He explicitly only considered bounded functions $f$.
Given a partition $P=\lbrace x_0,x_1,\cdots x_n \rbrace$ 
of the interval $[a,b]$ such that 
$$ a=x_0<x_1<x_2<\cdots < x_n=b,$$
he then defined the well-known upper
and lower sums. If for $i \in \lbrace1,2,\cdots ,n \rbrace$, $M_i$ 
and $m_i$ were given by
$$M_i:= \sup \set{f(x)}{x_{i-1}\leq x\leq x_i}$$
and 
$$m_i:=\inf \set{f(x)}{x_{i-1}\leq x\leq x_i},$$
then the upper sum was defined by 
$$ M(P):=\sum_{i=1}^n M_i (x_i-x_{i-1})$$
and the lower sum by 
$$  m(P):=\sum_{i=1}^n m_i (x_i-x_{i-1}).$$
He also defined 
\begin{align}
    \Delta(P):= M(P)-m(P)=\sum_{i=1}^n(M_i-m_i)(x_i-x_{i-1}). \label{eqn:Delta} 
\end{align}

Darboux was not the first to come up with this definition. 
Before Darboux, Riemann considered the idea of defining the 
`oscillation' of the function 
in the intervals $[x_{i-1},x_i]$ given by $M_i-m_i$. 
Riemann called this oscillation of $f$ in the interval 
$[x_{i-1},x_i]$ `the difference between the 
largest and smallest value in this interval'. However, 
Darboux was the first to treat this concept with 
the required rigour.  Today we know that a function $f$ might not attain 
its maximum/minimum in a given interval, and Darboux 
made note of this in his work, explicitly
referring to the supremum and infimum\footnote{
    The concept of supremum and infimum dates 
    back to Bernard Bolzano (1781-1848). We note that 
    Darboux was not the first to use these concepts. 
}
and exemplifying  
Darboux's careful attention to rigour for the time. 
Given these definitions, he defined the upper integral $\overline{\int_a^b}f(x)\, dx$ to 
be the limit of $M(P)$ and the lower integral $\underline{\int_a^b}f(x)\, dx$ 
the limit of $m(P)$ as the partition $P$ is refined.

He also reformulated Riemann's ideas. 
If we have a partition of $[a,b]$ given by
\[
a= x_0 < x_1 <x_2 < \cdots <x_n=b,
\]
then for $i \in \lbrace 1,2,\cdots, n \rbrace$, let $\delta_i:= x_i - x_{i-1}$, and $\theta_i \in [0,1]$. 
He then defined, what is now known as the `Darboux sum' $S$, as 
\[
    S= \sum_{i=1}^n \delta_i f(x_{i-1}+\theta_i \delta_i).
\]
If the reader compares these sums with \eqref{eqn:riemann_sum}, one notes that 
these are essentially the same. 
However, he took Riemann's ideas further by proving the following theorem\footnote{
    Here, we paraphrase the statement of the original result and write it in a modern 
    setting.}: 
\begin{theorem}
    The Darboux sum (and hence the Riemann sum) of a function $f$ converge 
    if and only if $\Delta{P} \rightarrow 0$ as $P$ is refined (with $\Delta P$ given by \eqref{eqn:Delta}), 
    which is equivalent to 
    \[
        \overline{\int_a^b}f(x)\, dx=\underline{\int_a^b}f(x)\, dx.
    \]
\end{theorem}
Using this characterisation of integrability he proved the following results:
\begin{itemize}
    \item Every continuous function is Darboux (Riemann) integrable. 
    \item If $F(x)= \int_a^xf(t)\, dt$ where $f$ is Darboux (Riemann) 
    integrable, then $F$ is continuous.
    \item Let $F(x)= \int_a^xf(t)\, dt$ where $f$ is Darboux (Riemann) 
    integrable. If 
        $f$ is continuous at $x_0 \in [a,b]$, then $F$ is 
        differentiable at $x_0$ and 
        $F'(x_0)=f(x_0)$.
\end{itemize}
In the same publication, he also showed the following theorem, which finalised a
rigorous proof of the fundamental theorem of calculus. 
\begin{theorem}\label{thm:ftc2}
    If $F$ is differentiable on $[a,b]$, with a bounded and integrable derivative $f$, 
    then $F(x)-F(a)= \int_a^x f(t) \, dt$ for all $x \in [a,b]$. 
\end{theorem}

It is in the clause `bounded and integrable derivative' in the above theorem 
where we encounter the first problem with Riemann's integral. 
The process 
of differentiation was not yet completely reversible.
A differentiable function 
$F$ might produce a derivative $F'$ which is not Riemann integrable.  In 
1881, the Italian mathematician Vito Volterra (1860-1940) published an 
example of a function $V$ 
whose derivative $V'$ is bounded but not Riemann 
integrable (Figure \ref{fig:Volterra}).
The function $V$ is constructed in a way that produces discontinuities of $V'$ on 
the Smith-Volterra-Cantor\footnote{
    Henry John Stephen Smith (1826-1883) was a British mathematician who published work on 
    the Riemann integral, giving examples of when the Riemann integral fails. Of note is his 
    example, published in 1875, of a `meagre' set that has measure zero. This construction 
    resembles that of the Cantor set, which first appeared in 1883 and was named after the father of modern set theory: 
    Georg Ferdinand Ludwig Philipp Cantor (1845-1918). 
} set, a nowhere dense set with positive 
measure (we refer the readers to \cite[Section~3.1]{Hawkins} 
for the detailed construction of the set).  
 
\begin{figure} 
    \includegraphics[width=\linewidth]{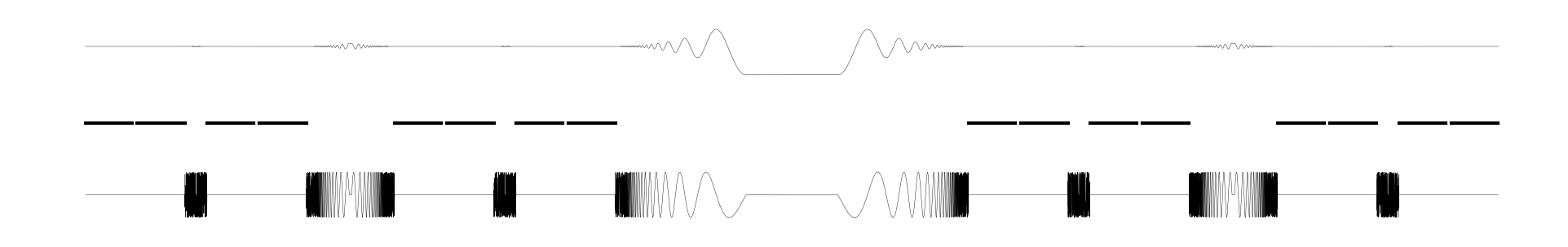}
    \caption{The first three iterations in the construction of the Volterra function. 
    The top is the function $V$. The middle is the Smith-Volterra-Cantor Set. 
    The bottom is the derivative $V'$.}\label{fig:Volterra}
\end{figure}

The second problem with the Riemann integral has to do with the original 
assumption made by Fourier: the exchange of the limit and integral. 
This, in general, is not possible, even with the power of the 
modern Lebesgue integral. An example of this was given by Darboux. 
Consider the telescoping series given by 
\[
    f(x)=-2xe^{-x^2}=\sum_{n=1}^\infty 
        \left[ -2n^2xe^{-n^2 x^2} +2(n+1)^2 x e^{-(n+1)^2 x^2}\right] . 
\]
Note that 
\[
    \int_0^1 f(x) \, dx=  \int_0^1 -2xe^{-x^2} \, dx 
        =  e^{-x^2} \bigg|_0^1= \frac{1}{e}-1,
\]
and 
\begin{align*}
    \sum_{n=1}^\infty \int_0^1\left(-2n^2xe^{-n^2 x^2} +2(n+1)^2 x e^{-(n+1)^2 x^2}\right)\, dx
    &=\sum_{n=1}^\infty \left[(e^{-n^2}-1)-(e^{-(n+1)^2}-1) \right]\\
    &=\sum_{n=1}^\infty \left[e^{-n^2}-e^{-(n+1)^2}\right]=\frac{1}{e} .
\end{align*}
This holds regardless of which integral we use, 
and hence we cannot expect to exchange the integral and 
sum in all cases. However, it was shown by Cesare Arzel\`a (1847-1912) in 
1885 (we refer the readers to \cite[Section~4.4, p. 117]{Hawkins}) 
that if a series of uniformly bounded Riemann integrable 
functions converge uniformly 
towards a Riemann integrable limit, then 
the limit and integral may be interchanged. It is again important to note 
the requirement that 
the limit needs to be Riemann integrable, as a bounded 
sequence of Riemann integrable functions may not 
converge to a Riemann integrable one. 

The final important problem with the Riemann integral comes up when one tries to extend the ideas to 
higher dimensions. Consider a function $f: \mathbb{R}^2\rightarrow \mathbb{R}$. 
Given a domain $A \subset \mathbb{R}^2$, what would be the meaning of 
\[
    \int_A f(x,y) \, dA?  
\]
When considering one dimensional integrals, one typically only considers 
integrals over intervals. Here one begins to 
question the meaning of integrals 
over arbitrary sets. The first thing one may consider 
is covering the set $A$ with rectangles
and then defining a Riemann sum over these rectangles 
and `refining the mesh' (Figure \ref{fig:Jordan}). However, this 
may not always work, as one may not be able to `fit' any 
rectangles inside our set $A$. 
One may be tempted to define the admissible sets to be 
those whose boundary is given by 
a continuous closed curve, but the problems with this 
notion was soon made apparent in
1890 when Giuseppe Peano (1858-1932) discovered the   
`Peano curve', a surjective continuous function 
$P:[0,1]\rightarrow [0,1]^2$. 

\subsection{The beginnings of measure}

The French mathematician Marie Ennemond Camille 
Jordan (1838-1922) was the first to tackle 
the issue of extending the Riemann integral to higher dimensions. 
In his 1892 work {\it Remarques sur les int\'egrates d\'efinies},
Jordan first went about considering a class of sets over which the definition 
of the integral would make sense. 
 
\begin{figure} 
    \minipage{0.32\textwidth}
        \includegraphics[width=\linewidth]{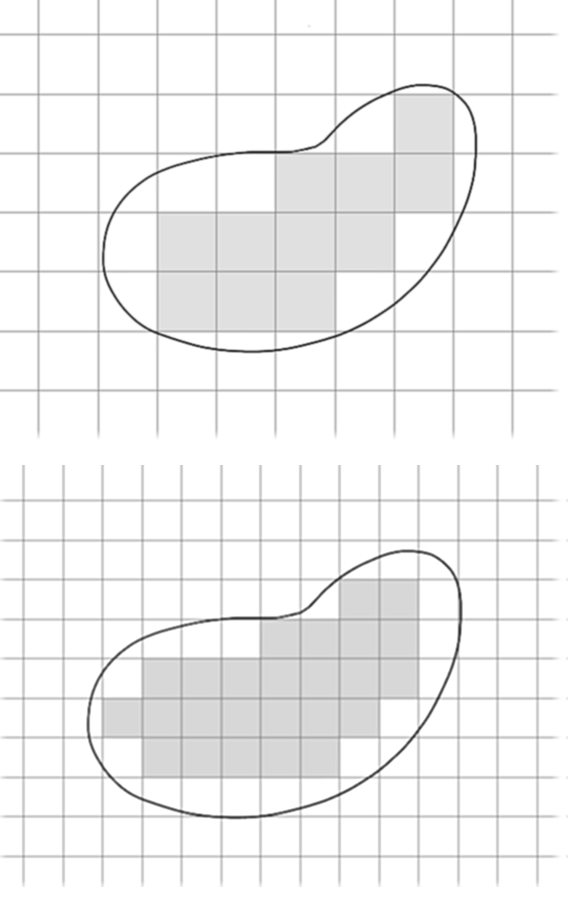}
        \caption*{$S$}
    \endminipage\hfill
    \minipage{0.32\textwidth} 
        \includegraphics[width=\linewidth]{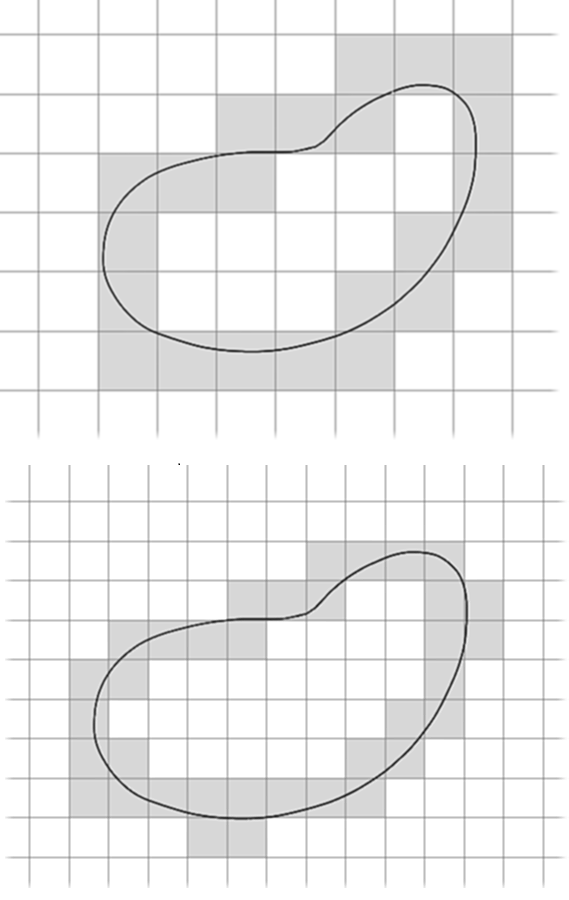}
        \caption*{$S'$}
    \endminipage\hfill
    \minipage{0.32\textwidth}%
        \includegraphics[width=\linewidth]{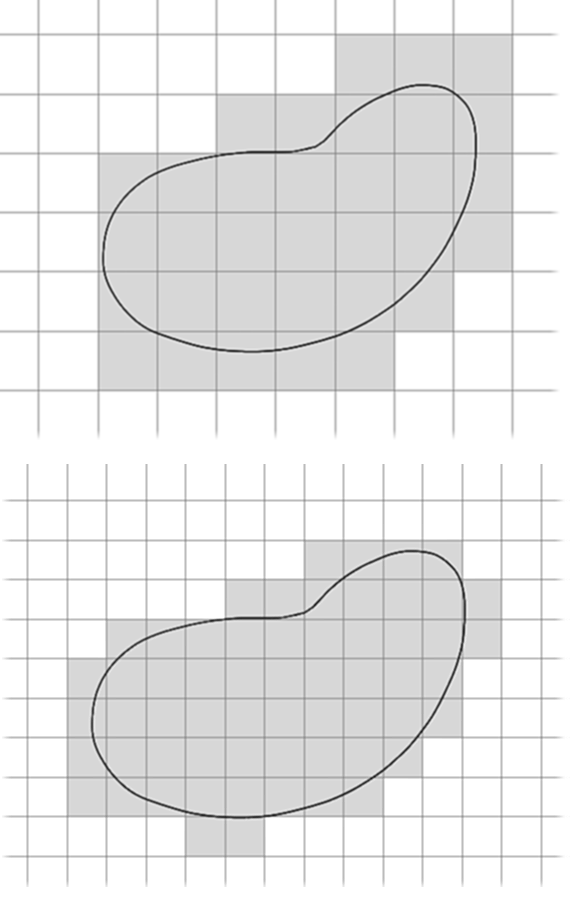}
        \caption*{$S+S'$}
    \endminipage
    \caption{Refining the mesh. }\label{fig:Jordan}
\end{figure}

Jordan's definition of integral starts with an early notion of measure.
For a set $E$ in the plane and a partition $P$, 
let $S$ be 
the sum of the areas of the rectangles defined by $P$ which are
strictly contained in $E$, and $S'$ be 
the sum of the areas of the rectangles containing points of 
both $E$ and its complement. 
The quantity $S+S'$ is then the sum of the areas of the rectangles covering $E$. 
The inner content $\underline{\mathcal{I}}(E)$  is defined to 
be the supremum of the sums $S$ 
as the partition $P$ is refined and similarly, the outer 
content $\overline{\mathcal{I}}(E)$ is defined to be 
the infimum of the sums $S+S'$ as the partition $P$ is refined. 
A set $E$ is (Jordan) {\it measurable} 
when $\underline{\mathcal{I}}(E)=\overline{\mathcal{I}}(E)$, 
and the common value 
is called the {\it content} of the set $E$ and is denoted $\mathcal{I}(E)$. 
Note that 
this concept of measurability is independent of dimension, 
and that it can be applied
to intervals as well.

Once the notion of a (Jordan) measurable set was developed, 
Jordan continued extending the definition 
of the Riemann integral to measurable sets. In his work, the 
influence of Darboux was clearly 
seen. Given a measurable set $E$, and a partition $P$ of $E$ into $n$ disjoint 
measurable sets 
\[
P:= \lbrace E_1, E_2, \cdots, E_n \rbrace,    
\]
for $k \in \lbrace 1,2,\cdots, n \rbrace$, let $M_k:= \sup \set{f(x)}{ x\in E_k}$ 
and $m_k=\inf \set{f(x)}{x \in E_k}$. Similar to Darboux's integral,
 $S(P)$ and $s(P)$ are defined to be the sums 
\[
    S(P):= \sum_{k=1}^n M_k \mathcal{I}(E_k), \text{ and }  
\]
\[
    s(P):= \sum_{k=1}^n m_k \mathcal{I}(E_k).  
\]
The upper and lower sums $\overline{\int_E} f\, dA$ 
and $\underline{\int_E} f \, dA$ are again similarly defined 
to be the limit of $S(P)$ and $s(P)$ as the 
partition $P$ is refined, and the function $f$ is said to be
integrable over $E$ when these 
values coincide. The common value is denoted by $\int_E f \, dA$. It is here 
that the relationship between the integral and the measure of 
sets became apparent. The question now becomes how far one can extend the class 
of measurable sets. In 1893, Jordan incorporated his 
ideas of integration and measure 
into the second edition of his   textbook 
{\it Cours d'Analyse de l'\'Ecole Polytechnique}.
This textbook was widely successful and was read by 
Félix Édouard Justin Émile Borel (1871-1956) and 
Henri Léon Lebesgue (1875-1941), who 
would in the following years
 extend Jordan's ideas of measure to the one we know today. 

Borel completed his doctoral dissertation in 1894 at the 
{\it \'Ecole Normale Sup\'erieure} in 
Paris. His dissertation was on the theory of complex functions, 
and on the convergence of a certain 
series (we refer the readers to \cite[Subsection~9.5.3]{history_analysis} 
for more details). 
His dissertation brought up questions
about sets on which this series converges, and he treated 
this in a series of lectures
given at the {\it \'Ecole Normale} in the academic year of 1896/1897. 
After a successful 
reception, these results were published in 1898, where the notion 
of what is known now as `Borel sets' 
first appeared in the following form: 
\begin{quotation}
    ``If a set consists of a countably infinite totality of disjoint intervals
of entire length $s$, then we say this set has measure $s$. If
two disjoint sets have the measures $s$ and $s'$, then their union
has measure $s + s'$ \dots
More generally: If there is a countably infinite number of
disjoint sets with measures $s_1, s_2, \cdots, s_n, \cdots,$ then their union
has measure $s_1+ s_2 + \cdots + s_n+ \cdots$ \dots  
All this follows from the definition of measure. Now here are
some new definitions: If a set $E$ of measure $s$ contains all points
of a set $E'$ of measure $s'$, then the set $E - E'$ has the measure
$s - s'$. \dots
Those sets are called measurable, to which a measure can be
assigned with the aid of the above definitions. \dots"(Borel, 1898.)
\end{quotation}

From this definition, Borel noted four important properties of his measure.
\begin{itemize}
    \item The measure is countably additive.
    \item The measure of a difference $m(E \setminus E')$ (assuming $E \subset E'$) 
        is the difference of the measures $m(E)-m(E')$.
    \item The measure is always nonnegative.
    \item A set of positive measure is uncountable. 
\end{itemize}

A student at the {\it \'Ecole Normale} during the time of Borel's lecture series 
was Henri Lebesgue.
Lebesgue studied at the {\it \'Ecole Normale} from 1894 to 1897, 
and it is likely that 
he attended the lectures given by Borel. 
Lebesgue completed his doctoral dissertation in 1902, where 
he developed the Lebesgue theory of measure as we know it today.  
In the introduction of his dissertation
he made it clear that he attempted to address the known 
shortcomings of the Riemann integral: 
\begin{quotation}
    ``It is known that there are derivatives which are not integrable,
if one accepts Riemann's definition of the integral; the kind of
integration as defined by Riemann does not allow in all cases to
solve the fundamental problem of calculus:\\
Find a function with a given derivative.

It thus seems to be natural to search for a definition of the integral
which makes integration the inverse operation of differentiation
in as large a range as possible.'' (Lebesgue, 1902.)
\end{quotation}
It is worth noting that Lebesgue was the first mathematician to criticise
Riemann's definition (\cite[Subsection~9.3.1, p. 272]{history_analysis}). Although 
there were many publications exploring the weak points of the Riemann integral, 
these publications were not intended nor interpreted as criticisms of the Riemann 
integral and Riemann's ideas were still viewed as sufficiently general. 

In writing his dissertation, Lebesgue, like Jordan from 
whom Lebesgue took inspiration,
first considered the problem 
of defining measures of sets. Referring to Borel, he required his measure 
$m$ to be countably additive. To be explicit, if $\set{E_i}{i\in I}$ is a collection
of at most countably many disjoint measurable subsets of $\mathbb{R}$, then 
\[
    m \left( \bigcup_{i\in I} E_i \right)=\sum_{i \in I} m(E_i).  
\] 
Furthermore, if $\set{I_j}{j\in J}$ is an at most countably many 
set of intervals such that
$E \subset \bigcup_{j \in I}$, then one would want the following to hold:
\[
    m(E)\leq m\left( \bigcup_{j\in J} I_j \right) \leq \sum_{j \in J} l(I_j), 
\]
with $l(I_j)$ denoting the length of the interval $I_j$.
Consequently, Lebesgue defined the outer measure $m_e(E)$ as 
\[
    m_e(E):= \inf \set{\sum_{j \in J} l(I_j)}{E \subset \bigcup_{j \in J} I_j}.  
\]
It was also known at the time that if $E\subset [a,b]$ was
 measurable in the sense of Jordan (recall 
that Jordan's ideas of measurability could also be 
applied to one dimensional sets), then 
$\mathcal{I}(E)+\mathcal{I}([a,b] \setminus E)= b-a. $ 
This motivated Lebesgue in defining the inner measure 
of a set $E \subset [a,b]$ to be 
\[
    m_i(E):= b-a - m_e([a,b] \setminus E).  
\]
Note that $m_i(E)=m_e(E)$ if and only if $m_e(E)+m_e([a,b] \setminus E)=b-a$, 
and so Lebesgue defined the 
set $E$ to be measurable precisely when this condition holds. 
From this definition, Lebesgue noted,
as the Jordan inner and outer contents were essentially 
coverings by a finite number of intervals, that
\[
    \underline{\mathcal{I}}(E)\leq m_i(E)\leq m_e(E) \leq \overline{\mathcal{I}}(E)  
\]
and concluded that all Jordan measurable sets were also measurable by his new measure. 
Using this more powerful tool for measuring sets, he went further than Borel and applied this generalised
idea to develop a stronger theory of integration. For a function $f$ defined on $[a,b]$ 
with range in $[m,M]$, Lebesgue partitioned 
the range $[m,M]$ instead of partitioning the domain as Riemann had done.
Let $P=\lbrace a_0, a_1, \cdots, a_n \rbrace$ be a partition 
\[
    m=a_0<a_1<\cdots < a_n=M  
\]
of $[m,M]$. Define the sums 
\[
S(P)= \sum_{k=1}^n a_{k} m(f^{-1}([a_{k-1},a_{k}])), 
\]
and 
\[
    s(P)=   \sum_{k=1}^n a_{k-1} m(f^{-1}([a_{k-1},a_{k}])).
\]
Note that if $\norm{P}$ denotes the maximum of 
$a_i-a_{i-1}$ for $i=1,2,\cdots, n$ in the partition $P$ then
\[
    0\leq S(P)-s(P)= \sum_{k=1}^n (a_{k}-a_{k-1}) m(f^{-1}[a_{k-1},a_{k}])  
        \leq \sum_{k=1}^n \norm{P}m(f^{-1}[a_{k-1},a_{k}]) \leq \norm{P} (b-a).
\]
Hence, as the partition is refined and $\norm{P} \rightarrow 0$, 
the sums $S(P)$ and $s(P)$ both converge to the same value.
Lebesgue defined the integral to be the value to which these sums converge. 
The only condition 
required for a function $f$ to be (Lebesgue) integrable is that 
$$f^{-1}([u,v])=\set{x}{u\leq f(x) \leq v}$$ 
needs to be a measurable set for each $u,v \in \mathbb{R}$,
which Lebesgue showed to be a weak requirement. 
He therefore showed that this integral is 
an extension of Riemann's. Furthermore, 
as the concept of measure is easily generalised to higher 
dimensions, his definition of the integral 
also easily carries over to the multi-variable case. 
In his dissertation he proved the following theorems\footnote{
    Once again we are paraphrasing the statement of the original results.}: 
\begin{theorem}
    If a function $f$ on $[a,b]$ has bounded 
    derivative $f'$, then $f'$ is (Lebesgue) integrable
    and $\int_a^b f'(x) \, dx = f(b)-f(a)$. 
\end{theorem}
\begin{theorem}
    If a sequence of functions $(f_n)$ whose absolute 
    value is uniformly bounded from above has limit $f$, 
    then the integral of $f$ is the limit of the integrals of the functions $f_n$. 
\end{theorem}
Recall that the big problems with Riemann's integral were:
\begin{enumerate}[(i)]
    \item the definition does not extend to higher dimensions,
    \item  there are differentiable functions with bounded derivative whose derivative is not Riemann integrable,
    \item and the limit of a sequence of Riemann integrable functions may not be Riemann integrable.
\end{enumerate}
With these theorems Lebesgue, addressed these weakpoints of Riemann's integral.  

\subsection{The context of Daniell's papers}

In the period following Lebesgue's dissertation, his work 
was developed by himself and many other mathematicians. 
Pierre Joseph Louis Fatou (1878-1929)
proved the Parseval's equality for 
the Fourier coefficients for general measurable functions in 1906. 
Guido Fubini (1879-1943) 
published his results pertaining to iterated integration in 1907. 
Most important 
of all, Frigyes Riesz (1880-1956) solidified Lebesgue's 
theory in the heart of modern analysis 
by proving the Riesz-Fischer representation theorem in 1907 and introducing the 
$L^p$-spaces in 1910. 

During this time, the notion of measure on the real line started becoming more 
abstract; the underlying properties required of the sets and the set functions in the theory of 
integration were starting to become clear. Furthermore, mathematicians in France and Italy became interested 
in what we today would call `functionals' (\cite[Epilogue, p. 182]{Hawkins}). 
The Stieltjies integral was known at the time and Riesz proved in 1909 that 
it could be extended to functions of bounded variation. Riesz 
showed that if $U$ is a linear functional such that 
$U(f_n) \rightarrow U(f)$ whenever $f_n$ tends to $f$ uniformly, then 
one can find a function $g$ of bounded variation such that $U$ is equal to the 
Stieltjies integral
\[
    U(f)= \int_a^b f(x) \, dg(x).
\]
Beyond this, in 1910 Lebesgue published a paper where he extended his ideas
to multiple dimensions. He made the important observation that integrals may 
just as well be considered as additive set functions. 
If $E$ is a measurable set in $n$-dimensional space, and $f$ is a fixed measurable
function, one can define a set function 
\[
    F(E):=\int_E f \,dA.  
\]
Lebesgue made the following important observations:
If $(E_n)$ is a sequence of subsets, then 
\begin{itemize}
    \item if $m(E_n)\rightarrow 0$, then $F(E_n)\rightarrow 0$; and
    \item if the sets $E_n$ are pairwise disjoint, then 
    $F(\bigcup_{n=1}^\infty E_n)=\sum_{n=1}^\infty F(E_n)$.
\end{itemize}

This observation was the key for the Austrian 
mathematician Johann Karl August Radon (1887-1956) in 1913 to 
merge the Stieltjies integral with the ideas of Lebesgue. He abstracted 
the requirements needed in Lebesgue's theory, only requiring an additive 
set function, and by doing so generalised the Lebesgue theory even further. 

The development of the idea of `functionals' and 
the movement towards abstraction during this period motivated 
Percy John Daniell (1889-1946) towards a more general 
form of integral, first introduced in 1918. 
Daniell explicitly makes 
reference to the publication of Radon in 1913 and his 
generalisation of the Stieltjies integral. 
He states in his introductory paragraph that his work is an 
attempt to abstract the notion of integration even further,
 to functions of elements of arbitrary nature. 
\begin{quotation}
    ``In this paper a theory is developed which is independent of the nature 
    of the elements. They may be points in space of a denumerable number
    of dimensions or curves in general or classes of events so far as the 
    theory is concerned.'' (Daniell, 1918.)
\end{quotation}
In his series of papers, Daniell developed a notion of integration that did not 
rely on a measure theory of the underlying sets, and only used properties of 
functionals and spaces of functions. 
From this he derived many of the important integration results, 
including 
all the standard integration theorems due to Lebesgue, and even an analogue of 
the Radon-Nikodym theorem \cite[Theorem~3.4]{daniell1924deriative}. It is worth noting that Daniell, in his paper
\cite{daniell_infinite}, was the first to produce examples of sufficiently general integrals of functions 
on a denumerably 
infinite number of dimensions. This development by Daniell was later rediscovered by 
the Soviet mathematician Andrey Kolmogorov (1903-1987) in the context of probability theory and 
stochastic processes
and is now known as the Kolmogorov extension theorem which was proved in 1933. 

The application of Daniell’s integral to stochastic processes was noted long before Kolmogorov.
Norbert Wiener (1894-1964) in \cite{wiener_mean} and 
\cite{wiener_differential}, used Daniell’s theory as the framework on which 
he developed the rigorous treatment 
of Brownian Motion. We discuss this in detail in Section \ref{Section:Wiener}.

In a more modern setting, the Daniell integral 
finds applications in the theory of measure on compact 
topological groups. Vladimir Bogachev, in his 
book on measure theory \cite[p. 445]{bogachev2007measure}, states
\begin{quotation}
    ``Daniell’s construction turned out to be very
    efficient in the theory of integration on locally compact spaces. It enabled
    one to construct the integral without prior constructing measures, which is
    convenient when the corresponding measures are not $\sigma$-finite. This was manifested
    especially by the theory of Haar measures. In that case, it turned
    out to be preferable to regard measures as functionals on spaces of continuous
    functions.'' (Bogachev, 2007.)
\end{quotation}

In the years following Daniell's papers there 
was a widespread point of view that the 
Daniell integral is the more appropriate way of 
teaching the subject of integration \cite[p. 446]{bogachev2007measure}. 
This is due to the fact that using Daniell's method, the 
standard theorems quickly arises. 
Teaching in this way was thus perceived to be more economical. 
However, in applications,
measures turn out to be the principle object of study, 
and thus in a Daniell-based treatment of the subject
one needs to prove these theorems in any case. 
It is therefore that today the pedagogical 
treatment of the Daniell integral 
is less important than that of measure. 

Today the Daniell integral serves as a useful tool 
when considering measures as functionals on spaces 
of continuous functions. In the words of Bogachev in \cite[p. 446]{bogachev2007measure}:
\begin{quotation}
    ``Certainly, for the researchers in measure theory and functional
analysis, acquaintance with Daniell’s method is necessary for broadening the
technical arsenal." (Bogachev, 2007.)
\end{quotation}

\section{Daniell's Integral and Extension} \label{section:extension}  

In this section, we review the extension procedure as 
described in the first of Daniell's papers on discussing his theory of integration \cite{daniell_1}.
Let $\overline{\mathbb{R}}:= \mathbb{R}\cup \lbrace -\infty, \infty \rbrace$ 
be the extended real numbers. 
Let ${\overline{\mathbb{R}}}^X$ be given by 
\[
    {\overline{\mathbb{R}}}^X:= \set{x}{x \text{ is a function from } 
            X \text{ to } \overline{\mathbb{R}}}.
\]

For $x,y \in {\overline{\mathbb{R}}}^X$ we say $x \leq y$ if $x(t) \leq y(t)$ 
for all $t \in X$. 
Thus $({\overline{\mathbb{R}}}^X, \leq )$ forms a partially ordered set. In fact, it 
forms a lattice with the meet and join given as follows:  
for  $x,y \in {\overline{\mathbb{R}}}^X$ 
$$ 
   (x\land y)(t)= \min \lbrace x(t), y(t) \rbrace, \text{ for } t \in X, 
$$
and
\[
  (x\lor y)(t)= \max \lbrace x(t), y(t) \rbrace, \text{ for } t \in X.
\]
\begin{remark}
    In the source material, meet $\land$ and join $\lor$ is 
    referred to as {\it logical addition} and 
    {\it logical product} respectively. We keep the 
    terminology meet and join throughout this 
    document.
\end{remark}

Let $T_0 \subset {\overline{\mathbb{R}}}^X$ be a class of functions 
that is closed under $\land$, $\lor$, and linear combinations. 
More explicitly, if $x,y \in T_0$ and $\alpha,\beta \in \mathbb{R}$, 
then $\alpha x + \beta y$, $x \lor y$, and 
$x \wedge y$ are also in $T_0$.\\

\begin{remark}\label{remark:ambig}
    If it were the case that all the functions are 
finite-valued, this class would be called a {\it vector lattice}. Care is taken to 
allow functions which possibly take on infinite values. In the 
book {\it Real Analysis} by H. Royden \cite{royden1988real} such a set 
is also called a vector lattice. However, we use the terminology {\it extended vector lattice} to make 
explicit note of the difference.
We thus need clarification in the ambiguous case when the 
addition is of the form ``$\infty - \infty$''. 
Let $x,y$ be functions in $T_0$. Let $P_{x+y}$ be the set 
where addition is ambiguous. More explicitly, 
\[
    P_{x+y}:= \lbrace t \in X: x(t)= \infty \text{ and } y(t)=-\infty  \text{, or } x(t)= -\infty \text{ and } y(t)=\infty\rbrace. 
\]
For any $c \in \overline{\mathbb{R}}$, we require the following to be in $T_0$: 
\[
    f(t):=  \begin{cases}
        x(t)+y(t), &\text{ if } t \not \in P_{x+y}\\
        c, &\text{ if }t \in P_{x+y}\\
    \end{cases} .
\]
Furthermore, for $x,y \in T_0$ we define $x+y$ as follows: For $t\in X$ 
\[
    (x+y)(t):=  \begin{cases}
        x(t)+y(t), &\text{ if } t \not \in P_{x+y}\\
        0, &\text{ if } t \in P_{x+y}\\
    \end{cases}.
\]
\end{remark}

If $(x_n)$ is a sequence in $\left( {\overline{\mathbb{R}}}^X, \leq \right)$ 
and $x \in {\overline{\mathbb{R}}}^X$, then we 
use the notation $ (x_n)\searrow x$ (respectively $(x_n)\nearrow x$) 
to mean a pointwise decreasing (respectively increasing)
sequence with a pointwise limit $x$. Similarly, if $(\alpha_n)$ 
is a sequence in $\mathbb{R}$ and $\alpha \in \mathbb{R}$, we 
use the notation $ (\alpha_n)\searrow \alpha$ (respectively $(\alpha_n)\nearrow \alpha$) 
to mean a decreasing (respectively increasing)
sequence with limit $\alpha$.\\

If $(x_n)$ is a sequence in $\left({\overline{\mathbb{R}}}^X, \leq \right)$, 
 then $\bigvee \set{x_n}{n \in \mathbb{N}}$ and 
$\bigwedge \set{x_n}{n \in \mathbb{N}}$ always exists. We define 
\[
    \sup_n x_n:= \bigvee \set{x_n}{n \in \mathbb{N}}; 
\] 
and
\[
    \inf_n x_n:= \bigwedge \set{x_n}{n \in \mathbb{N}}.    
\]
Note that the supremum and infimum is given by the following:
\[
    \left(\sup_n x_n \right)(t)=\sup_n x_n(t) \text{ for } t \in X;      
\]
and 
\[
    \left(\inf_n x_n\right)(t)=\inf_n x_n(t) \text{ for } t \in X.  
\]

\begin{definition}\label{def:I_Integral}
    Let $I:T_0 \rightarrow \mathbb{R}$. 
    We call $I$ an {\it $I$-integral} if for all $x,y \in T_0$, $\alpha, \beta \in \mathbb{R}$, and sequences $(x_n) \subset T_0$:
    
    \begin{enumerate}[(D1)]
        \item $I(\alpha x + \beta y)= \alpha I(x)+ \beta I(y)$;
        \item $(x_n) \searrow 0$ implies $\lim_{n \rightarrow \infty} I(x_n)=0$; and
        \item if $0\leq x$ then $0\leq I(x)$.\\
    \end{enumerate}
    
\end{definition}
 

\subsection{Extending the $I$-integral}
Throughout this section, let $T_0 \subset {\overline{\mathbb{R}}}^X$ be
an extended vector lattice and 
$I: T_0 \rightarrow \mathbb{R}$ be an $I$-integral.
\begin{definition} \label{definition:Tinc}
    Define $T_1 \subset {\overline{\mathbb{R}}}^X $ as follows: 
    \[
        T_1:= \lbrace x \in {\overline{\mathbb{R}}}^X: \text{ there exists a sequence } (x_n) \subset T_0 \text{ such that }
        (x_n) \nearrow x\rbrace.
    \]
\end{definition}

We aim to show that:
\begin{enumerate}[({A}1)]
    \item $I$ extends to $T_1$; and 
    \item if $(x_n)$ is a sequence in $T_1$ such that $(x_n) \nearrow x$, then $x\in T_1$. \\
\end{enumerate}

Note that $T_1$ is not an extended vector lattice since $T_1$ 
is only closed under nonnegative scalar multiplication. 
However, $T_1$ is closed under addition, $\land$, and $\lor$. 
For if $x$ and $y$ are in $T_1$ with
$(x_n)$ and $(y_n)$ sequences in $T_0$ such that 
$(x_n)\nearrow x$ and $(y_n)\nearrow y$, 
then $(x_n +y_n)\nearrow x+y, (x_n \lor y_n)\nearrow x \lor y,\text{ and } 
(x_n \wedge y_n)\nearrow x \wedge y .$\\

\begin{lemma}\label{lemma:Dprime}
    Let $x \in T_1$ and $(x_n)$ be a sequence in $T_0$ 
    such that $(x_n) \nearrow x$. If $h \in T_0$ 
    such that $h \leq x$, then $I(h)\leq \lim_{n \rightarrow \infty} I(x_n) $ 
    (where $\lim_{n \rightarrow \infty} I(x_n)$ might possibly be infinite).
\end{lemma}

\begin{proof}
    Let $x \in T_1$ and $(x_n)$ be a sequence in $T_0$ such that $(x_n) \nearrow x$. Let $h \in T_0$ 
    such that $h \leq x$.
    For $n \in \mathbb{N}$, define $h_n:=x_n \wedge h$. Then for 
    each $n\in \mathbb{N}$,
     $h_n \in T_0$ and $h_n \leq x_n$. Also note that $(h_n)$ is an increasing sequence.
     We then have 
     $$\sup_n h_n=\sup_n (x_n \wedge h) = \left(\sup_n x_n\right) \wedge h=x\wedge h=h,$$ 
    and thus $(h_n) \nearrow h$. 
    Therefore, $(h-h_n)\searrow 0$ and $\lim_{n \rightarrow \infty} I(h-h_n)=0$ 
    by (D2) of Definition~\ref{def:I_Integral}. But 
    \[
    0=\lim_{n \rightarrow \infty} I(h-h_n)= \inf_n [I(h)-I(h_n)] =I(h)- \sup_n I(h_n),
    \]
    and so $I(h)=\sup_n I(h_n)$. 
    For all $n \in \mathbb{N}$ we have that $h_n \leq x_n$ 
    and thus by (D3) of Definition~\ref{def:I_Integral} that
    $I(h_n) \leq I(x_n)$. Hence  
    $$I(h)=\sup_n I(h_n) \leq \sup_n I(x_n)=\lim_{n \rightarrow \infty} I(x_n),$$ 
    proving the result.
\end{proof}

\begin{lemma} 
    \label{lemma:limit}
    Let $x,y \in T_1$ with $(x_n)$ and $(y_n)$ sequences in $T_0$ 
    such that $(x_n) \nearrow x$ and $(y_n) \nearrow y$. If $y \leq x$ then 
    \[
    \lim_{n \rightarrow \infty} I(y_n) \leq \lim_{n \rightarrow \infty} I(x_n).
    \]
\end{lemma}

\begin{proof}
    For each $m \in \mathbb{N}$ we have $y_m \in T_0$ and $y_m \leq y \leq x$. By Lemma \ref{lemma:Dprime}, 
    we know that $I(y_m) \leq \lim_{n \rightarrow \infty}  I(x_n)$. Thus 
    $$\lim_{m \rightarrow \infty} I(y_m) =\sup_m I(y_m) \leq \lim_{n \rightarrow \infty} I(x_n),$$
    proving the result.
\end{proof}

\begin{corollary} \label{cor:well_def}
    If $(x_n)$ and $(y_n)$ are both 
sequences in $T_0$ that increase to the same vector $x \in T_1$, 
then $\lim_{n \rightarrow \infty} I(x_n)=\lim_{n \rightarrow \infty} I(y_n)$.
\end{corollary}

\begin{definition} \label{definition:seqi}
    Let $I_1: T_1 \rightarrow \mathbb{R}\cup \lbrace\infty \rbrace$ be defined as follows: 
    if $x \in T_1$ with $(x_n)$ a sequence in $T_0$ such that $(x_n) \nearrow x$, then define
    $I_1(x) := \lim_{n \rightarrow \infty} I(x_n)$.
\end{definition}
Note that if $x \in T_0$, then $x \in T_1$ and $I_1(x)=I(x)$. 
This completes our first goal (A1)
of extending $I$ to $T_1$. 
We proceed to our second goal (A2) with the following lemma.

\begin{lemma} \label{lemma:incr_closed}
    If $x\in {\overline{\mathbb{R}}}^X$ and $(x_n)$ is a sequence in $T_1$ such that $(x_n) \nearrow x$, 
    then $x \in T_1$ and $I_1(x)=\lim_{n \rightarrow \infty} I_1(x_n)$.
\end{lemma}

\begin{proof}
    Let $r \in \mathbb{N}$. We have $x_r \in T_1$. Therefore, there exists 
    a sequence $(x_{r,k})$ in $T_0$ such that $(x_{r,k}) \nearrow x_r$. 
    Then for each $n \in \mathbb{N}$, define 
    $$g_n:=\sup \lbrace x_{r,k}: r,k\leq n \rbrace. $$ 
    Note that, for all $n \in \mathbb{N}$, $g_n \leq x$, and since $g_n$ is a finite join of functions in $T_0$, $g_n \in T_0$. 
    Also note that $(g_n) \subset T_0$ is increasing,
     and $\lim_{n \rightarrow \infty} g_n \in T_1$ by definition of $T_1$. 
    We also have $g_n \leq x_n$ for all $n \in \mathbb{N}$, and thus the following limits hold pointwise:
    \begin{align}
        \lim_{n \rightarrow \infty} g_n \leq \lim_{n \rightarrow \infty} x_n=x. \label{eqn:incr_closed_1}
    \end{align}
    If $r,n\in \mathbb{N}$ such that $r \leq n$, then we have that $x_{r,n}\leq g_n$ 
    and so the following limits hold pointwise:
    \begin{align}
        x_r=\lim_{n \rightarrow \infty} x_{r,n}\leq \lim_{n \rightarrow \infty} g_n. \label{eqn:incr_closed_2}
    \end{align}
    This holds for all $r \in \mathbb{N}$ and so, combining \eqref{eqn:incr_closed_1} and \eqref{eqn:incr_closed_2}, we have 
    \[
        x= \lim_{r \rightarrow \infty} x_r \leq   \lim_{n \rightarrow \infty} g_n\leq \lim_{n \rightarrow \infty} x_n=x,
    \] and the inequalities becomes equalities. Therefore $(g_n)$ is a sequence in $T_0$ increasing to $x$, hence $x \in T_1$. 
    Since $g_n\leq x_n$ for all $n \in \mathbb{N}$, Lemma \ref{lemma:limit} implies 
    that $I(g_n)=I_1(g_n) \leq I_1(x_n)$ and so 
    \begin{align}
        I_1(x)=\lim_{n \rightarrow \infty} I(g_n) \leq \lim_{n \rightarrow \infty} I_1(x_n). \label{eqn:incr_closed_3}
    \end{align}
    If we fix $r \in \mathbb{N}$, then $x_{r,n}\leq g_n $ for all $n \geq r$ and it follows
    that $$I_1(x_r)=\lim_{n \rightarrow \infty} I(x_{r,n})\leq \lim_{n \rightarrow \infty} I(g_n) =I_1(x).$$
    Thus $\lim_{r \rightarrow \infty} I_1(x_r)\leq I_1(x)$. This, together with \eqref{eqn:incr_closed_3},
    allows us to conclude that \break $I_1(x)=\lim_{n \rightarrow \infty} I_1(x_n)$.
\end{proof}

\begin{remark}
    Recall that $T_1$ is closed with respect to addition and 
positive scalar multiplication. We see that $I_1$ also respects these operations.
If $x,y \in T_1$ with $(x_n)$ and $(y_n)$
sequences in $T_0$ such that $(x_n)\nearrow x$ and $(y_n)\nearrow y$, then $(x_n+y_n)\nearrow x+y$ and so 
$$I_1(x+y)=\lim_{n \rightarrow \infty}I(x_n+y_n)=\lim_{n \rightarrow \infty}[I(x_n)+I(y_n)]
=\lim_{n \rightarrow \infty}I(x_n)+\lim_{n \rightarrow \infty}I(y_n)=I_1(x)+I_1(y).$$ 
Furthermore if $c\in \mathbb{R}$ with $c \geq 0$, then $(cx_n) \nearrow cx$ and 
$$I_1(cx)= \lim_{n \rightarrow \infty}I(cx_n)=c\lim_{n \rightarrow \infty}I(x_n)=cI_1(x).$$
\end{remark}

\subsection{The class of integrable functions}

Let $T_0 \subset {\overline{\mathbb{R}}}^X$ once again denote an extended vector lattice and 
$I: T_0 \rightarrow \mathbb{R}$ an $I$-integral. Furthermore, we denote by $T_1$ the 
extension of $T_0$ as in Definition \ref{definition:Tinc} and $I_1$ the extension of $I$ to $T_1$.
\begin{definition}\label{def:upper_lower}
    For $x \in {\overline{\mathbb{R}}}^X$ define:
    \begin{align*}
        \overline{I}(x) := \inf \lbrace I_1(\varphi) : x\leq \varphi, \varphi \in T_1 \rbrace 
        \text{ and }\underline{I}(x) := -\overline{I}(-x) \text{.}
    \end{align*}
    We say $x \in {\overline{\mathbb{R}}}^X$ is {\it $I$-integrable} if 
    $\overline{I}(x)=\underline{I}(x)$ and this value is finite. 
    We denote the common value by $\int x$ and the class of integrable functions by $\mathcal{L}$.
    Thus we have
    $$\mathcal{L}:= \left\lbrace x \in {\overline{\mathbb{R}}}^X  : -\infty < \underline{I}(x)=\overline{I}(x) < \infty \right\rbrace$$
    and $\int: \mathcal{L} \rightarrow \mathbb{R}$ defined by 
    \begin{align*}
        \int x &:= \underline{I}(x)=\overline{I}(x). 
    \end{align*}
        
\end{definition}
Note that if $x \in T_0$, then $x \in T_1$ and so 
\[
    \overline{I}(x)=I_1(x)=I(x).  
\]
Furthermore, since $T_0$ is an extended vector lattice, $-x \in T_0$ and 
\[
    \underline{I}(x)=-\overline{I}(-x)=-I_1(-x)=-I(-x)=I(x).  
\]
Consequently $\overline{I}(x)=\underline{I}(x)=I(x)$ and we see that $x \in \mathcal{L}$ with $\int x =I(x).$ Thus $\int$ 
extends $I$. 
\begin{lemma} \label{BigLemma}
    If $c \in \real$ with $ c \geq 0$ and $x_1,x_2 \in \overline{\mathbb{R}}^X$,
    then: 
    \begin{enumerate}[(i)]
        \item $\overline{I}(cx_1)=c\overline{I}(x_1)$;
        \item $\overline{I}(x_1+x_2)\leq \overline{I}(x_1) +\overline{I}(x_2)$;
        \item if $x_1\leq x_2$ then $\overline{I}(x_1) \leq \overline{I}(x_2)$;
        \item if $\underline{I}(x_1)$ is finite, then $\underline{I}(x_1) \leq \overline{I}(x_1)$;
        \item $\overline{I}(x_1 \vee x_2)+\overline{I}(x_1 \wedge x_2) \leq
            \overline{I}(x_1)+\overline{I}(x_2)$; and
        \item $0 \leq \overline{I}(|x_1|)-\underline{I}(|x_1|)\leq \overline{I}(x_1)-\underline{I}(x_1)$.
    \end{enumerate}
\end{lemma}

\begin{proof}
\hfill

     \begin{enumerate}[(i)]
        \item If $c \geq 0$, then 
            \begin{align*}
            \overline{I}(cx_1)&= \inf \set{I_1(\varphi)}{cx\leq \varphi, \varphi \in T_1}\\
            &=\inf \set{I_1(c\varphi)}{cx\leq c\varphi, c\varphi \in T_1}\\
            &=c \inf \set{I_1(\varphi)}{x\leq \varphi, \varphi \in T_1}\\
            & =c\overline{I}(x_1) \text{.}
            \end{align*}
        \item 
        If one of the sets $\lbrace I_1(\varphi) : x_1\leq \varphi, \varphi \in T_1 \rbrace$ or 
        $\lbrace I_1(\varphi) : x_2\leq \varphi, \varphi \in T_1 \rbrace$ is empty, then one of 
        $\overline{I}(x_1)$ or $\overline{I}(x_2)$ is infinite and the inequality holds. 
        
        Assume $\lbrace I_1(\varphi) : x_1\leq \varphi, \varphi \in T_1 \rbrace$ and 
         $\lbrace I_1(\varphi) : x_2\leq \varphi, \varphi \in T_1 \rbrace$ 
        are both non-empty. Let $\varphi_1, \varphi_2 \in T_1$ be arbitrary such that 
            $x_1 \leq \varphi_1$ and $x_2 \leq \varphi_2$. 
            Then $\varphi_1+\varphi_2 \in T_1$ with 
            $ x_1+x_2 \leq \varphi_1+\varphi_2$ and therefore 
            $$\overline{I}(x_1+x_2)\leq I_1(\varphi_1+\varphi_2)=I_1(\varphi_1)+I_1(\varphi_2).$$
            Hence we have that
            $\overline{I}(x_1+x_2)\leq \overline{I}(x_1) +\overline{I}(x_2)$.

        \item Assume $x_1 \leq x_2$. If $\varphi \in T_1$ such that $x_2\leq \varphi$, then 
            $x_1\leq \varphi$ and  $\overline{I}(x_1) \leq I_1(\varphi)$.
            Thus $\overline{I}(x_1)$ is a lower bound for $\lbrace I_1(\varphi) : x_2 \leq \varphi, \varphi \in T_1 \rbrace$ 
            and  $\overline{I}(x_1) \leq \overline{I}(x_2)$.
        
        \item 
        Assume $\underline{I}(x_1)$ is finite. Then $0 = \overline{I}(0) =\overline{I}(x_1-x_1)$ and by (ii) we know that 
            $$0\leq \overline{I}(x_1-x_1) \leq \overline{I}(x_1)+\overline{I}(-x_1)=\overline{I}(x_1)-\underline{I}(x_1).$$
            It follows that $\underline{I}(x_1) \leq \overline{I}(x_1)$. 
        
        \item Let $\varphi_1,\varphi_2 \in T_1$ such that $x_1 \leq \varphi_1$ and 
            $ x_2\leq \varphi_2$. Then $ x_1 \vee x_2\leq \varphi_1 \vee \varphi_2$ 
            and $x_1 \wedge x_2\leq \varphi_1 \wedge \varphi_2 $. 
            Thus 
            \[
               \overline{I}(x_1 \vee x_2)+\overline{I}(x_1 \wedge x_2) \leq 
            I_1(\varphi_1 \vee \varphi_2)+I_1(\varphi_1 \wedge \varphi_2)=I_1(\varphi_1)+I_1(\varphi_2). 
            \]
            Since this holds for any  $\varphi_1, \varphi_2 \in T_1$ such that $x_1 \leq \varphi_1$
            and $x_2 \leq \varphi_2$, we have the result.

        \item We know that $| x_1| = x_1\vee (-x_1)$ and $-|x_1| = x_1\wedge(-x_1)$. Thus by (v) above
            $$\overline{I}(|x_1|)+\overline{I}(-|x_1|)=\overline{I}(x_1\vee (-x_1))+\overline{I}(x_1\wedge(-x_1))\leq \overline{I}(x_1)+\overline{I}(-x_1),$$ 
            whence
            $$\overline{I}(|x_1|)-\underline{I}(|x_1|)\leq \overline{I}(x_1)-\underline{I}(x_1).$$
            By (iv) this is greater than 0.   \qedhere
    
        \end{enumerate}

\end{proof}

\begin{remark} \label{rem:absolute_val}
    We are now in a position to prove that $\mathcal{L}$ is an extended vector lattice of extended real-valued functions
    and that $\int$ is an $I$-integral on $\mathcal{L}$. Before we do so, we mention the following vector lattice
    identities:
    \begin{align}
        x \vee y= \frac{1}{2}(x+y+|x-y|), \label{eqn:lattice_1}
    \end{align}
    and
    \begin{align}
        x \wedge y =\frac{1}{2}(x+y-|x-y|) \label{eqn:lattice_2}.
    \end{align}
    Identities \eqref{eqn:lattice_1} and \eqref{eqn:lattice_2} 
    also hold in extended vector lattices from Remark \ref{remark:ambig}.
    Recall that, since we are dealing with possibly infinite-valued functions, the 
    right hand sides of \eqref{eqn:lattice_1} and \eqref{eqn:lattice_2} may not be defined. 
    This is solved by letting any ambiguity equal 0.
    To illustrate, consider the case when $x(t_0)= \infty$, and $y(t_0)=-\infty$ for some $t_0 \in X$. 
    Then $(x\lor y )(t_0) = \infty$. Now $x(t_0)+y(t_0)$ is not defined, however if $f$ is 
    the function from Remark \ref{remark:ambig} with $f(t)= 0$ on $P_{x+y}$ and 
    $f(t)=x(t)+y(t)$ elsewhere, then we have 
    $$\frac{1}{2}(f(t_0)+|x(t_0)-y(t_0)|)= (x\lor y )(t_0)=\infty .$$
    The other cases can be treated similarly. 
\end{remark}

\begin{proposition} \label{thm:lattice}
    The class $\mathcal{L}$ is an extended vector lattice of functions and $\int$ is linear.
\end{proposition}

\begin{proof}
    From \eqref{eqn:lattice_1} and \eqref{eqn:lattice_2} it is sufficient for us to show
    $\mathcal{L}$ to be closed under addition, scalar multiplication, 
    and absolute value to conclude 
    that $\mathcal{L}$ is an extended vector lattice.

    We first prove that $\mathcal{L}$ is closed with respect to scalar multiplication.
    Let $x$ be integrable and $c \in \mathbb{R}$. Recall from Lemma \ref{BigLemma} (i) 
    that $\overline{I}(cx)=c\overline{I}(x)$ for nonnegative scalars $c$. 

    {\bf Case 1: }If $c$ is a nonnegative scalar then 
    $$\overline{I}(cx)=c\overline{I}(x)=c\int x,$$
    and 
        $$\underline{I}(cx)=
        -\overline{I}(c(-x))=c(-\overline{I}(-x))=c\underline{I}(x)=c \int x.$$ 
    Thus $cx \in \mathcal{L}$ and 
    $\int cx= c \int x$. 

    {\bf Case 2: }If $c$ is negative, 
    $$\overline{I}(cx)=(-c) \overline{I}(-x)=c\underline{I}(x)=c \int x,$$
    and
       $$ \underline{I}(cx)=-\overline{I}(-cx)=c\overline{I}(x)= c\int x.$$ 
    Thus $cx \in \mathcal{L}$
    and $\int cx= c \int x$.\\

    We now prove $\mathcal{L}$ to be closed under addition.
    Let $x,y \in \mathcal{L}$. Then by Lemma \ref{BigLemma} (ii) we have
    $$\overline{I}(x+y) \leq \overline{I}(x)+\overline{I}(y)=\int x +\int y.$$
    Also, again using Lemma \ref{BigLemma} (ii) we have 
    $$-\underline{I}(x+y)=\overline{I}(-x-y)\leq \overline{I}(-x)+\overline{I}(-y)
        =-\underline{I}(x)-\underline{I}(y). $$
    
    Thus $\underline{I}(x) + \underline{I}(y) \leq \underline{I}(x+y)$ and we have
    $$\int x + \int y \leq \underline{I}(x+y) \leq \overline{I}(x+y) \leq \int x + \int y.$$ 
    Thus $x+y \in \mathcal{L}$ and $\int (x+y)= \int x + \int y$.\\

    The fact that $x \in \mathcal{L}$ implies $|x| \in \mathcal{L}$ follows
    readily from Lemma \ref{BigLemma} (vi). Since if $x \in \mathcal{L}$, then we have 
    \[
        0 \leq \overline{I}(|x|)- \underline{I}(|x|) \leq \overline{I}(x)- \underline{I}(x)=0,
    \]
    hence $\overline{I}(|x|)= \underline{I}(|x|)$ and we conclude $|x| \in \mathcal{L}$.
\end{proof}

\begin{definition}\label{def:D-Integral}
    We call $\int$ on the class $\mathcal{L}$ the Daniell Integral induced by $I$ on $T_0$. 
\end{definition} 

Proper care needs to be taken in order to ensure 
that the Daniell Integral is in fact an $I$-integral (we treat this
in Remark \ref{remark:extend_d_int} below). 
In order to do so, we first prove the following theorem:

\begin{theorem}[Monotone Convergence Theorem for the Daniell Integral] \label{thm:Monotone}
    Let $(x_n)$ be a sequence in $\mathcal{L}$ such that $(x_n) \nearrow x$.  If $\lim_{n \rightarrow \infty} \int x_n$
    is finite, then $x \in \mathcal{L}$ and $\int x= \lim_{n \rightarrow \infty} \int x_n$.
\end{theorem}

\begin{proof}
    Let $(x_n)$ be a sequence in $\mathcal{L}$ such that $(x_n) \nearrow x$ and assume 
    $\lim_{n \rightarrow \infty} \int x_n$ is finite.
    We have $x_1 \leq x_2 \leq x_3 \leq \cdots \leq x$. Then for each $n \in \mathbb{N}$ we have $-x \leq -x_n$
    and so $\overline{I}(-x) \leq \int -x_n$ which is the same as 
    $\int x_n \leq \underline{I}(x)$.
    Thus $\lim_{n \rightarrow \infty} \int x_n \leq \underline{I}(x)$. \\

    To prove the reverse inequality, let $ \varepsilon>0$ be arbitrary. 
    Then using the definition of infimum on Definition \ref{def:upper_lower} 
     there exists a sequence $(\varphi_n)$ in $T_1$
    such that, 
       $$ x_1 \leq \varphi_1, \quad
        0 \leq x_2 -x_1 \leq \varphi_2, \quad
        \cdots, \quad
        0\leq x_n-x_{n-1}\leq \varphi_n$$
    with the property that $I_1(\varphi_1)\leq \int x_1 +\frac{\varepsilon}{2}$
    and $I_1(\varphi_n)\leq \int (x_n - x_{n-1})+ \frac{\varepsilon}{2^n}$ for each $n \in \mathbb{N}, n>1$.  \\

    Define, for $n\in \mathbb{N}$, the function $\rho_n := \varphi_1 + \varphi_2 + \cdots + \varphi_n$.
    Note, since $0 \leq \varphi_n$ for $n \geq 2$, that $(\rho_n)$ is an increasing sequence in $T_1$.  
    Also note that 
    $$x_n =x_1 +(x_2-x_1)+ \cdots + (x_n -x_{n-1})\leq \varphi_1 + \varphi_2 + \cdots + \varphi_n \leq\rho_n.$$
    Let $\psi \in T_1$
    such that $x \leq \psi$, and for $n \in \mathbb{N}$, set $\psi_n:=\psi \wedge \rho_n $. 
    Then $(\psi_n)$ is an increasing sequence in $T_1$
    and $x_n\leq \psi_n \leq \psi $. 
    Thus the pointwise limit $\lim_{n \rightarrow \infty} \psi_n$ exists and is in 
    $T_1$ by Lemma \ref{lemma:incr_closed}. 
    We also have
    $$x=\lim_{n \rightarrow \infty} x_n\leq \lim_{n \rightarrow \infty} \psi_n,$$
    and thus, again using Lemma \ref{lemma:incr_closed}, we conclude
    $$\overline{I}(x) \leq I_1 \left(\lim_{n \rightarrow \infty} \psi_n \right) =\lim_{n \rightarrow \infty} I_1(\psi_n).$$\\

    Note that, for each $n \in \mathbb{N}$, we have that 
    $$I_1(\rho_n) =I_1(\varphi_1)+\cdots + I_1(\varphi_n)\leq 
    \left( \int x_1 +\frac{\varepsilon}{2} \right) + \left( \int (x_2 -x_1) + \frac{\varepsilon}{4}\right)
    +\cdots + \left(\int (x_n -x_{n-1}) + \frac{\varepsilon}{2^n}\right) $$ and consequently,
    $$I_1(\psi_n)\leq I_1(\rho_n)\leq \int x_n +\left(\frac{\varepsilon}{2}+\frac{\varepsilon}{4}+
     \cdots+\frac{\varepsilon}{2^n}\right) \leq \int x_n + \varepsilon \leq \lim_{n \rightarrow \infty} \int x_n + \varepsilon.$$ \\
 
    Therefore  $$\overline{I}(x) \leq \lim_{n \rightarrow \infty} I_1(\psi_n) 
    \leq \lim_{n \rightarrow \infty} \int x_n + \varepsilon,$$ 
    and this holds for all  $\varepsilon >0$. 
    Thus $\overline{I}(x) \leq \lim_{n \rightarrow \infty} \int x_n $ 
    and we have $$\lim_{n \rightarrow \infty} \int x_n \leq \underline{I}(x) \leq 
    \overline{I}(x) \leq \lim_{n \rightarrow \infty} \int x_n. $$ 
    Thus $x \in \mathcal{L}$ and $\lim_{n \rightarrow \infty} \int x_n=\int x$ as desired.
\end{proof}
\begin{remark} \label{remark:extend_d_int}
    Theorem \ref{thm:Monotone} also holds for decreasing sequences by multiplying the sequence by $-1$.  
    The fact that $\int$ is linear shows (D1) from Definition \ref{def:I_Integral}.
Property (D3) of Definition \ref{def:I_Integral} for $\int$ follows from Lemma \ref{BigLemma} (iii), 
for if $x \in \mathcal{L}$ and $0 \leq x$, then $$0=\int 0 = \overline{I}(0) \leq \overline{I}(x)
= \int x.$$ 
If $(x_n)$ is a sequence in $ \mathcal{L}$ such that $(x_n) \searrow 0$, 
then $(-x_n) \nearrow 0$ and by Theorem \ref{thm:Monotone} $(- \int x_n) = (\int -x_n) \nearrow 0$.
Thus $ (\int x_n )\searrow 0$, proving (D2) of Definition \ref{def:I_Integral}. We conclude $ \int$ is an $I$-integral. 
\end{remark}

We now state and prove analogues of standard results from measure theory viz. Fatou's Lemma and 
Lebesgue's Dominated Convergence Theorem. For this purpose, we need the following lemma.
\begin{lemma} \label{lemma:inf_int}
    If $(x_n)$ is a sequence of nonnegative functions in $\mathcal{L}$, then $\inf_n x_n \in \mathcal{L}$.
\end{lemma}
\begin{proof}
    Let $(x_n)$ be a sequence of nonnegative functions in $\mathcal{L}$.
    For $n \in \mathbb{N}$ define  
     $$\varphi_n := x_1 \wedge x_2 \wedge \cdots \wedge x_n.$$
    Then by Theorem \ref{thm:lattice}, $\varphi_n \in \mathcal{L}$ for all $n \in \mathbb{N}$. Note also that 
    $(\varphi_n)$ is decreasing and nonnegative and thus $(-\varphi_n)$ is an increasing sequence bounded above by 0.
    Therefore $$\lim_{n\rightarrow \infty} \int -\varphi_n \leq 0 < \infty,$$
    and by the Theorem \ref{thm:Monotone} we have  $$\inf_n x_n = \lim_{n\rightarrow \infty}\varphi_n \in \mathcal{L}.$$
\end{proof}

\begin{theorem}[Fatou's Lemma for the Daniell Integral]\label{thm:Fatou}
    Let $(x_n)$ be a sequence of non-negative functions in $\mathcal{L}$. If $\liminf \int x_n < \infty$, then $\liminf x_n$ is
    in $\mathcal{L}$ and 
    \[  
        \int \liminf x_n \leq \liminf \int x_n .
    \]
\end{theorem}
\begin{proof}
    Let $\psi_n := \inf_{k\geq n} x_k$. From Lemma \ref{lemma:inf_int}, $\psi_n \in \mathcal{L}$ for $n \in \mathbb{N}$.
    Furthermore, $(\psi_n)$ is an increasing sequence and $\psi_n \leq x_n$ for $n \in \mathbb{N}$. 
    Thus $\int \psi_n \leq \int x_n$ for all $n \in \mathbb{N}$, and so 
    \[
        \lim_{n \rightarrow \infty} \int \psi_n = \liminf \int \psi_n \leq \liminf \int x_n < \infty. 
    \]
    Thus by the Theorem \ref{thm:Monotone} we have 
    $\lim_{n \rightarrow \infty} \psi_n \in \mathcal{L}$ and
    \[
    \int \lim_{n \rightarrow \infty} \psi_n=\lim_{n \rightarrow \infty} \int \psi_n \leq \liminf \int x_n.
    \] 
    But $\lim_{n \rightarrow \infty} \psi_n =\liminf x_n$ and we conclude $\int \liminf x_n \leq \liminf \int x_n$.
\end{proof}

\begin{theorem}[Lebesgue's Dominated Convergence Theorem for the Daniell Integral]\label{thm:dom}
    Let $(x_n)$ be a sequence in $\mathcal{L}$ such that $\left| x_n \right| \leq z$ for some $z \in \mathcal{L}$. 
    If $\lim_{n \rightarrow \infty} x_n =x$ pointwise, then $x \in \mathcal{L}$ 
    and $\lim_{n \rightarrow \infty} \int x_n=\int x$. 
\end{theorem}

\begin{proof}
    The proof is the same as that of the standard theorem in measure theory. We provide it for completeness.
    Note $0 \leq x_n + z \leq 2z$ for all $n \in \mathbb{N}$. Now $2z \in \mathcal{L}$, 
    hence $\liminf \int (x_n +z) \leq \int 2z <\infty$ for all $n \in \mathbb{N}$. 
    Thus by Theorem \ref{thm:Fatou}, we have 
    \begin{align}
        \int x+ \int z=\int (x+z) =\int \liminf(x_n+z)\leq \liminf \int (x_n+z) 
            = \liminf \int x_n +\int z. \label{eqn:leb_1}
    \end{align}
        
    Therefore, after subtracting $\int z$ from both sides, $\int x \leq \liminf \int x_n$. 
    Using the same argument on $-x_n +z$, we get
    \begin{align}
        -\int x+ \int z=\int (-x+z) &=\int \liminf(-x_n+z) \nonumber \\
        &\leq \liminf \int (-x_n+z) = -\limsup \int x_n +\int z. \label{eqn:leb_2}
    \end{align}
    Therefore, after subtracting $\int z$ from both sides, $-\int x \leq -\limsup \int x_n$ which is the same as $\limsup \int x_n \leq \int x$.
    
    Finally, combining \eqref{eqn:leb_1} and \eqref{eqn:leb_2} we have 
    \[
        \int x \leq \liminf \int x_n \leq \limsup \int x_n \leq \int x.
    \] 
    This shows that these inequalities are, in fact, equalities and we conclude that
    $$\lim_{n \rightarrow \infty} \int x_n =\int x.$$
\end{proof}

Another result of interest gives necessary and sufficient conditions for a function $x \in  {\overline{\mathbb{R}}}^X $
to be integrable. 

\begin{theorem}\label{thm:approximation}
    For any $x \in  {\overline{\mathbb{R}}}^X $ we have $x \in \mathcal{L}$ if and only if for every
    $\varepsilon>0$ there exists an $x_\varepsilon \in T_0$ such that $\overline{I}(|x-x_\varepsilon|)<\varepsilon$. 
\end{theorem}
\begin{proof}
    Let  $x \in  {\overline{\mathbb{R}}}^X $.
    First assume for every $\varepsilon>0$ there exists a $x_\varepsilon \in T_0$ such that $\overline{I}(|x-x_\varepsilon|)<\varepsilon$. 
    Let $\varepsilon>0$ and $x_\varepsilon \in T_0$ as in the hypothesis. Then 
    \[
        x=x_\varepsilon+x-x_\varepsilon \leq x_\varepsilon+|x-x_\varepsilon|,
    \]
    and so, using Lemma \ref{BigLemma}, we obtain
    \begin{align}
        \overline{I}(x)\leq I(x_\varepsilon)+\overline{I}(|x-x_\varepsilon|)< I(x_\varepsilon)+\varepsilon. \label{eqn:approx_1}
    \end{align}
    Similarly, we have
    \[
        -x= -x_\varepsilon -x + x_\varepsilon\leq -x_\varepsilon +|x-x_\varepsilon|,  
    \]
    and 
    \begin{align}
        \overline{I}(-x)\leq I(-x_\varepsilon)+\overline{I}(|x-x_\varepsilon|)< I(-x_\varepsilon)+\varepsilon.  \label{eqn:approx_2}
    \end{align}
    Adding together \eqref{eqn:approx_1} and \eqref{eqn:approx_2}, we obtain 
    \[
        0\leq \overline{I}(x)-\underline{I}(x)=\overline{I}(x)+\overline{I}(-x)\leq I(x_\varepsilon)+I(-x_\varepsilon)+2\varepsilon=2\varepsilon.  
    \]
    Since this holds for all $\varepsilon>0$ we have that $x$ is integrable. 

    Conversely, assume that $x\in \mathcal{L}$. Then for every $\varepsilon>0$ there exists $\varphi\in T_1$ with $x \leq \varphi$ such that 
    \[
        \overline{I}(x) \leq I_1(\varphi)<\overline{I}(x)+\frac{\varepsilon}{2}.  
    \]
    Thus, using the fact that $x$ is integrable and that 
    $x\leq \varphi$, we have from Lemma \ref{BigLemma}
    \begin{align}
        \overline{I}(|\varphi-x|)=\overline{I}(\varphi-x)
        \leq I_1(\varphi)+\overline{I}(-x)=I_1(\varphi)-\underline{I}(x)
    \end{align}
    Since $x \in \mathcal{L}$ we know that $\overline{I}(x)= \underline{I}(x)$ and so 
    \begin{align}
        \overline{I}(|\varphi-x|) \leq I_1(\varphi)-\underline{I}(x)=I_1(\varphi)-\overline{I}(x)  
        <\frac{\varepsilon}{2}. \label{eqn:approx_3}
    \end{align}
    Since $\varphi\in T_1$ there exists an $x_\varepsilon \in T_0$ such that $x_\varepsilon\leq \varphi$ and 
    \[
        I_1(\varphi)-\frac{\varepsilon}{2}<I(x_\varepsilon)
    \]
    Therefore 
    \[
        I_1(|\varphi-x_\varepsilon|)= I_1(\varphi- x_\varepsilon)=I_1(\varphi)-I_1(x_\varepsilon)=I_1(\varphi-x_\varepsilon)=I_1(|\varphi-x_\varepsilon|)<\frac{\varepsilon}{2}. 
    \]
    Using Lemma \ref{BigLemma} together with \eqref{eqn:approx_3} we have that 
    \[
        \overline{I}(|x-x_\varepsilon|)\leq \overline{I}(|x-\varphi|+|\varphi-x_\varepsilon|) \leq \overline{I}(|x-\varphi|)+I_1(|\varphi-x_\varepsilon|)<\varepsilon.    
    \]
    This proves the result.    
\end{proof}

To end this section, we describe what the functions in $\mathcal{L}$ `look like'. This can be seen as an 
analogue to the regularity theorem for the Lebesgue measure (see Appendix, Theorem~\ref{thm:regularity}).

\begin{definition}
    We say $x \in \mathcal{L}$ is a null function if $\int |x|=0$. 
\end{definition}

\begin{proposition}\label{prop:null}
    If $x \in \mathcal{L}$ is a null function and $y\in {\overline{\mathbb{R}}}^X$ 
    is such that $|y|\leq |x|$, then $y$ is also a null function.
\end{proposition}

\begin{proof}
    Let $x \in \mathcal{L}$ be a null function and $y\in {\overline{\mathbb{R}}}^X$ 
    is such that $|y|\leq |x|$.
   Note that $$0\leq \underline{I}(y\lor 0)\leq \overline{I}(y\lor 0) \leq\overline{I}(|y|) \leq \overline{I}(|x|)=\int |x|=0 .$$ 
   Thus $y\lor 0 \in \mathcal{L}$. Similarly, $(-y) \lor 0 \in \mathcal{L}$ and $y=[y\lor 0] -[(-y)\lor 0] \in \mathcal{L}$ with 
    $\int y =\int (y\lor 0) - \int ((-y)\lor 0) = 0$. Therefore $y$ is also a null function.\\
\end{proof}

\begin{definition} \label{def:T2}
    Define $T_2 \subset \mathcal{L}$ as follows:
    \[
        T_2:= \left\lbrace x \in \mathcal{L}: \text{ there exists } (x_n) \subset T_1 \text{ such that }
        (x_n) \searrow x \text{ and } -\infty < \int x < \infty \right\rbrace  . 
    \]
\end{definition}

Up until this point we have defined multiple sets of functions: $T_0$, $T_1$, $T_2$, and $\mathcal{L}$.
The relationship between these sets are 
\[
    T_0 \subset T_1 \cap \mathcal{L} \subset T_2 \subset \mathcal{L}.
\]

\begin{proposition} \label{prop:T2}
    If $x \in T_2$, then there exists a 
    sequence $(\varphi_n) \subset T_1$ with $I_1(\varphi_n)<\infty$ for $n \in \mathbb{N}$
    such that $(\varphi_n) \searrow x$. 
\end{proposition}
\begin{proof}
    Assume $x \in T_2$. Then by Definition \ref{def:T2} 
    there exists a sequence $(\psi_n) \subset T_1$ such that 
    $(\psi_n) \searrow x$. Also, since $x \in \mathcal{L}$ we have 
    $\int x = \overline{I}(x)<\infty$ and there 
    exists a $\varphi \in T_1$ such that $x \leq \varphi$ and
    $$ \overline{I}(x) \leq I_1(\varphi) \leq \overline{I}(x)+1<\infty.$$
    We thus have for all $n\in \mathbb{N}$ that
    \[
        x \leq \varphi \land \psi_n \leq \psi_n,
    \]
    and since $(\psi_n)\searrow x$ we conclude                        
    $(\varphi \land \psi_n) \searrow x$. Furthermore, we have 
    \[
        I_1(\varphi \land \psi_n ) \leq I_1(\varphi) \leq \overline{I}(x)+1 < \infty  
    \]
    for 
    each $n \in \mathbb{N}$. Thus $(\varphi \land \psi_n)$ is a sequence in $T_1$ with 
    $I_1(\varphi \land \psi_n ) <\infty$ for $n \in \mathbb{N}$
    such that $(\varphi \land \psi_n) \searrow x$.
\end{proof}
\begin{theorem}[Characterisation Theorem] \label{thm:char}
    For any function $x\in {\overline{\mathbb{R}}}^X$, we have that 
    $x\in \mathcal{L}$ if and only if $x=y-z$ where $y\in T_2$ and $z$ is a nonnegative null function.
\end{theorem}
\begin{proof}
    Assume firstly that $x=y-z$ with $y \in T_2$ and $z\geq 0$ a null function. 
    Then both $z,y\in \mathcal{L}$ and so $x\in \mathcal{L}$ by Theorem \ref{thm:lattice}.\\

    Now assume that $x \in \mathcal{L}$. Then $\int x = \overline{I}(x) < \infty$ and for every $n \in \mathbb{N}$ there exists 
    $\varphi_n \in T_1$ such that $x \leq \varphi_n$ and 
    $$\overline{I}(x)\leq I_1(\varphi_n) \leq \overline{I}(x) + \frac{1}{n}.$$ 

    For $n \in \mathbb{N}$, let $$y_n:= \varphi_1 \land \varphi_2 \land \cdots \land \varphi_n.$$
    Note that for all $n \in \mathbb{N}$, $y_n \leq \varphi_n$ and that $y_n \in T_1$.
    Also $(y_n)$ is a decreasing sequence with $x \leq y_n$
    and 
    $$-\infty <\overline{I}(x)\leq I_1(y_n) \leq \overline{I}(x) + \frac{1}{n}<\infty$$ 
    for $n \in \mathbb{N}$.
    Since, for all $n\in \mathbb{N}$, $I_1(y_n) < \infty$ we have that $y_n \in \mathcal{L}$
    and $I_1(y_n)=\int y_n$. 
    Furthermore, $\lim_{n \rightarrow \infty} \int y_n=\lim_{n \rightarrow \infty} I_1 (y_n)$ is finite. 
    Let $y:=\lim_{n \rightarrow \infty} y_n$ pointwise. Then by Theorem 
    \ref{thm:Monotone} we have $y \in \mathcal{L}$, hence also in $T_2$, 
    and $\lim_{n \rightarrow \infty} \int y_n = \int y$.
    
    We thus have 
    $$\int x = \overline{I}(x) = \lim_{n \rightarrow \infty} I_1(y_n) = \lim_{n \rightarrow \infty} \int y_n = \int y.$$ 
    We also know $x \leq y$ so if we set $z:=y-x$ then it follows that $0 \leq z$. 
    Furthermore, $z \in \mathcal{L}$ and 
    $$\int |z|= \int z = \int y -\int x = 0.$$ 
    Hence $z$ is a null function and $x=y-z$ where $y \in T_2$ and 
    $z$ is a null function. This completes the proof.
\end{proof}


\section{Measure from the Integral} \label{section:measure}

We described a procedure that extends a linear functional to one that 
allows the classical measure-theoretic limit theorems in \ref{section:extension}. 
We now show that this extension is, in fact, the same as the measure-theoretic construction. 
The material of this section 
closely follows that of \cite[Chapter~16]{royden1988real}. \\

Throughout this section, we let $X$ be a set, $T_0 \subset {\overline{\mathbb{R}}}^X$
 denote an extended vector lattice, and 
$I: T_0 \rightarrow \mathbb{R}$ an $I$-integral. 
Let $\mathcal{L}$ be the extension of $T_0$ and $\int$ 
be the extension of $I$ as in Definition \ref{def:upper_lower}.
If $A\subset X$, let $\chi_A : X \rightarrow \mathbb{R}$ denote the characteristic function. 
Explicitly, for $t \in X$ define $\chi_A$ by
\[
    \chi_A(t):= \begin{cases}
        1, &\text{ if } t \in A\\
        0, &\text{ if } t \notin A.
    \end{cases}  
\]
To begin with, we define the Daniell analogue of measurable functions. 

\begin{definition}
    We say a nonnegative function $x \in {\overline{\mathbb{R}}}^X$ is {\it Daniell measurable} if for all $\varphi \in \mathcal{L}$ 
    we have $\varphi \land x \in \mathcal{L}$. We say a set $A \subset X$ is {\it Daniell measurable} if its characteristic function
    $\chi_A$ is Daniell measurable.
\end{definition}

Using this definition we have: 
\begin{lemma} \label{lemma:cap_lemma}
    Let $x,y \in {\overline{\mathbb{R}}}^X$. Then
    \begin{enumerate}[(i)]
        \item if $x$ and $y$ are Daniell measurable functions, then so are $x \lor y$ and $x \land y$; and
        \item if $(x_n)$ is a sequence of nonnegative Daniell measurable 
        functions converging pointwise to a function $x$, then $x$ is also Daniell measurable.
    \end{enumerate}    
     
\end{lemma}
\begin{proof}
    To prove (i), let $x$ and $y$ be nonnegative and measurable and let $\varphi \in \mathcal{L}$. 
    Then because $\mathcal{L}$ is closed under $\land$ with both 
    $x \land \varphi$ and $y \land \varphi$ in $\mathcal{L}$, we have
    \[
        (x\land y)\land \varphi  =(x \land \varphi)\land (y\land \varphi) \in \mathcal{L}.
    \]
    Similarly, because $\mathcal{L}$ is closed under $\lor$ with both
    $y \land \varphi$ and $x \land \varphi$ in $\mathcal{L}$, we have
    \[
        (x \lor y)\land \varphi= (x \land \varphi) \lor (y \land \varphi) \in \mathcal{L}.
    \]
    
    To prove (ii), let $(x_n)$ be a sequence of nonnegative Daniell measurable functions. 
    If $\varphi \in \mathcal{L}$, 
    then since $0 \leq x_n$ for $n \in \mathbb{N}$, we have $| x_n \land \varphi| \leq |\varphi|$. 
    We also know that $|\varphi|\in \mathcal{L}$ by Theorem \ref{thm:lattice}. 
    Hence by Theorem \ref{thm:dom}, we have that 
    $$x \land \varphi= \lim_{n \rightarrow \infty} x_n \land \varphi \in \mathcal{L}.$$ 
    As this holds for all $\varphi \in \mathcal{L}$ we conclude that $x$ is measurable. \qedhere

\end{proof}

\begin{lemma}\label{lemma:measurable_lemma}
    Let $x \in {\overline{\mathbb{R}}}^X$ be nonnegative.
    We have $\varphi \land x \in \mathcal{L}$ 
    for each $\varphi \in T_0$ if and only if $x$ is a 
    Daniell measurable function.
\end{lemma}
\begin{proof}
    Assume $\varphi \land x \in \mathcal{L}$ for each $\varphi \in T_0$.
    We claim that $\varphi \land x \in \mathcal{L}$ for each $\varphi \in T_1$ with $I_1(\varphi)<\infty$.
    If $\varphi \in T_1$ with $(\varphi_n) \subset T_0$ such that $(\varphi_n)\nearrow \varphi$
    with $I_1(\varphi)<\infty$, then $\varphi_n \land x \in \mathcal{L}$ for 
    each $n \in \mathbb{N}$ by assumption. Consequently, 
    since 
    $$\lim_{n \rightarrow \infty} \int \varphi_n \land x \leq \lim_{n \rightarrow \infty} \int \varphi_n \leq \int \varphi <\infty,$$
    we have that $\varphi \land x = \lim_{n \rightarrow \infty} \varphi_n \land x \in \mathcal{L}$
     by Theorem \ref{thm:Monotone}. Thus our hypothesis implies 
    $\varphi \land x \in \mathcal{L}$ for each integrable $\varphi \in T_1$.

    We claim that $\varphi \land x \in \mathcal{L}$ for each $\varphi \in T_2$.
    Let $\varphi \in T_2$ with $(\varphi_n)\subset T_1$ such that $(\varphi_n) \searrow \varphi$.
    By Proposition \ref{prop:T2} we may choose the sequence $(\varphi_n)$ so that
    $I_1(\varphi_n)<\infty$, and thus by our first claim, that
     $\varphi_n \land x \in \mathcal{L}$ for each $n \in \mathbb{N}$.
    Recall that $T_2 \subset \mathcal{L}$ and so $\varphi \in \mathcal{L}$.
    Since $x$ is nonnegative, we have 
    \begin{align}
        \varphi \land 0 \leq \varphi \land x \leq \varphi_n \land x \leq \varphi_1 \land x, \label{eqn:bounded}
    \end{align}
    for each $n \in \mathbb{N}$. We know by Proposition \ref{thm:lattice} 
    that $\mathcal{L}$ is an extended vector lattice,
    and thus that $\varphi \land 0 \in \mathcal{L}$. We also have $\varphi_1 \land x\in \mathcal{L}$ by our previous claim,
    since $\varphi_1 \in T_1$.
    Hence from \eqref{eqn:bounded}, the sequence $(\varphi_n \land x)$ is bounded by functions in $\mathcal{L}$, 
    and using Theorem \ref{thm:dom} we conclude that 
    $\varphi \land x = \lim_{n \rightarrow \infty} \varphi_n \land x \in \mathcal{L}$.
    Thus $\varphi \land x \in \mathcal{L}$ for all $\varphi \in T_2$.
    
    If $\varphi \in \mathcal{L}$, then by Theorem \ref{thm:char} we have that 
    $\varphi= \psi - z$ for some $\psi \in T_2$ and nonnegative null function $z \in \mathcal{L}$. 
    We then have $\psi \land x \in \mathcal{L}$ and
    \begin{align*}
        0 &\leq \psi \land x - \varphi \land x \\
        &= \frac{1}{2}(\psi + x - |\psi -x|) -\frac{1}{2}(\varphi+x - |\varphi - x|) \\
        &= \frac{1}{2}(\psi -\varphi + |\varphi - x| - |\psi - x|)\\
        &\leq \frac{1}{2}(\psi -\varphi + |\psi - \varphi|)\\
        &=\psi - \varphi = z.
    \end{align*}
    Since $z$ is a null function, it follows from Proposition \ref{prop:null} that 
    $\psi \land x - \varphi \land x$ is also null.
    Hence $\varphi \land x$ differs from $\psi \land x$ by a null function and 
    we conclude, using Theorem \ref{thm:char}, that $\varphi \land x \in \mathcal{L}$. 
    This holds for arbitrary $\varphi \in \mathcal{L}$, therefore we conclude that $x$ 
    is a Daniell measurable function.

    The converse follows from the fact that $T_0 \subset \mathcal{L}$. 
    Hence if $\varphi \land x \in \mathcal{L}$ for each $\varphi \in \mathcal{L}$, then 
    $\varphi \land x \in \mathcal{L}$ for each $\varphi \in T_0$.
\end{proof}

We refer the readers to the Appendix for the definitions of $\sigma$-ring (Definition \ref{def:sigma_ring}) 
and $\sigma$-algebra (Definition \ref{def:sigma_alg}).
\begin{theorem} \label{thm:algebra}
    Let $\mathcal{A}$ be the class of all Daniell  measurable sets. 
    Then $\mathcal{A}$ is a $\sigma$-ring. Furthermore if $\chi_X$ is a Daniell measurable function, then 
    $\mathcal{A}$ is a $\sigma$-algebra. 
\end{theorem}
\begin{proof}
    To show that $\mathcal{A}$ is a $\sigma$-ring, it suffices to show that, $\mathcal{A}$
    is closed under 
    \begin{enumerate}[(i)]
        \item relative difference, and
        \item countable union.
    \end{enumerate}
    
    To prove (i), let $A,B \in \mathcal{A}$.  
    \medskip
    \begin{claim}
    For all $\varphi \in \mathcal{L}$,
    \begin{align}
        \varphi\land\chi_{A \setminus B}
        =\varphi \land \chi_A -\varphi\land (\chi_A \land \chi_B)+\varphi\land 0 \label{eqn:and}
    \end{align}    
    \end{claim}
    \medskip
    \begin{claimproof}
        Let $t \in X$. If $\varphi(t) \leq 0$, since characteristic functions are all 
        nonnegative, then we have
        $$(\varphi\land \chi_{A \setminus B} )(t) = \varphi(t)
            =(\varphi \land \chi_A)(t) -(\varphi\land (\chi_A \land \chi_B))(t)+(\varphi\land 0)(t).$$
        If $\varphi(t) \geq 0$ and $t \in A\setminus B$, then 
            $$(\varphi\land \chi_{A \setminus B} )(t) = \min \lbrace \varphi(t), 1 \rbrace,$$
        and 
            $$(\varphi \land \chi_A)(t) -(\varphi\land (\chi_A \land \chi_B))(t)+(\varphi\land 0)(t)
                =\min \lbrace \varphi(t), 1 \rbrace - 0 + 0=\min \lbrace \varphi(t), 1 \rbrace.$$
        If $\varphi(t) \geq 0$ and $t \in A\cap B$, then 
            $$(\varphi\land \chi_{A \setminus B} )(t) = 0,$$
        and 
            $$(\varphi \land \chi_A)(t) -(\varphi\land (\chi_A \land \chi_B))(t)+(\varphi\land 0)(t)
                =\min \lbrace \varphi(t), 1 \rbrace - \min \lbrace \varphi(t), 1 \rbrace + 0=0.$$
        Finally, if $\varphi(t) \geq 0$ and $t \notin A$, then
            $$(\varphi\land \chi_{A \setminus B} )(t) = 0,$$
        and 
            $$(\varphi \land \chi_A)(t) -(\varphi\land (\chi_A \land \chi_B))(t)+(\varphi\land 0)(t)
                =0.$$
        Thus, for all $t \in X$, we have that \eqref{eqn:and} holds. This proves the claim. 
    \end{claimproof}
    \medskip

    Thus, if $\varphi \in \mathcal{L}$ then each term on the right side of \eqref{eqn:and} is in $\mathcal{L}$. 
    We conclude that $\varphi\land\chi_{A \setminus B}$ is a Daniell measurable function 
    and $A \setminus B$ is a Daniell measurable set.\\

    To show (ii), let $(A_n)$ be a sequence of sets in $\mathcal{A}$. Then for each $n \in \mathbb{N}$ we have that 
    $$\chi_{\bigcup_{i=1}^n A_i}=\chi_{A_1}\lor \chi_{A_2}\lor \cdots \lor \chi_{A_n}$$
    and this is Daniell measurable by Lemma \ref{lemma:cap_lemma}.
    Furthermore $ (\chi_{\bigcup_{i=1}^n A_i}) \nearrow \chi_{\bigcup_{i=1}^\infty A_i}$ and 
    again using 
    Lemma \ref{lemma:cap_lemma} we conclude that $\chi_{\bigcup_{i=1}^\infty A_i}$ is a
     Daniell measurable function.
     Thus $\bigcup_{i=1}^\infty A_i$ is a Daniell measurable set.

    These properties show that $\mathcal{A}$ is a $\sigma$-ring. 

     If $\chi_X$ is a Daniell measurable function, then
    $X$ is a Daniell measurable set. Hence $X \in \mathcal{A}$ and $\mathcal{A}$ is a 
    $\sigma$-algebra.
\end{proof}

\begin{theorem}\label{thm:measure}
    Assume that $\chi_X$ is Daniell measurable.
    Let $\mu : \mathcal{A} \rightarrow \mathbb{R}\cup \lbrace \infty\rbrace$ be defined 
    for $E \in \mathcal{A}$ by
    \[ 
    \mu(E):= 
        \begin{cases}
            \int \chi_E &\text{ if $\chi_E$ is integrable}\\
            \infty &\text{ otherwise. }\\
        \end{cases}  
    \]
    Then $\mu$ is a measure.
\end{theorem}
\begin{proof}
    We first show that $\mu$ is nonnegative and that $\mu (\emptyset)=0$.
    If $E \in \mathcal{A}$, then we have that $0 \leq \chi_E$. 
    If $\mu (E)$ is finite, then since $0 \leq \chi_E$ we have
    $$0 \leq \int \chi_E=\mu (E).$$
    Furthermore, $\chi_\emptyset=0$ and so
    Since $\int$ is linear by Proposition \ref{thm:lattice}, we have that 
    $$\mu (\emptyset)= \int \chi_\emptyset = \int 0=0.$$

    We now show that $\mu$ is countably additive. 
    First assume that $(E_n)$ is a sequence of disjoint sets 
    in $\mathcal{A}$ such that all of $\mu (E_n)$ are finite.
    Since these sets are disjoint we have 
    $\chi_{\bigcup_{n=1}^\infty E_n}= \sum_{n=1}^\infty \chi_{E_n}$.
    If $ \sum_{n=1}^\infty \int \chi_{E_n}<\infty$, then by Theorem \ref{thm:Monotone}, we have 
    $$\chi_{\bigcup_{n=1}^\infty E_n}=\sum_{n=1}^\infty \chi_{E_n} \in \mathcal{L}, $$
    and 
    $$\mu \left(\bigcup_{n=1}^\infty E_n \right)= \int \chi_{\bigcup_{n=1}^\infty E_n}
    =\int \sum_{n=1}^\infty \chi_{E_n} = 
    \sum_{n=1}^\infty \int \chi_{E_n}=\sum_{n=1}^\infty \mu (E_n).$$ 
    If $ \sum_{n=1}^\infty \int \chi_{E_n}=\infty$, 
    then $\mu \left(\bigcup_{n=1}^\infty E_n\right) = \infty$. For if this is not the case, 
    then $\chi_{\bigcup_{n=1}^\infty E_n}\in \mathcal{L}$,
    and for each $N \in \mathbb{N}$ we have 
    $$\sum_{n=1}^N \int \chi_{E_n} = \int \chi_{\bigcup_{n=1}^N E_n}
    \leq \int \chi_{\bigcup_{n=1}^\infty E_n} <\infty,$$ 
    contradicting the assumption that $ \sum_{n=1}^\infty \int \chi_{E_n}=\infty$. 
    We conclude that $\mu$ is a measure.
\end{proof}

We defined a measure on our set $X$ and a Daniell Integral $\int $ on 
the class $\mathcal{L}\subset {\overline{\mathbb{R}}}^X$ of $I$-integrable functions. 
The next theorem 
connects the two and finalises the connection between the Daniell Integral and a measure. 

\begin{theorem}[Daniell-Stone]\label{thm:DS}
    Let $T_0$ be an extended vector lattice of functions on a set $X$ with the
     property that $\chi_X \land \varphi \in T_0$ whenever $\varphi \in T_0$. Let
     $I$ be an $I$-integral on $T_0$. Then there is a $\sigma$-algebra 
    $\mathcal{A}$ on $X$ and a measure $\mu$ on $\mathcal{A}$ such that
    each function on $X$ is integrable with respect to $\mu$ if and only 
    if it is $I$-integrable. 
    Furthermore if $\varphi$ is both $I$- and $\mu$-integrable then 
    \[
        \int \varphi =\int \varphi \, d\mu.
    \]
\end{theorem}

\begin{proof}
    Our assumption that $\chi_X \land \varphi \in T_0$ for each 
    $\varphi \in T_0$ together with 
    Lemma \ref{lemma:measurable_lemma} allows us to conclude that 
    $\chi_X$ is a Daniell measurable function.
    Theorems \ref{thm:algebra} and \ref{thm:measure} guarantee the 
    existence of the $\sigma$-algebra 
    $\mathcal{A}$ and measure $\mu$ on $X$.\\

    \begin{claim}
        If $x\in \mathcal{L}$, then $x$ is $\mu$-measurable.\\
    \end{claim}
    \medskip
    \begin{claimproof} 
    Since every function in $\mathcal{L}$
    is the difference of two nonnegative functions, it is sufficient to 
    consider nonnegative functions in $\mathcal{L}$.
    Without loss of generality let $x\in \mathcal{L}$ with $x \geq 0$ 
    and let $$E_\alpha:= \lbrace t \in X: x(t) > \alpha \rbrace.$$ 
    If $\alpha \leq 0$ then $E_\alpha = X \in \mathcal{A}$. So assume $\alpha >0$. 
    Let $$y= \frac{1}{\alpha}x-\left(\frac{1}{\alpha}x\right)\land \chi_X.$$ 
    Since $\chi_X$ is Daniell measurable,
     $\left(\frac{1}{\alpha}x\right)\land \chi_X \in \mathcal{L}$
    and so $y$ is also in $\mathcal{L}$. 
    Let $t \in X$. If $t \in E_\alpha$ then $ \alpha<x(t)$. This implies that $\chi_X(t)<\frac{1}{\alpha}x(t)$ and 
    thus $0<y(t)$. If $t \notin E_\alpha$, then $x(t) \leq \alpha$ and $\frac{1}{\alpha}x(t)\leq \chi_X(t)$. 
    Hence $y(t)=0$. In other words $0<y(t)$ if and only if $t \in E_\alpha$.
    For $n \in \mathbb{N}$, define 
    \begin{align}
        \varphi_n:=\chi_X\land ny \label{eqn:phi_incr_claim}
    \end{align}
    (see Figure \ref{fig:incr} for an illustration). Then $(\varphi_n)$ is a sequence of Daniell measurable
    functions such that $(\varphi_n)\nearrow \chi_{E_\alpha}$. 
    Thus by Lemma \ref{lemma:cap_lemma} $\chi_{E_\alpha}$ is Daniell measurable
    and $E_\alpha \in \mathcal{A}$. This proves the claim.
    \end{claimproof}
     \medskip
        
    \begin{figure}[h]
        \includegraphics[width=\linewidth]{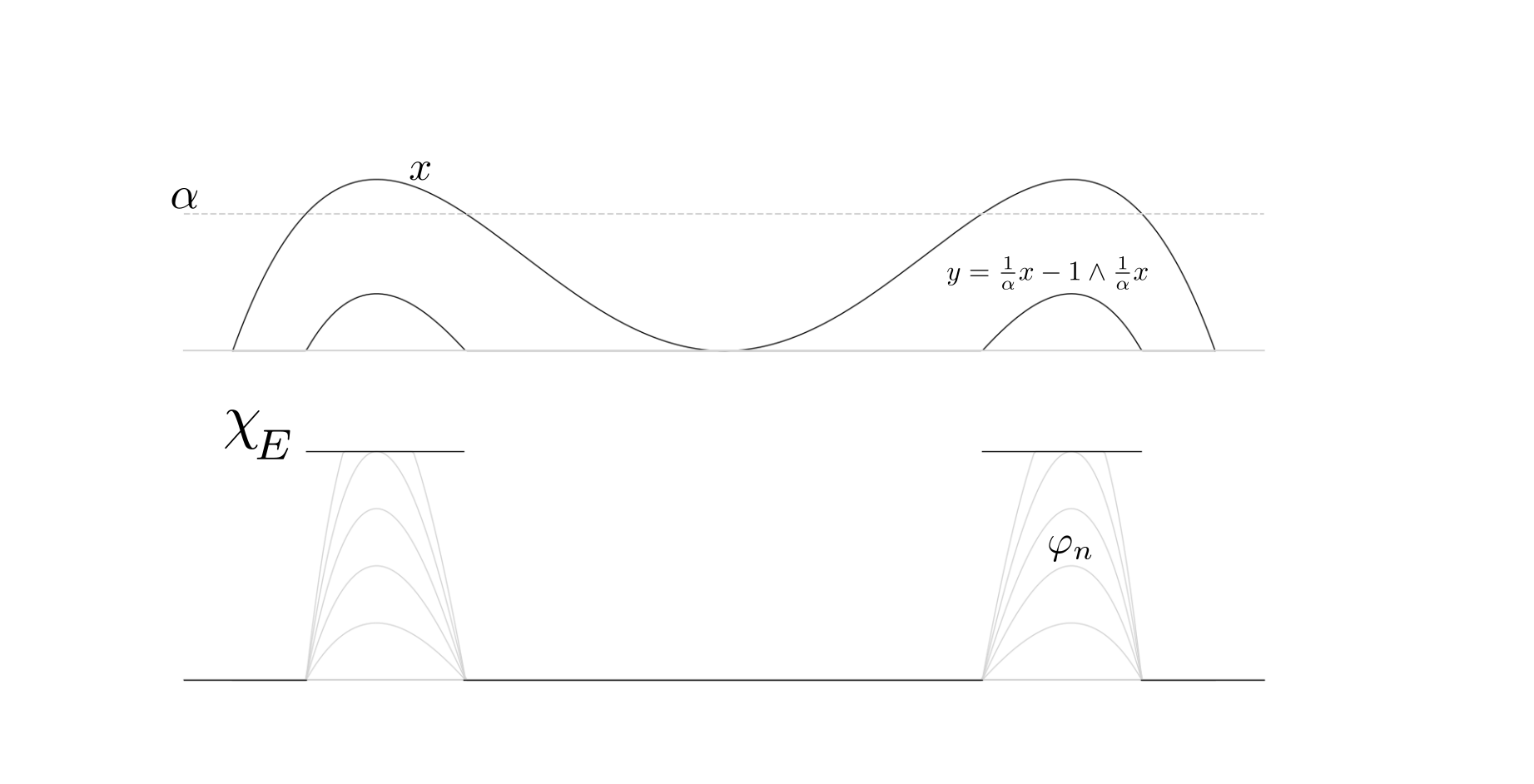}
        \caption{An illustration of $\varphi_n$ in \eqref{eqn:phi_incr_claim}.}\label{fig:incr}
    \end{figure}

    Let $x$ be nonnegative in $\mathcal{L}$. We aim to show that $x$ 
    is integrable with respect to $\mu$.
    For $k,n \in \mathbb{N}$, define 
    \begin{align}
        E_{k,n}:=\lbrace t \in X: x(t)> k2^{-n} \rbrace. \label{eqn:ekn}
    \end{align}
   From the proven claim above we know $E_{k,n}\in \mathcal{A}$ 
   and $\chi_{E_{k,n}}$ is Daniell measurable. 
    Furthermore, $$\chi_{E_{k,n}}=\chi_{E_{k,n}} \land (2^n k^{-1}x) \in \mathcal{L}.$$ 
    For $n \in \mathbb{N}$, define the sequence $(\varphi_n)\subset \mathcal{L}$ by
    \begin{align}
       \varphi_n := 2^{-n} \sum_{k=1}^{2^{2n}} \chi_{E_{k,n}}. \label{eqn:phin}
    \end{align}

    \begin{claim}
        The sequence $(\varphi_n)$ is increasing. 
    \end{claim}
    \medskip
    \begin{claimproof}
    Fix $n \in \mathbb{N}$.
    Note that 
    \begin{align}
        \varphi_{n+1}=  2^{-(n+1)} \sum_{k=1}^{2^{2(n+1)}} \chi_{E_{k,n+1}}
        = 2^{-n} \sum_{j=1}^{2^{2(n+1)-1}}\frac{1}{2}\left( \chi_{E_{2j,n+1}} 
            +\chi_{E_{2j-1,n+1}}\right). \label{eqn:inc_phi}
    \end{align}
    Also, if $t \in E_{j,n}$, then 
    $$x(t)>j2^{-n}=(2j)2^{-(n+1)},$$ 
    and so $t \in \chi_{E_{2j,n+1}}$.
    Similarly, if $t \in E_{j,n}$, then 
    $$x(t)>j2^{-n}>(2j-1)2^{-(n+1)},$$ 
    and so $t \in \chi_{E_{2j-1,n+1}}$. 
    Thus $E_{j,n}\subset E_{2j,n+1} $ and $ E_{j,n} \subset E_{2j-1,n+1}$, 
    and we have 
    $$\chi_{E_{j,n}} \leq \chi_{E_{2j,n+1}} 
        \text{ and } \chi_{E_{j,n}} \leq \chi_{E_{2j-1,n+1}}. $$
    Using these inequalities together with \eqref{eqn:inc_phi} we get
    \[
        2^{-n} \sum_{j=1}^{2^{2(n+1)-1}}\frac{1}{2}\left( \chi_{E_{2j,n+1}} +\chi_{E_{2j-1,n+1}}\right)
        \geq 2^{-n} \sum_{j=1}^{2^{2(n+1)-1}}\frac{1}{2}\left( \chi_{E_{j,n}} +\chi_{E_{j,n}}\right)
        \geq 2^{-n} \sum_{j=1}^{2^{2n}} \chi_{E_{j,n}} = \varphi_n.
    \]
    and we conclude $\varphi_{n+1} \geq \varphi_n$. Since this holds for all $n \in \mathbb{N}$, $(\varphi_n)$ is an increasing sequence.
    This proves the claim that $(\varphi_n)$ is increasing.
\end{claimproof}
\medskip

    \begin{claim}
        The sequence $(\varphi_n)$ has pointwise limit $x$.
    \end{claim}
    \medskip
    \begin{claimproof}
    Fix $t \in X$. There are three cases.

    {\bf Case 1:} If $x(t)=0$ then for each 
    $k,n \in \mathbb{N}$ we have $\chi_{E_{k,n}}(t)=0$. 
    Thus $\varphi_n(t)=0$ and $x(t)=\lim_{n\rightarrow \infty}\varphi_n(t)$. 
    
    {\bf Case 2:} If $0<x(t)$ and $x(t)$ is finite, then there is an $N \in \mathbb{N}$ such that 
    for $n>N$ we have $2^n > x(t)$. Fix $n >N$ and let
    \[
        K:=\max\lbrace k\in \mathbb{N}: x(t) > 2^{-n}k \rbrace.
    \]
    Note that if $k > K$, then $2^{-n}k\geq x(t)$ and so $t \notin E_{k,n}$.
    Hence $\chi_{E_{k,n}}=0$ for $k>K$.
    We thus have
     \[
        x(t) - \varphi_n(t)=x(t)- 2^{-n} \sum_{k=1}^{2^{2n}} \chi_{E_{k,n}}(t)
            =x(t)- 2^{-n} \sum_{k=1}^K \chi_{E_{k,n}}(t).
    \]
    However, we know that $t \in E_{k,n}$ when $k<K$. Therefore, 
    for all $k<K$ we have $\chi_{E_{k,n}}(t)=1$. We also know that 
    $x(t) < 2^{-n}(K+1)$. Hence
    \[
        0\leq x(t) - \varphi_n(t)=x(t)- 2^{-n} \sum_{k=1}^K \chi_{E_{k,n}}(t)= 
            x(t)-2^{-n}K<2^{-n}(K+1)-2^{-n}K=2^{-n}
    \]
    and we conclude $x(t)=\lim_{n\rightarrow \infty}\varphi_n(t)$. 
    
    {\bf Case 3:} If $x(t) = \infty$ then $\varphi_n(t) = 2^{-n} \sum_{k=1}^{2^{2n}} \chi_{E_{k,n}}(t) 
    = 2^{-n}2^{2n}=2^{n}$ and 
    $x(t)=\lim_{n\rightarrow \infty}\varphi_n(t)$. 
    Thus we have $x(t)=\lim_{n\rightarrow \infty}\varphi_n(t)$ for 
    each $t\in X$ and we conclude $(\varphi_n)\nearrow x$. 
    This proves the claim that $(\varphi_n)$ has pointwise limit $x$.
\end{claimproof} 
\medskip 

We have thus shown that $(\varphi_n) \nearrow x$.
Finally, we aim to show that 
$$ \int x = \int x \,d\mu.$$
    By Theorem \ref{thm:Monotone} we have 
    \begin{align}
        \int x = \lim_{n \rightarrow \infty} \int \varphi_n
            =\lim_{n \rightarrow \infty} \int 2^{-n} \sum_{k=1}^{2^{2n}} \chi_{E_{k,n}}
            = \lim_{n \rightarrow \infty} 2^{-n} \sum_{k=1}^{2^{2n}} \int \chi_{E_{k,n}}
            = \lim_{n \rightarrow \infty} 2^{-n} \sum_{k=1}^{2^{2n}} \mu(E_{k,n}). \label{eqn:backward_1}
    \end{align}
    The measure-theoretic definition of the integral for simple functions gives us 
    \begin{align}
        \lim_{n \rightarrow \infty} 2^{-n} \sum_{k=1}^{2^{2n}} \mu(E_{k,n})=
        \lim_{n \rightarrow \infty} \int \varphi_n \, d\mu. \label{eqn:backward_2}
    \end{align}
    Applying the classical measure-theoretic Monotone Convergence Theorem leaves 
    \[
        \lim_{n \rightarrow \infty} \int \varphi_n \,d\mu=\int x \, d\mu
    \]
    and we conclude that 
    \[
    \int x= \int x \, d\mu.
    \]

    Thus $x$ is integrable with respect to $\mu$ and the two 
    definitions of the integral coincide. \\

    We now show the converse of the theorem for nonnegative functions.
    Assume $x \in  {\overline{\mathbb{R}}}^X$ with $x \geq 0$ 
    and (measure-theoretically) integrable with respect to $\mu$. 
    We aim to show that $x$ is $I$-integrable, and that the two integrals coincide.
    As in \eqref{eqn:ekn} for $k,n \in \mathbb{N}$, define 
    \begin{align}
        E_{k,n}:=\lbrace t \in X: x(t)> k2^{-n} \rbrace, \label{eqn:ekn}
    \end{align}
    and as in \eqref{eqn:phin}, define the sequence $(\varphi_n)$ by
    \begin{align}
       \varphi_n := 2^{-n} \sum_{k=1}^{2^{2n}} \chi_{E_{k,n}}. \label{eqn:phin}
    \end{align}
    Note that, if $t \in E_{k,n}$, then we have $k 2^{-n}<x(t)$ and thus 
    $1<\frac{2^n}{k} x(t)$. Consequently, 
    $$\chi_{E_{k,n}} \leq \frac{2^n}{k} \chi_{E_{k,n}} x \leq \frac{2^n}{k} x.$$
    Since the (measure-theoretic) integral with respect to
     $\mu$ of $x$ is finite, we have for $n,k \in \mathbb{N}$
     \[
        \mu (E_{k,n})= \int \chi_{E_{k,n}} \, d\mu \leq \frac{2^n}{k}\int x \, d\mu< \infty.
     \]
    and $\chi_{E_{k,n}} \in \mathcal{L}$. 
    Thus for all $n \in \mathbb{N}$ we have that $\varphi_n$,
     being a linear combination of $\chi_{E_{k,n}}$'s,
    is also in $\mathcal{L}$. As was done previously, we have that 
    $(\varphi_n)\nearrow x$. Going through the previous argument backward, using 
    the classical Monotone Convergence Theorem, we have
    \[
        \int x \, d\mu = \lim_{n \rightarrow \infty} \int \varphi_n \, d\mu.
    \]
    In the same way as in \eqref{eqn:backward_1} and \eqref{eqn:backward_2} we get 
    \[
        \lim_{n \rightarrow \infty} \int \varphi_n \, d\mu 
            = \lim_{n \rightarrow \infty} \int \varphi_n,
    \]
    and applying Theorem \ref{thm:Monotone} we have 
    \[
        \int x = \lim_{n \rightarrow \infty} \int \varphi_n 
        =  \lim_{n \rightarrow \infty} \int \varphi_n \, d\mu   = \int x \, d\mu.
    \]
    We conclude that $x$ is $I$-integrable and that the integrals coincide.

    If $x \in \mathcal{L}$ is an arbitrary $I$-integrable function, not necessarily 
    nonnegative, then we can write $x$ as a difference of two 
    nonnegative functions 
    \[
        x= x\lor 0 - (-x)\lor 0  
    \]
    which are $I$-integrable by Proposition \ref{thm:lattice} and see that $x$ is also 
    integrable with respect to $\mu$. Similarly if $x$ is an arbitrary $\mu$-integrable
    function, not necessarily nonnegative, then $x$ can be written as a 
    difference between two nonnegative $\mu$-integrable functions and so $x$ is $I$-integrable.
    This completes the proof. 
\end{proof}    

As a consequence of Theorem \ref{thm:DS}, we obtain Carath\'eodory's Extension Theorem.
We refer the readers to the Appendix for the definitions of 
pre-measure (Definition \ref{def:pre_measure}), and the notion of a $\sigma$-algebra
generated by a ring (Definition \ref{def:sigma_gen}). 
Before we prove Carath\'eodory's Extension Theorem, we first consider a lemma.
\begin{lemma}\label{lemma:simple}
    Let $\mathcal{R}$ be a ring of subsets of $X$.
    Let $\mu : \mathcal{R} \rightarrow \mathbb{R}\cup \lbrace \infty \rbrace$ be a pre-measure.
    We define  
    \[
        T := \text{span}\set{\chi_R}{R \in \mathcal{R} \text{ and } \mu (R) < \infty}.
    \]
    Define $I: T \rightarrow \mathbb{R}$ as follows: 
    For each $ \sum_{i=1}^n c_i \chi_{R_i} \in T$ let 
    \[
        I\left(\sum_{i=1}^n c_i \chi_{R_i} \right) := \sum_{i=1}^n c_i \mu (R_i). 
    \]

    Then 
    \begin{enumerate}[(i)]
        \item $T$ is an extended vector lattice; and
        \item $I$ is an $I$-integral.
    \end{enumerate}
\end{lemma}
\begin{proof}
    Before we continue with the proof, we first make the following claim.
    \begin{claim}
        If $x= \sum_{i=1}^n c_i \chi_{R_i} \in T$, then there exists 
        a collection of pairwise disjoint sets 
        $\lbrace B_1, B_2, \cdots, B_n \rbrace$ in $\mathcal{R}$ and real numbers 
        $\lbrace b_1, b_2,\cdots, b_n \rbrace$ such that $x= \sum_{k=1}^l c_k \chi_{B_k}$.
    \end{claim}\\
    We provide the proof in the Appendix (Lemma \ref{lemma:disjoint_indicator}).

    We prove (i). By definition, $T$ is closed under linear combinations.
     We thus only need to show that $T$ is closed 
    under $\land$ and $\lor$. By the identities in Remark \ref{rem:absolute_val}, 
    it is sufficient to show that $T$ is closed under 
    absolute value. Let  $x= \sum_{i=1}^n c_i \chi_{R_i} \in T$. Using the above claim, we may assume
    without loss of generality that the sets $R_1, R_2, \cdots, R_n$ are pairwise disjoint. 
    We then get the following representation: for $t \in X$
    \[
        x(t):= 
        \begin{cases} 
            c_i &\text{ if } t \in R_i \\
            0 &\text{ else.}
        \end{cases}  
    \]
    From this it is seen that $|x|$ is given by 
    \[
        |x(t)|:= 
        \begin{cases}
            |c_i| &\text{ if } t \in R_i \\
            0 &\text{ else,}
        \end{cases}  
    \]
    for $t \in X$. Thus $x= \sum_{i=1}^n |c_i| \chi_{R_i} \in T$. Hence $T$ is a vector 
    lattice, and thus an extended vector lattice.\\

    We now show (ii). The proof that $I$ is well-defined and linear, 
    and hence satisfies (D1) from Definition \ref{def:I_Integral}, is a standard
    result in measure theory and we defer the proof to another source \cite[p. 78-79]{royden1988real}. 

    Assume that $x=\sum_{i=1}^n c_i \chi_{R_i} \in T_0$ such that $x \geq 0$. 
    Using the claim we may assume without loss of 
    generality that the sets $R_i$ for $i=1,2,\cdots, n$ are pairwise disjoint. 
    For each $i= 1,2, \cdots, n$ we pick a $t_i \in R_i$. Then we get that for $i=1,2,\cdots, n$
    \[
        c_i=   \sum_{i=1}^n c_i \chi_{R_i}(t_i)=x(t_i)\geq 0
    \]
    and so 
    \[
        I(x)=I\left( \sum_{i=1}^n c_i \chi_{R_i}\right)  =\sum_{i=1}^n c_i \mu (R_i) \geq 0.
    \]
    Since this holds for all $x \in T$ with $x \geq 0$, we conclude that $I$ is satisfies (D3) from Definition
    \ref{def:I_Integral}.

    Recall that $T$ consists of linear combinations of characteristic function of 
    sets in $\mathcal{R}$.
    Assume $(x_n) \subset T$ is a sequence such that $(x_n) \searrow 0$ and fix $\varepsilon>0$. 
    Let $$A:= \set{t \in X}{x_1(t)>0}.$$
    Note that, since $A$ is the finite union of 
     the sets in that characteristic functions making up $x_1$, $A \in \mathcal{R}$.
    For each $n \in \mathbb{N}$, define 
    $A_n:= \set{t \in X}{x_n(t) > \frac{\varepsilon}{2 \mu (A)}}$. 
    Note once again that, for every $n \in \mathbb{N}$, 
    because $A_n$ is the finite union of 
    some of the sets in that characteristic functions making up $x_n$, we have
   that $A_n \in \mathcal{R}$. Since $(x_n)$ is a decreasing sequence,
    we have the inclusions 
    \[
        A \supset A_1 \supset A_2 \supset A_3 \supset \cdots. 
    \]
    For each $n \in \mathbb{N}$, define $B_n:= A_n \backslash A_{n+1}$. Let $N \in \mathbb{N}$ be arbitrary. We then have
    \[
        \bigcup_{n=N}^\infty B_n 
        =\bigcup_{n=N}^\infty A_n \backslash A_{n+1} = A_N \backslash \bigcap_{n=1}^\infty A_n.
    \]
    Since we assumed that $(x_n) \searrow 0$, we have for all $t \in X$, that 
    $\lim_{n \rightarrow \infty} x_n(t) =0$, and so $\bigcap_{n=1}^\infty A_n= \emptyset$. 
    Consequently, $A_N = \bigcup_{n=N}^\infty B_n$. Since the sets $B_n$ are disjoint, 
    \[
        \sum_{n=N}^\infty \mu (B_n) = \mu (A_N) \leq \mu (A) < \infty.   
    \]
    If $x_1=0$, then since $(x_n) \searrow 0$, we have that $x_n=0$ for all $n\in \mathbb{N}$ 
    and the result holds. Therefore, without loss of generality, we may assume that $x_1 \neq 0$ and 
    thus, that $\max \set{x_1(t)}{t \in X}>0$.
    Choose $N$ large enough so that 
    \begin{align}
        \mu (A_N) = \sum_{n=N}^\infty \mu (B_n) < 
            \frac{\varepsilon}{2 \max \set{x_1(t)}{t \in X}}. \label{eqn:mu_sum}
    \end{align}
    
    If $n > N$ and $x_n$ is given by $x_n = \sum_{k=1}^K c_k \chi_{R_k}$
     where all the $R_k$'s are disjoint
    (note that we are hiding the $R_k$'s and $c_k$'s dependence on $n$), then we may write $x_n$
    in the folllowing way: 
    \[
        x_n =  \sum_{k=1}^N c_k \chi_{R_k \backslash A_n} + \sum_{k=1}^N c_k \chi_{R_k \cap A_n}.
    \]
    Note that if $c_k > \frac{\varepsilon}{2 \mu (A)}$, then for all $t \in R_k$ we have $x_n(t) >\frac{\varepsilon}{2 \mu (A)}$.
    Hence if $c_k > \frac{\varepsilon}{2 \mu (A)}$, then $R_k \subset A_n$ and $R_k \backslash A_n = \emptyset$.
    We thus have 
    \begin{align*}
        I(x_n)= \sum_{k=1}^N c_k \mu(R_k \backslash A_n) + \sum_{k=1}^N c_k \mu(R_k \cap  A_n) 
         &\leq \frac{\varepsilon}{2 \mu (A)} \sum_{k=1}^N \mu(R_k \backslash A_n) + \left(\max_k c_k\right) \mu (A_n) \\ 
         & \leq \frac{\varepsilon}{2 \mu (A)} \mu(A) + \left(\max_k c_k\right) \mu (A_n) \\
         & \leq \frac{\varepsilon}{2} + \left(\max \set{x_n(t)}{t \in X} \right)\mu(A_n) \\
         & \leq \frac{\varepsilon}{2} + \left(\max \set{x_1(t)}{t \in X} \right)\mu(A_N).
    \end{align*}
    Thus, in light of \eqref{eqn:mu_sum}, we have 
    \[
        I(x_n) \leq \frac{\varepsilon}{2} + \left( \max \set{x_1(t)}{t \in X}\mu(A_N )\right) 
        < \frac{\varepsilon}{2} + \left( \max \set{x_1(t)}{t \in X} \right) \frac{\varepsilon}{2 \max \set{x_1(t)}{t \in X}} 
        =\varepsilon.
    \]
    Thus, we have shown that $\lim_{n \rightarrow \infty}I(x_n)=0$, proving (D2) of Definition \ref{def:I_Integral}. This completes the proof.
\end{proof}

\begin{theorem}(Carath\'eodory's Extension Theorem) \label{thm:caratheodory}
    Let $\mathcal{R}$ be a ring  of subsets of $X$.
    Let $\mu: \mathcal{R} \rightarrow \mathbb{R}\cup\lbrace \infty \rbrace$ be a pre-measure.
    Let $\sigma(\mathcal{R})$ be the $\sigma$-algebra generated by $\mathcal{R}$. 
    Then $\mu$ extends to a measure 
    $\mu' : \sigma(\mathcal{R})  \rightarrow \mathbb{R}\cup \lbrace \infty \rbrace$.
\end{theorem}
\begin{proof}
     We want to apply Theorem \ref{thm:DS}. Let $T$ and $I: T \rightarrow \mathbb{R}$ be 
     as defined in Lemma \ref{lemma:simple}.
     By Lemma \ref{lemma:simple}, 
     $T$ is an extended vector lattice and
     $I$ is an $I$-integral. We thus only need to prove that 
     $\chi_X \land \varphi \in T$ for each $\varphi \in T$. 
     Note that if $\varphi = \sum_{i=1}^n c_i \chi_{R_i} \in T$ then
     \[
         \chi_X \land \varphi = \chi_X \land \left( \sum_{i=1}^n c_i \chi_{R_i} \right) 
            =\sum_{i=1} \min(c_i,1) \chi_{R_i} \in T. 
     \]
     Hence the condition for Theorem \ref{thm:DS} holds 
     and we conclude the existence of a $\sigma$-algebra
     $\mathcal{A}$ a measure $\mu'$ on $X$. We show 
     that $\mathcal{R} \subset \mathcal{A}$ and that 
     $\mu'=\mu$ on $\mathcal{R}$. 
     If it is the case that $R \in \mathcal{R}$ 
     with $\mu (R)<\infty$ we have 
     \[
        \mu' (R)= \int \chi_R = I(\chi_R)= \mu (R),
     \]
     and so $R \in \mathcal{A}$.

     If it is the case that $R \in \mathcal{R}$ such that $\mu (R) = \infty$, 
     then let $\varphi= \sum_{i=1}^n c_i \chi_{R_i} \in T$.
    Without loss of generality, by  Lemma \ref{lemma:disjoint_indicator}, we may assume that the 
    sets $R_i$ for $i=0,1,\dots, n$ are disjoint. We then have
     \begin{align}
        \varphi \land \chi_R = \left(\sum_{i=1}^n c_i \chi_{R_i} \right) \land \chi_R 
        = \sum_{c_i < 0} c_i \chi_{R_i} + \sum_{c_i \geq 0} c_i \chi_{R_i\cap R}  \in T. \label{eqn:T_land}
     \end{align}
     Since $\varphi \in T$ was arbitrary, we know that 
     $\varphi \land \chi_R \in T$ for each $\varphi \in T$
     and we conclude, using Lemma \ref{lemma:measurable_lemma} 
     that $\chi_R$ is a Daniell measurable function 
     and hence that $R \in \mathcal{A}$. 
     \medskip
     \begin{claim}
        We have $\mu'(R)= \mu(R)= \infty$.
     \end{claim}
     \medskip 
     \begin{claimproof}
        Suppose, to the contrary, that $\mu'(R)< \infty$.
        Then, by the definition of the measure obtained from Theorem \ref{thm:DS},
        $$ \mu'(R) = \int \chi_R = \overline{I}(\chi_R),$$
        and there exists a $\varphi \in T_1$ such that $\chi_R \leq \varphi$ and
        $$\overline{I}(\chi_R)\leq I_1(\varphi) \leq \overline{I}(\chi_R) + 1 = \mu'(R) +1< \infty.$$
        Since $\varphi \in T_1$, there exists a sequence $(\varphi_n) \subset T$ such that 
        $(\varphi_n) \nearrow \varphi$. Since $\chi_R \leq \varphi$, we have that 
        $(\varphi_n \land \chi_R) \nearrow \chi_R$. Let $n \in \mathbb{N}$. In the same way as 
        was done in \eqref{eqn:T_land}, we write
        \[
            \varphi_n \land \chi_R = \left(\sum_{i=1}^n c_i \chi_{R_i} \right) \land \chi_R 
            = \sum_{c_i < 0} c_i \chi_{R_i} + \sum_{c_i \geq 0} c_i \chi_{R_i\cap R}, 
        \]
        with the $R_i$'s disjoint and conclude that $\varphi_n \land \chi_R \in T$ 
        (once again, we hide the dependence of the $R_i$'s and $c_i$'s on $n$).
        If we replace all the positive $c_i$'s with 1, and drop all the negative $c_i$'s, 
        we are left with 
        \[
            \sum_{c_i \geq 0} \chi_{R_i\cap R}=\chi_{\bigcup_{c_i \geq 0} R_i} \in T.
        \]
        Define $R_n:= \bigcup_{c_i \geq 0} R_i$.
        Note that, since $\chi_{R_n} \in T$,
        we have that $\mu (R_n)< \infty$. Also note that we have 
        $$\varphi_n \land \chi_R \leq \chi_{R_n} \leq \chi_R.$$
        Therefore, since $(\varphi_n \land \chi_R)$ is increasing, 
        so is $(\chi_{R_n})$ and we have that $(\chi_{R_n}) \nearrow \chi_R$. 
        Thus $R= \bigcup_{n=1}^\infty R_n$ with $R_1\subset R_2 \subset \cdots \subset R$,
        and we have 
        \[
            \mu(R) =\mu\left(\bigcup_{n=1}^\infty R_n\right) =\lim_{n \rightarrow \infty} \mu(R_n)
                = \lim_{n \rightarrow \infty} \int \chi_{R_n} \leq \int \chi_R = \mu'(R)< \infty.
        \]
        This contradicts our assumption that $\mu(R) = \infty$, and we conclude that $\mu(R)= \mu'(R)= \infty$.
        This proves the claim.
    \end{claimproof}
        \medskip 

        We conclude that 
        $\mathcal{R} \subset \mathcal{A}$ and hence that 
        $\sigma(\mathcal{R})\subset \mathcal{A}$ and that the measure 
     $\mu= \mu'$ on $\mathcal{R}$.
\end{proof}
    \subsection{Application to $X=\mathbb{R}$} 
    We aim to construct the Lebesgue integrable functions on $\mathbb{R}$. 
    Let $T_0$ be the continuous functions on $\mathbb{R}$ with compact support.
    For the purposes 
    of this section, denote the Riemann integral of a function $f$ 
    by $\int_{-\infty}^\infty f(x) \, dx$.
    From elementary analysis, we know the Riemann integral to be linear 
    and positive. Hence the Riemann integral satisfies (D1) and (D3) of 
    Definition \ref{def:D-Integral}. To show (D3) of Definition \ref{def:D-Integral}, 
    we state it as a proposition.
    \begin{proposition}
        If we have a sequence $(f_n)$ of continuous functions with compact support such that 
        $(f_n)\searrow 0$, then 
        $\lim_{n \rightarrow \infty} \int_{-\infty}^\infty f_n(x)\, dx =0$.
    \end{proposition}
    \begin{proof}
        Let $(f_n)$ be a decreasing sequence of continuous functions with compact support such that 
        $(f_n)\searrow 0$.
        Then since $f_1$ has compact support, there exists 
        $a,b \in \mathbb{R}$ such that $f_1$ is zero outside of $[a,b]$. Since for all 
        $n \in \mathbb{N}$ we have that $0 \leq f_n \leq f_1$, 
        the functions $f_n$ are all zero outside 
        of $[a,b]$. 
        We thus have for all $n \in \mathbb{N}$,
        \[
            0 \leq \int_{-\infty}^\infty f_n(x)\, dx =\int_{a}^b f_n(x)\, dx 
                    \leq \int_{a}^b \left(\sup_{t \in [a,b]} f_n(t) \right)\, dx 
                    =(b-a)\sup_{t \in [a,b]} f_n(t).
        \]
        It thus suffices to show that 
        $\lim_{n \rightarrow \infty}\sup_{t \in [a,b]} f_n(t) =0$, i.e. 
        that $(f_n)$ converges to 0 uniformly. This follows from the standard result in analysis
        known as Dini's theorem which we state fully in the Appendix (Theorem \ref{thm:dini}).
    \end{proof}

    Consequently, the Riemann integral is a Daniell integral on $T_0$.
    From Theorem \ref{thm:algebra}
    and Theorem \ref{thm:measure}, we know there exists a $\sigma$-algebra $\mathcal{A}$ and 
    a measure $\mu$ on $\mathbb{R}$.
    We claim that $\mu$ is the Lebesgue measure.

    Denote the Lebesgue measure on $\mathbb{R}$ by $m$. 
    We first show that $\mathcal{A}$ contains the Borel $\sigma$-algebra. 
    It is sufficient to show that $\mathcal{A}$ 
    contains all the open intervals of finite length. 
    Let $J=(a,b)$ be an open interval. 
    We construct a sequence of continuous functions that increase toward $\chi_J$. 
    For each $n \in \mathbb{N}$, let $f_n : \mathbb{R} \rightarrow \mathbb{R}$ be 
    \begin{align}
        f_n(x):= 
        \begin{cases}
            \frac{2n}{b-a}(x-a) &\text{ if } x \in \left[a,a+\frac{b-a}{2n}\right]\\
            1 & \text{ if } x \in \left[a+\frac{b-a}{2n},b-\frac{b-a}{2n}\right]\\
            \frac{2n}{b-a}(b-x) &\text{ if } x \in \left[b-\frac{b-a}{2n}, b\right] \\
            0 &\text{ else. }
        \end{cases} \label{eqn:borel_incr}
    \end{align}

    \begin{figure}[h] 
        \includegraphics[width=\linewidth]{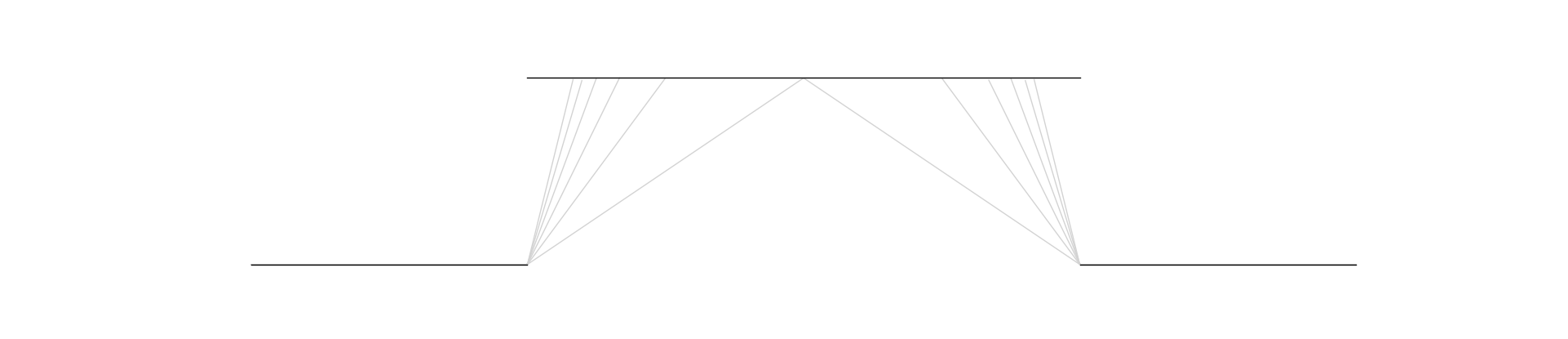}
        \caption{An illustration of $f_n$ in \eqref{eqn:borel_incr}.}\label{fig:borel_incr}
    \end{figure}

    Then $(f_n) \nearrow \chi_J$ (see Figure \ref{fig:borel_incr}) and $\chi_J \in T_1$. 
    Furthermore, if $I_1$ denotes the extension from 
    Definition \ref{definition:seqi} of the Riemann integral to $T_1$, we have 
    \begin{align*}
        I_1(\chi_J)&= \lim_{n \rightarrow \infty} \int_{-\infty}^\infty f_n(x) \, dx \\
            &= \lim_{n \rightarrow \infty} \left[ \int_a^{a+\frac{b-a}{2n}} \frac{2n}{b-a}(x-a) \, dx  
                + \int_{a+\frac{b-a}{2n}}^{b-\frac{b-a}{2n}} 1 \, dx 
                + \int_{b-\frac{b-a}{2n}}^b \frac{2n}{b-a}(b-x) \, dx\right]\\
            &= \lim_{n \rightarrow \infty} \left[\left(\frac{b-a}{4n}\right)+\left(b-a - 
                \frac{b-a}{n}\right)+\left(\frac{b-a}{4n}\right) \right] =b-a.
    \end{align*}
    Thus $\chi_J \in \mathcal{L}$ and is a Daniell measurable function.
    Hence $J \in \mathcal{A}$ and since this holds for an arbitrary open interval $J$, $\mathcal{A}$ 
    contains the Borel $\sigma$-algebra.
    Furthermore, the measure $\mu$ obtained from Theorem \ref{thm:measure} satisfies
    $$\mu(J)= \int \chi_J = b-a = m(J).$$ 
    Thus $\mu$ corresponds to $m$ when restricted to the open intervals, and hence the Borel $\sigma$-algebra.\\

    Using a result in measure theory (Theorem \ref{thm:appendix_completion} in the 
    Appendix), to show $\mathcal{A}$ contains the 
    Lebesgue $\sigma$-algebra and that $\mu=m$ on the Lebesgue measurable sets, it is
    sufficient to prove that $\mu$ is complete (for the definition of completeness, we refer the readers to 
    Definition \ref{def:complete_measure} in the Appendix).

    \begin{lemma}
        The measure $\mu$ given by Theorem \ref{thm:DS} is complete.
    \end{lemma}
    \begin{proof}
        Let $A\in \mathcal{A}$ such that $\mu(A)=0$ and $B\subset A$. 
        Since $B \subset A$, we have $0 \leq \chi_B \leq \chi_A$ and
         $$0 \leq \underline{I}(\chi_B) \leq \overline{I}(\chi_B)\leq \overline{I}(\chi_A) = \int \chi_A=0.$$
         Hence the inequality is, in fact, an equality and 
        $\chi_B \in \mathcal{L}$ with $\mu(B)=\int \chi_B =0$. 
        Thus $B \in \mathcal{A}$ and $B$ is a null set. We conclude $\mu$ is complete.
    \end{proof}

    We conclude that $\mu$ coincides with the Lebesgue measure on $\mathbb{R}$ 
    and that the Daniell Integral constructed in this way is the Lebesgue integral. 


    \section{ $S$-Integral} \label{section:Sintegral}
    
    As in Section \ref{section:extension}, let $X$ be a set and
    let $T_0 \subset {\overline{\mathbb{R}}}^X$ be an extended vector lattice. 
    We discussed Daniell's method of extending an $I$-integral on $T_0$ to a 
    larger class. Daniell mentioned another integral in his paper \cite{daniell_1} 
    which generalises the notion of the Stieltjies Integral. 
    In this section, denote the set of all nonnegative functions in $T_0$ by
    $$T_0^+:= \set{x \in T_0}{ 0 \leq x}.$$
      
    \begin{definition} \label{def:S-integral}
        Let $S:T_0 \rightarrow \mathbb{R}$.
    We call $S$ an {\it  $S$-integral} if for all $x,y \in T_0$, 
    $\alpha, \beta \in \mathbb{R}$, and sequences $(x_n) \subset T_0$:
    \begin{enumerate}[(S1)]
        \item $S(\alpha x + \beta y)= \alpha S(x)+ \beta S(y)$;
        \item $(x_n) \searrow 0$ implies $\lim_{n \rightarrow \infty} S(x_n)=0$; and
        \item there exists a function $M: T_0^+ \rightarrow \mathbb{R}$ with the properties that 
        if $\varphi,\psi \in T_0^+$ such that $\varphi \leq \psi$ 
        then $M(\varphi)\leq M(\psi)$ and for all
        $x\in T_0$ we have $|S(x)| \leq M(|x|)$.
    \end{enumerate}
    \end{definition}
    
    \begin{proposition}
        An $I$-integral is an   $S$-integral.
    \end{proposition}
    \begin{proof}
        Any $I$-integral satisfies (S1) and (S2), as these are 
        (D1) and (D2) of Definition \ref{def:I_Integral}.
         For all $x\in T_0$ we have $-|x|\leq x \leq |x|$ and 
         thus by (D3) we have $-I(|x|) \leq I(x) \leq I(|x|)$. 
        Thus $|I(x)| \leq I(|x|)$ and $I$ takes the role of $M$ in (S3). 
    \end{proof}
    
    We now discuss the relationship between the $S$-integral and the $I$-integral. We aim to show that an    $S$-integral can be 
    expressed as the difference between two $I$-integrals. 

    \begin{definition} \label{def:pos_int}
        If $x \in T_0^+$ define 
        \[P(x):=\sup \set{S(\varphi)}{\varphi \in T_0 \text{ such that } 0\leq \varphi \leq x}.\]
        Recall that a function $x \in T_0$, can be expressed as $x= x\lor0 - (-x)\lor0$ 
        where both $ x\lor0$ and $(-x)\lor0$ are nonnegative. 
        Then for $x \in T_0\setminus T_0^+$ define 
        \[S^+(x):=P(x\lor 0)-P((-x)\lor 0).\]
    \end{definition}
    
    \begin{remark} 
        In Definition \ref{def:pos_int}, the supremum of the set
        $$\set{S(\varphi) \in \mathbb{R}}{\varphi \in T_0 \text{ such that } 0\leq \varphi \leq x}$$
         exists. 
        For if $x \in T_0^+$ and $\varphi \in T_0$ with 
        $0 \leq \varphi \leq x$ we have 
        $S(\varphi)\leq M(|\varphi|) = M(\varphi) \leq M(x)$. 
        Thus the set $ \set{S(\varphi)}{\varphi \in T_0 \text{ such that } 0\leq \varphi \leq x}$ is bounded above by $M(x)$ and the supremum exists. 
        Also note that $S^+$ corresponds with $P$ on the set $T_0^+$.
    \end{remark} 
    
    \begin{theorem}
        $S^+$ is an $I$-integral.
    \end{theorem}
    We break the proof into a sequence of lemmas. 
    \begin{lemma} \label{lemma:pos_additive}
        Let $x_1,x_2 \in T_0^+$. Then $P(x_1+x_2)=P(x_1)+P(x_2)$.
    \end{lemma}
    \begin{proof}
    Let $x_1, x_2 \in T_0$ such that $0 \leq x_1$ and $0 \leq x_2$. 
    If $\varphi_1, \varphi_2 \in T_0$ with $0 \leq \varphi_1 \leq x_1$ and $0 \leq \varphi_2 \leq x_2$ then
    $0 \leq \varphi_1 + \varphi_2 \leq x_1 + x_2$ and so 
    $$S(\varphi_1)+S(\varphi_2) =S(\varphi_1 + \varphi_2) \leq P(x_1+x_2).$$
        Since this holds for all
    $\varphi_1, \varphi_2 \in T_0$ such that $0 \leq \varphi_1 \leq x_1$ and $0 \leq \varphi_2 \leq x_2$
    we have 
    \begin{align}
        P(x_1)+P(x_2)\leq P(x_1+x_2). \label{eqn:thms_1}
    \end{align}
    
    To show the reverse inequality, let $\varphi \in T_0$ such that $0 \leq \varphi \leq x_1+x_2$. 
    Then $\varphi- x_1 \leq x_2$ and since $0 \leq x_2$ we have 
    \begin{align}
        (\varphi-x_1)\lor 0 \leq x_2. \label{eqn:thms_2}
    \end{align}
        Furthermore,
    $$\varphi + x_1 = \varphi \land x_1 + \varphi \lor x_1= \varphi \land x_1 + (\varphi - x_1)\lor 0 + x_1.$$
    Subtracting $x_1$ from both sides we get $\varphi= \varphi \land x_1 + (\varphi - x_1)\lor 0$. 
    Using \eqref{eqn:thms_2} and the fact that
    $0 \leq \varphi \land x_1 \leq x_1$, we have 
    $$S(\varphi \land x_1) \leq P(x_1) $$
    and
    $$S((\varphi- x_1)\lor 0)\leq P(x_2).$$ 
    Consequently, we get that 
    \begin{align*}
        S(\varphi)&=S(\varphi \land x_1 + (\varphi - x_1)\lor 0) \\
        &=S(\varphi \land x_1) + S((\varphi - x_1)\lor 0) \\
        &\leq P(x_1) + P(x_2).
    \end{align*}
    Since this holds for all $\varphi \in T_0$ with $0 \leq \varphi \leq x_1+x_2$ we have 
    $$P(x_1+x_2)\leq P(x_1)+P(x_2).$$
    
    This, together with \eqref{eqn:thms_1} allows us to conclude 
    that for $x_1, x_2 \in T_0^+$ we have $$P(x_1+x_2)=P(x_1)+P(x_2).$$ 
    \end{proof}
    \begin{lemma} \label{lemma:Slinear}
        Let $x \in T_0^+$, then $P(cx)=cP(x)$ for each nonnegative scalar $c \in \mathbb{R}$.
    \end{lemma}
    \begin{proof}
        Let $x \in T_0^+$ and $c$ a nonnegative scalar. 
        Then 
        \begin{align*}
            P(cx)&= \sup \set{S(\varphi)}{\varphi \in T_0 \text{ such that } 0\leq \varphi \leq cx} \\
        &= \sup \set{S(c\varphi)}{\varphi \in T_0 \text{ such that } 0\leq \varphi \leq x} \\
        &= c \, \sup \set{S(\varphi)}{\varphi \in T_0 \text{ such that } 0\leq \varphi \leq x} \\
        &=c P(x). \qedhere
        \end{align*}
    \end{proof}
        We have shown that $P$ is additive and preserves 
        positive scalar multiplication on $T_0^+$. 
        We now proceed to prove linearity of $S^+$ on $T_0$.

        \newpage
    \begin{lemma}
        We have:
        \begin{enumerate}[(i)]
            \item For all $x \in T_0$, if $x = \varphi-\psi$ with 
                $\varphi \geq 0$ and $\psi \geq 0$, then $S^+(x)=P(\varphi)-P(\psi)$.
            \item For $x_1,x_2 \in T_0$, $S^+(x_1+x_2)=S^+(x_1)+S^+(x_2)$.
            \item For all $x \in T_0$ and $c \in \mathbb{R}$, $S^+(cx)=cS^+(x)$. 
        \end{enumerate}
    \end{lemma}
    \begin{proof}
        We prove (i).
        Let $x \in T_0$ and $\varphi,\psi \in T_0$ both nonnegative 
        such that $x= \varphi - \psi$. Recall that we can write $x= x\lor0 - (-x)\lor0$. Thus 
        \[
            \varphi-\psi= x=  x\lor0 - (-x)\lor0,
        \] and adding $\psi + (-x)\lor0$ to the equality we have 
        \[
            \varphi+(-x)\lor 0= x \lor 0 + \psi.
        \]
         Because all these functions are nonnegative, using Lemma \ref{lemma:pos_additive}, we have  
        \begin{align*}
            P(x\lor 0)+ P(\psi)&=P(x\lor 0+ \psi)\\
            &=P(\varphi + (-x)\lor 0)\\
            &=P(\varphi) + P((-x)\lor 0).
        \end{align*}
        Rearranging, we have
        $$S^+(x)= P(x\lor0) - P((-x)\lor0)=P(\varphi)-P(\psi).$$ 
        This proves (i) and we conclude that the way we decompose $x$ as the difference
        between two positive functions does not matter in the definition of $S^+$. 
    
        We prove (ii).
        Let $x_1, x_2 \in T_0$ and $\varphi_1, \varphi_2, \psi_1, \psi_2 \in T_0$ all nonnegative such that 
        $x_1= \varphi_1-\psi_1$ and $x_2=\varphi_2-\psi_2$. 
        Then we have $x_1+x_2 = (\varphi_1+\varphi_2)- (\psi_1+\psi_2)$. Hence, using Lemma \ref{lemma:pos_additive},  
        \begin{align*}
        S^+(x_1+x_2)&=P(\varphi_1+\varphi_2)-P(\psi_1+\psi_2)\\
        &=P(\varphi_1)+P(\varphi_2)-P(\psi_1)-P(\psi_2)\\
        &=P(\varphi_1)-P(\psi_1)+P(\varphi_2)-P(\psi_2).
        \end{align*}
        Using (i), we know that $S^+(x_1)=P(\varphi_1)-P(\psi_1)$ and 
        $S^+(x_2)=P(\varphi_2)-P(\psi_2)$. Thus we have 
        \[
            S^+(x_1+x_2)= P(\varphi_1)-P(\psi_1)+P(\varphi_2)-P(\psi_2)=S^+(x_1)+S^+(x_2).
        \]
         This shows that $S^+$ is additive.
    
        We prove (iii).
        Let $x \in T_0$. If $c \in \mathbb{R}$ is a nonnegative scalar, then we have 
        \begin{align*}
             S^+(cx)&=P((cx)\lor 0) - P((-cx)\lor 0) \\
             &= cP(x \lor 0)-cP((-x)\lor 0)\\
             &=cS^+(x).   
        \end{align*}
        If $c$ is a negative scalar, then 
        \begin{align*}
              S^+(cx)&=P((cx)\lor 0) - P((-cx)\lor 0)\\
              &=(-c)P((-x)\lor 0)- (-c) P(x\lor 0)\\
              &= cP(x \lor 0)-cP((-x)\lor 0)\\  
              &=cS^+(x).
        \end{align*}
            We conclude that $S^+(cx)=cS^+(x)$ for all scalars $c$, 
            and thus $S^+$ is linear. 
    \end{proof}

    Lemma \ref{lemma:Slinear} (ii) and (iii) is (D1) of Definition \ref{def:I_Integral}. 
    \begin{remark} \label{remark:Spos_D3}
        Property (D3) of Definition \ref{def:I_Integral} is satisfied by $S^+$, 
        for if $x \in T_0$ with $x \geq 0$, 
        then $S(0) \in \set{S(\varphi)}{\varphi \in T_0 \text{ such that } 0\leq \varphi \leq x}$ 
        and so $0=S(0)\leq S^+(x)$.\\
    \end{remark}

    What remains is to show that $S^+$ satisfies (D2) of Definition \ref{def:I_Integral}. 
    \begin{lemma}
        If $(x_n)\subset T_0$ is a sequence such that $(x_n)\searrow 0$ then 
        $$\lim_{n \rightarrow \infty}S^+(x_n)=0.$$
    \end{lemma}
    \begin{proof}
        
        Let $(x_n)$ be a sequence in $T_0$ such that $(x_n)\searrow 0$. Fix $\varepsilon>0$. 
        Then, using the definition of supremum, for each $n \in \mathbb{N}$,
        there exists a $\varphi_n \in T_0$ such that $0 \leq \varphi_n \leq x_n$ and 
        \begin{align}
        P(x_n)=  \sup \set{S(\varphi)}{\varphi \in T_0 \text{ such that } 0\leq \varphi \leq x_n}
                \leq S(\varphi_n)+ \frac{\varepsilon}{2^n}. \label{eqn:def_phi}
        \end{align}
         Define the sequence $(\psi_n) \subset T_0$ by 
         $$\psi_n:= \varphi_1 \land \varphi_2 \land \cdots \land \varphi_n,$$
         for $n \in \mathbb{N}$. 
        We show inductively that $P(x_n)\leq S(\psi_n)+ (\frac{1}{2}+\frac{1}{4}+\cdots+ 
            \frac{1}{2^n})\varepsilon$ for $n \in \mathbb{N}$.
        We know $\psi_1=\varphi_1$ and so the inequality 
        $P(x_1)\leq S(\psi_1)+\frac{\varepsilon}{2}$ 
        holds. Hence the base case is satisfied.
        Let $n \in \mathbb{N}$. As the induction hypothesis, assume that 
        \begin{align}
            P(x_n) \leq S(\psi_n)+ \left(\frac{1}{2}+\frac{1}{4}+\cdots+
                 \frac{1}{2^n}\right)\varepsilon. \label{eqn:induction}
        \end{align}
        Let $n \in \mathbb{N}$. We have $0 \leq \psi_n \leq x_n$ and $0 \leq \varphi_{n+1} \leq x_{n+1} \leq x_n$. 
        Consequently, we get the following inequality 
        $$0 \leq \psi_n \lor \varphi_{n+1}\leq x_{n}$$
         and, applying the induction hypothesis, we have 
        \begin{align}
           S(\psi_n \lor \varphi_{n+1})\leq P(x_n)
           \leq S(\psi_n)+ \left(\frac{1}{2}+\frac{1}{4}+\cdots+ 
           \frac{1}{2^n}\right)\varepsilon. \label{eqn:thms_3}
        \end{align}
        Also, recall the vector lattice identity 
        $$\varphi_{n+1} + \psi_n = \varphi_{n+1}\lor \psi_n+ \varphi_{n+1} \land \psi_n,$$
        and note that $\varphi_{n+1} \land \psi_n$ is equal to 
        $\psi_{n+1}$ by definition. We thus have the equality
        \[
            \varphi_{n+1} + \psi_n=\varphi_{n+1}\lor \psi_n+ \psi_{n+1}.
        \]
        Isolating $\psi_{n+1}$ in the above equation, we get  
        $$\psi_{n+1}=\varphi_{n+1}+\psi_n - \varphi_{n+1}\lor \psi_n.$$ 
        From our choice of $\varphi_{n+1}$ in \eqref{eqn:def_phi} we have 
        \begin{align} 
            P(x_{n+1}) -  \frac{\varepsilon}{2^{n+1}}\leq S(\varphi_{n+1}), \label{eqn:thms_4}
        \end{align}
        and from \eqref{eqn:thms_3} we have 
        \begin{align}
            -\left( \frac{1}{2}+\frac{1}{4}+\cdots+ \frac{1}{2^n}\right)  \varepsilon \leq
                    S(\psi_n)-S(\psi_n \lor \varphi_{n+1}). \label{eqn:thms_5}
        \end{align}
        Adding \eqref{eqn:thms_4} and \eqref{eqn:thms_5} together we obtain 
        \[
            P(x_{n+1})-\frac{1}{2^{n+1}}\varepsilon-\left( \frac{1}{2}+\frac{1}{4}+\cdots+ \frac{1}{2^n}\right)  \varepsilon
                    \leq S(\varphi_{n+1})+S(\psi_n)-S(\psi_n \lor \varphi_{n+1})=S(\psi_{n+1}),
        \]
        and so
        \[
            P(x_{n+1}) \leq S(\psi_{n+1})+ \left( \frac{1}{2}+\frac{1}{4}+\cdots+ \frac{1}{2^{n+1}}\right)\varepsilon.
        \]
        This concludes the inductive step and we have shown \eqref{eqn:induction} to hold for all $n \in \mathbb{N}$. Hence   
        \begin{align}
            S^+(x_n)=P(x_n)\leq S(\psi_n)+\left(\frac{1}{2}+\frac{1}{4}+\cdots+ \frac{1}{2^n}\right) \varepsilon 
            < S(\psi_n)+\varepsilon \label{eqn:S_epsilon}
        \end{align}
        for all $n\in \mathbb{N}$. 
        
        Property (D3) from Definition \ref{def:I_Integral} for $S^+$ 
        was treated in Remark \ref{remark:Spos_D3}. 
        This allows us to conclude that for all $n \in \mathbb{N}$
        we have $0\leq S^+(x_n)$. Furthermore, Property (D3) from Definition \ref{def:I_Integral}
        implies that the sequence $(S^+(x_n))$ is decreasing. 
        Thus $\lim_{n\rightarrow \infty} S^+(x_n)$ exists and
        is nonnegative. Recall that $\psi_n \leq x_n$ for each $n \in \mathbb{N}$. Therefore, since 
        $(x_n) \searrow 0$, we have $(\psi_n) \searrow 0$. 
        Using \eqref{eqn:S_epsilon} and (S2) from Definition \ref{def:S-integral}
        we know that the limits satisfy the following:
        \[
            0\leq \lim_{n\rightarrow \infty} S^+(x_n)\leq 
                \lim_{n\rightarrow \infty}S(\psi_n)+\varepsilon =\varepsilon.
        \]
        This holds for every $\varepsilon>0$ and so 
        $\lim_{n\rightarrow \infty} S^+(x_n)=0$. We therefore conclude property (D2) of Definition 
        \ref{def:I_Integral} as well.
    \end{proof}

    We have shown that $S^+$ satisfies (D1), (D2), and (D3) of Definition \ref{def:I_Integral} 
    and hence that $S^+$ is an $I$-integral. 
    Associated with $S$ there are two other $I$-integrals.
    \begin{definition}\label{def:riesz_decomp}
        Define $S^-: T_0 \rightarrow \mathbb{R}$ by 
        \[
            S^-(x):=S^+(x)-S(x).
        \]
        Define $|S|:T_0 \rightarrow \mathbb{R}$ by 
        \[
            |S|(x):=S^+(x)+ S^-(x)=2S^+(x)-S(x).
        \]
    \end{definition}
    
    \begin{proposition}\label{prop:riesz_decomp}
        Both $S^-$ and $|S|$ are $I$-integrals.
    \end{proposition}
    \begin{proof}
        The fact that $S^-$ and $|S|$ are linear follows
         as both $S^+$ and $S$ are linear. If $x \in T_0$ with $0 \leq x$, 
        then $S(x)$ is in the set over which we take supremum in Definition 
        \ref{def:pos_int}. Thus we have $S(x)\leq S^+(x)$
        and $0\leq S^+(x)-S(x)=S^-(x)$. Hence we also have $0 \leq S^+(x)+S^-(x)=|S|(x)$.
         If $(x_n)$ is a sequence in $T_0$ such that $(x_n) \searrow 0$, 
        then $S^-(x_n)$ and $|S|(x_n)$ are both decreasing (we have shown them to preserve order). 
        Furthermore
        $$ \lim_{n \rightarrow \infty } S^-(x_n)=\lim_{n \rightarrow \infty} 
            S^+(x_n)-\lim_{n \rightarrow \infty} S(x_n)=0,$$
        and
        $$ \lim_{n \rightarrow \infty } |S|(x_n)=\lim_{n \rightarrow \infty} 2S^+(x_n)
            -\lim_{n \rightarrow \infty} S(x_n)=0.$$ \qedhere
    \end{proof}
    \begin{remark}
        From Definition \ref{def:riesz_decomp}, one has $S= S^+ - S^-$ and that using Proposition
        \ref{prop:riesz_decomp} both $S^+$ and $S^-$ are $I$-integrals. If one phrases this 
        in the language of Riesz theory, we say $S$ has been decomposed into the difference 
        between two positive linear functionals. This result is stated formally as:
        \begin{quotation}
            If $M$ and $N$ are Riesz spaces with $N$ Dedekind complete, and if $T:M \rightarrow N$
            is an order bounded linear functional, then there exists positive linear functionals 
            $T^+: M \rightarrow N$ and $T^-: M \rightarrow N$ such that $T= T^+-T^-$.
        \end{quotation}
        
        Today we trace back the origins of Riesz theory to the work by 
        Frigyes Riesz around the time when he delivered the presentation 
        ``Sur la décomposition des opérations fonctionelles linéaires'' at the International 
        Congress of Mathematics, Bologna in 1928. However, as pointed out by Koos Grobler 
        in his plenary talk ``101 years of vector lattice theory'' given at the Positivity conference 
        in Pretoria in 2019, Daniell preempted much of this work by almost a decade. 
    \end{remark}
    In Section \ref{section:extension} we discussed how to extend a $I$-integral from class $T_0$. 
    We thus have three $I$-integrals: $S^+, S^-$, and $|S|$. Let the classes to which these
    integrals extend be $\mathcal{L}_+$ for $S^+$, $\mathcal{L}_-$ for $S^-$, 
    and $\mathcal{L}$ for $|S|$. We show that $\mathcal{L_+} \cap \mathcal{L_-}= \mathcal{L}$. 
    We then use these extensions to define an extension for an    $S$-integral. \\
    
    In Section \ref{section:extension}, we defined 
    \[
        T_1:= \lbrace x \in {\overline{\mathbb{R}}}^X: 
            \text{ there exists a sequence } (x_n) \subset T_0 \text{ such that }
            (x_n) \nearrow x\rbrace,
    \]
    and extended the given $I$-integral to a functional 
    $I_1: T_1 \rightarrow \mathbb{R}\cup \lbrace\infty \rbrace$. 
    Recall from Definition~\ref{definition:seqi} that
     if $x \in T_1$ with $(x_n)$ an increasing sequence in 
    $T_0$ such that $(x_n)\nearrow x$, then $I_1$ was defined by 
    \[
        I_1(x) := \lim_{n \rightarrow \infty} I(x_n).
    \]
    We further defined the upper sum 
    $\overline{I}(x):= \inf \set{I_1(\varphi)}{x\leq \varphi \text{ and } \varphi \in T_1}$.
    
    Denote by $I_+$ the extension of $S^+$ to $T_1$, by $I_-$ the extension 
    of $S^-$ to $T_1$, and by $I_1$ the 
    extension of $|S|$ to $T_1$. We further denote by $\overline{I}_1(x)$ 
    the upper sum for $|S|$, by $\overline{I}_+(x)$ the upper sum for $S^+$, and 
    by $\overline{I}_-(x)$ the upper sum for $S^-$. 
    We then have the following lemma:
    \begin{lemma} \label{lemma:modlemma}
        For all $x \in T_1$ we have
        \[
            I_1(x)=I_+(x)+I_-(x),
        \]
        and for all $x \in {\overline{\mathbb{R}}}^X$ we have 
        \[
            \overline{I}_1(x)=\overline{I}_+(x)+\overline{I}_-(x).
        \]
    \end{lemma}
    \begin{proof}
        We showed in Section \ref{section:extension} that the definitions for 
        $I_1$, $I_+$, and $I_+$ are well-defined (we refer the readers to Corollary \ref{cor:well_def}).
        Let $x \in T_1$ with $(x_n)$ an increasing sequence in $T_0$ such that $(x_n)\nearrow x$.
         Then 
         $$I_1(x)=\lim_{n \rightarrow \infty} |S|(x_n)=
            \lim_{n \rightarrow \infty} S^+(x_n)+S^-(x_n)=I_+(x)+I_-(x).$$ 
         This proves the first part.\\
    
         Let $x\in {\overline{\mathbb{R}}}^X$. If $\varphi \in T_1$ 
         such that $x \leq \varphi$, then 
         $$\overline{I}_+(x)+\overline{I}_-(x) \leq I_+(\varphi)+I_-(\varphi)=I_1(\varphi).$$
         This holds for all $\varphi \in T_1$, thus 
         \begin{align}
            \overline{I}_+(x)+\overline{I}_-(x) \leq \overline{I}_1(x). \label{eqn:modlemma_1}
         \end{align}
         
         To prove the reverse inequality, fix $\varepsilon>0$.  
        There exists $\varphi_1, \varphi_2 \in T_1$ with 
        $x\leq \varphi_1$ and $x\leq \varphi_2$ such that
         $$ I_+(\varphi_1)- \varepsilon<\overline{I}_+(x) 
            \text{ and }  I_+(\varphi_2)- \varepsilon<\overline{I}_-(x).$$ 
         Let $\psi := \varphi_1 \land \varphi_2$. Note that $x\leq \psi$ and 
         that $\psi \in T_1$. Then 
         \[
            \overline{I}_1(x)-2\varepsilon\leq I_1(\psi)-2\varepsilon
                =I_+(\psi)+I_-(\psi)-2\varepsilon\leq I_+(\varphi_1)+I_-(\varphi_2)-2\varepsilon
                <\overline{I}_+(x)+\overline{I}_-(x).
         \]
         Since this holds for all positive $\varepsilon$ we have 
         $\overline{I}_1(x)\leq \overline{I}_+(x)+\overline{I}_-(x)$.

         This, together with \eqref{eqn:modlemma_1}, 
         allows us to conclude that $\overline{I}_1(x)=\overline{I}_+(x)+\overline{I}_-(x)$. 
         This concludes the proof.
    \end{proof}
    
    Recall that we defined a function $x \in {\overline{\mathbb{R}}}^X$ to be $I$-integrable for a 
    given $I$-integral $I$ if $\overline{I}(x)=\underline{I}(x)$
     and this value is finite; furthermore, the integral is defined to be the common value.    
     Similarly, we define the terms $S^+$-integrable, $S^-$-integrable, 
     and $|S|$-integrable.
    Denote the Daniell integral induced by $S_+$ with $\mathcal{S}_+$, the Daniell integral 
    induced by $S_-$ with $\mathcal{S}_-$ and the Daniell integral induced by $|S|$ with
    $|\mathcal{S}|$. 
    
    \begin{theorem}
        Let $x \in {\overline{\mathbb{R}}}^X$. Then $x$ is $S_+$- and $S_-$- 
        integrable if and only if $x$ is $|S|$-integrable. 
        Furthermore, if $x$ is $|S|$-integrable, then 
        $|\mathcal{S}|(x)=\mathcal{S}_+(x)+\mathcal{S}_-(x)$.
    \end{theorem}
    \begin{proof}
        Let $x \in {\overline{\mathbb{R}}}^X$. Assume $x$ is $|S|$-integrable.
         We then have $\overline{I}_1(x)=\underline{I}_1(x)$ and is finite. 
        Thus by Lemma \ref{lemma:modlemma}  
        $$\overline{I}_+(x)+\overline{I}_-(x)= \overline{I}_1(x)=\underline{I}_1(x)=\underline{I}_+(x)+\underline{I}_-(x),$$ 
        and so 
        \begin{align}
            \overline{I}_+(x)-\underline{I}_+(x)+\overline{I}_-(x)-\underline{I}_-(x)=0. 
            \label{eqn:modthm_1}
        \end{align}
        By Lemma \ref{BigLemma} (iv) 
        \[
             0 \leq \overline{I}_+(x) - \underline{I}_+(x)  
             \text{ and } 0 \leq \overline{I}_-(x) - \underline{I}_-(x).
        \] 
        Therefore both differences in \eqref{eqn:modthm_1} are nonnegative, 
        and so both must be equal to 0. Thus
        $\overline{I}_+(x)=\underline{I}_+(x)$ and $\overline{I}_-(x)=\underline{I}_-(x)$ 
        and we conclude $x$ to be $S_+$- and $S_-$-integrable.
        
        To prove the converse, assume $x$ is $S_+$- and $S_-$-integrable. 
        We then have by Lemma \ref{lemma:modlemma} that 
        \[
            \overline{I}_1(x)= \overline{I}_+(x)+\overline{I}_-(x)=
                \underline{I}_+(x)+\underline{I}_-(x)=\underline{I}_1(x).
        \]
        Thus $x$ is $|S|$ integrable. 

        If $x$ is an integrable function with respect to the three integrals $S_+$, $S_-$ 
        and $|S|$, then using Lemma \ref{lemma:modlemma} and 
        recalling the definitions of the extended $S^+$-, $S^-$, and $|S|$-integrals we have 
        \[
            |\mathcal{S}|(x)=\overline{I}_1(x)= \overline{I}_+(x)+\overline{I}_-(x)
                =\mathcal{S}_+(x)+\mathcal{S}_-(x).
        \]
        This completes the proof.
    \end{proof}
    
    \begin{corollary}
        We have $\mathcal{L_+} \cap \mathcal{L_-}= \mathcal{L}$.
    \end{corollary}
    
    \begin{definition}
        Define $\mathcal{S}: \mathcal{L} \rightarrow \mathbb{R}$ by 
        \[
            \mathcal{S}(x)= \mathcal{S}_+ (x) - \mathcal{S}_- (x).
        \]
    \end{definition}
    \begin{proposition}
        The functional $\mathcal{S}$ extends $S$ and satisfies (S1), (S2), and (S3) of 
        Definition \ref{def:S-integral}.
    \end{proposition}
    \begin{proof}
        Note that $\mathcal{S}$ is defined on $\mathcal{L}$ 
        which contains $T_0$ and that on $T_0$ we have 
        \[
            S(x)=S^+(x)-S^-(x)=\mathcal{S}^+(x)-\mathcal{S}^-(x)= \mathcal{S}(x).
        \]  

        Properties (S1) and (S2) of Definition \ref{def:S-integral} follow as $\mathcal{S}$ is a 
        difference between two $I$-integrals, and $|\mathcal{S}|$ fulfills the role
        of $M$ in (S3) of Definition \ref{def:S-integral}. 
    \end{proof}

\section{The applications by Norbert Wiener} \label{Section:Wiener}
One of the mathematicians who saw the value in Daniell's 
ideas was the American mathematician and philosopher Norbert Wiener (1894-1964). 
We present in this section a brief summary of the development by Wiener. Te keep the 
material in this section brief, we do not treat everything in the fullest detail. 

\subsection{Extending Dirichlet problems}
In 1923, Wiener published the article 
``Discontinuous Boundary Conditions and the Dirichlet Problem'' detailing an application 
of Daniell's integral to extending the family of solutions to the Dirichlet
problem for the Laplace equation \cite{wiener_dirichlet}. This application provides a context 
where the integral, rather than the measure, is the most natural tool to use. 

\begin{problem}[The Dirichlet Problem in two dimensions] \label{prob:dirichlet}
Let $D$ be a bounded domain in $\mathbb{R}^2$. 
We denote the boundary of $D$ by $\partial D$.
Let $g: \partial D \rightarrow \mathbb{R}$ be continuous. 
For a twice differentiable function $u:D \rightarrow \mathbb{R}$, let
 $\Delta$ denote the Laplacian operator, i.e. 
\[
\Delta u:= \partial_1^2 u + \partial_2^2u .  
\]
Note that $\Delta$ is linear. 
Find a function $u\in C^2(D) \cap C(\overline{D})$ such that 
\begin{align}
\Delta u=0 \text{ on } D; \text{ and }  
\end{align}
\[
u=g \text{ on } \partial D.  
\]
\end{problem}
\begin{definition}
    A function $u \in C(D)$ that satisfies 
    \[
        \Delta u=0 \text{ on } D  
    \]
    is called {\it harmonic}.
\end{definition}

\begin{definition}\label{def:dir_integral}
    For the rest of this section, let $D$ be a bounded domain 
    (a connected, bounded, open set) such that for every continuous function $g$ on $\partial D$,
    we have a solution $u_g \in C^2(D) \cap C(\overline{D})$ satisfying Problem \ref{prob:dirichlet}. 
Let $x \in D$ be arbitrary and define the functional $I_x: C(\partial D) \rightarrow \mathbb{R}$
by 
\begin{align}
    I_x(g):=u_g(x).
\end{align} 
\end{definition}

As $C(\partial D)$ is a vector 
lattice, it therefore satisfies the conditions needed for $T_0$. 
We now need to show that $I_x$ is an $I$-integral, i.e.  
$I_x$ satisfies the conditions (D1)-(D3) of Definition \ref{def:I_Integral}. 
To this end we recall a well-known property of solutions to Problem \ref{prob:dirichlet}:
\begin{theorem}\textup{(Strong Maximal Principle, \cite[Ch. 2, Theorem~4]{evans1998partial})}

Suppose $u \in C^2(D)\cap C(\overline{D})$ is a solution to Problem \ref{prob:dirichlet},
then 
\[
    \max_{\overline{D}} u = \max_{\partial D} u.
\]
\end{theorem}   
From this, one notes that 
if $u$ attains its maximum in $D$ then $u$ must be constant on $\overline{D}$. 
It also follows 
by replacing $u$ with $-u$ that 
\[
\min_{\overline{D}} u = \min_{\partial D} u.
\]    

Furthermore, if $u_1$ and $u_2$ both satisfies Problem \ref{prob:dirichlet} with the same 
boundary condition $g$, then $u_1=u_2$. This is 
because, by linearity, $v=u_1-u_2$ will satisfy
Problem \ref{prob:dirichlet} with boundary condition 0 
and so $v=0$. Consequently, $u_1=u_2$ and 
we have uniqueness of the solutions to Problem \ref{prob:dirichlet}.
Therefore $I_x$ from Definition \ref{def:dir_integral} is indeed well-defined.

\begin{theorem}
For each $x \in D$, the map $I_x$ from Definition \ref{def:dir_integral} is an $I$-integral. 
\end{theorem}

\begin{proof}

Let $x\in D$. 
To verify (D1) of Definition \ref{def:I_Integral} for $I_x$,
consider arbitrary $g_1, g_2 \in C(\partial D)$ and $\alpha, \beta \in \mathbb{R}$
with $u_1$ being the solution to Problem \ref{prob:dirichlet} with boundary condition $g_1$ and 
$u_2$ being the solution corresponding to $g_2$. Then by linearity of $\Delta$ we have 
\[
\Delta(\alpha u_1 + \beta u_2)= \alpha \Delta u_1 + \beta \Delta u_2 =0 \text{ on } D,
\]
and 
\[
\alpha u_1 + \beta u_2 = \alpha g_1 + \beta g_2   \text{ on } \partial D.
\]
By uniqueness of the solution we then have $u_{\alpha g_1 + \beta g_2}=\alpha u_{g_1}+\beta u_{g_2}$, 
and consequently 
\[
I_x(\alpha g_1 + \beta g_2)=  u_{\alpha g_1 + \beta g_2}(x)=\alpha u_{g_1}(x)+\beta u_{g_2}(x)
= \alpha I_x(g_1)+ \beta I_x(g_2).
\]
Since $g_1, g_2 \in C(\partial D)$ were arbitrary, we have that 
$I_x$ is linear, proving (D1) of Definition \ref{def:I_Integral}.

To verify (D3) of Definition \ref{def:I_Integral}, let $g \in C(\partial D)$ be such that $g \geq 0$. 
We then have
\[
0 \leq \min_{\partial D} g =\min_{\partial D} u_g =\min_{\overline{D}} u_g \leq u_g(x) =I_x(g),
\]
which shows (D3) of Definition \ref{def:I_Integral}.

Finally, if $ (g_n) \subset C(\partial D)$ such that $ (g_n) \searrow 0$ then 
for each $n \in \mathbb{N}$ we have $0 \leq  g_n- g_{n+1}$, 
and hence that 
\[
    0 \leq I_x(g_n- g_{n+1})=I_x(g_n)-I_x(g_{n+1}).
\]
Note also that for each $n \in \mathbb{N}$ the function $g_n \geq 0$ and hence by what we have 
already shown, $0 \leq I_x(g_n)$.
Consequently, for each $n\in \mathbb{N}$, we have 
\begin{align}\label{eqn:descend}
0 \leq I_x(g_{n+1})\leq I_x(g_n).
\end{align}
To show that $\lim_{n\rightarrow \infty} I_x(g_n) =0$, note that 
\[
0 \leq I_x(g_n) = u_{g_n}(x) \leq \max_{\overline{D}} u_{g_n} = \max_{\partial D} u_{g_n}=\max_{\partial D} g_n.
\]
As $(g_n) \searrow 0$, we have that $(\max_{\partial D} g_n) \searrow 0$ and therefore
$\lim_{n\rightarrow \infty} I_x(g_n) =0$. This together with \eqref{eqn:descend} proves (D2) of Definition \ref{def:I_Integral}.
\end{proof}

The theory of Daniell now applies. For each $x \in D$ there
exists an extended vector lattice of functions $\mathcal{L}_x$ on $\partial D$. 
Furthermore, the functional $I_x$ extends to a Daniell Integral $\int_x : \mathcal{L}_x \rightarrow \mathbb{R}$. 
We now define a function that may be viewed as a solution to Problem \ref{prob:dirichlet}
with more general boundary conditions. 
\begin{definition}
If $f \in \bigcap_{x \in D} \mathcal{L}_x$, define the function $v_f:D \rightarrow \mathbb{R}$
by 
\[
    v_f(x):= \int_x f.
    \]  
\end{definition}
Note that the condition $f \in \bigcap_{x \in D} \mathcal{L}_x$ means that $f$ is 
$I_x$-integrable for each $x \in D$. We now prove a theorem which says that the requirement
 $f \in \bigcap_{x \in D} \mathcal{L}_x$
is superficial and that we only require $f$ to be $I_x$-integrable for 
one element in $D$. 

To do this, we first cite two theorems due to Harnack:
\begin{theorem}\textup{(Harnack's Inequality, \cite[Ch. 2, Theorem~11]{evans1998partial})}\label{thm:harnack_ineq}
    Let $\Omega \subset \mathbb{R}^2$ be a connected domain and let $u: \Omega \rightarrow \mathbb{R}$
    be harmonic and nonnegative. If $V\subset \Omega$ is open such that the closure 
    $\overline{V}\subset \Omega$, then there exists a constant $C>0$ depending only on 
    $\Omega$ and $V$, such that $ \sup_{x \in V} u(x) \leq C \inf_{x \in V} u(x)$.
\end{theorem}
\begin{theorem}\textup{(Harnack's Theorem, \cite[Ch. 11, Theorem~11.11]{rudin2006real})}\label{thm:harnack}
    Let $G \subset \mathbb{R}^2$ be a bounded and connected domain. 
    If $u_n: G \rightarrow \mathbb{R}$ is an increasing sequence of harmonic functions 
    on $G$, then either $\lim_{n \rightarrow \infty} u_n(x)$ is 
    finite for every $x \in G$ or $\lim_{n \rightarrow \infty} u_n(x)$ is infinite 
    for every $x \in G$. Furthermore, in the case when the limits are all finite, then 
    \[
        u:= \lim_{n \rightarrow \infty} u_n  
        \]
        is harmonic.
\end{theorem}

These results are used in the proof of the following theorem.
\begin{theorem}
If $x \in D$ and $f \in \mathcal{L}_x$, then $f \in \mathcal{L}_y$ for all 
$ y \in D$. Furthermore, in this case $v_f$ is harmonic. 
\end{theorem}

\begin{proof}
Let $x \in D$ and assume that $f \in \mathcal{L}_x$. 
Let $T_{1,x}$ be as in Definition \ref{definition:Tinc}. Note that there is no real dependence 
on $x$ in this definition, however we keep the subscript as a reminder that we are specifically working 
with the $I$-integral $I_x$.
We first treat the case when $f \in T_{1,x}$. There exists 
a sequence $(g_n) \subset C(\partial D)$ such that $(g_n)\nearrow f$. In this case 
\[
    \int_x f= \lim_{n \rightarrow \infty} I_x (g_n)= \lim_{n \rightarrow \infty} u_{g_n}(x).  
\]
Since we assumed that $f \in \mathcal{L}_x$, we know that $\int_x f< \infty$. 
Using Theorem \ref{thm:harnack}, since $f$ is (Daniell) integrable with respect to $I_x$, 
we have that $\int_x f= \lim_{n \rightarrow \infty} u_{g_n}(x)$ is finite
and we get that $\int_y f= \lim_{n \rightarrow \infty} u_{g_n}(y)$ is finite for each $y \in D$. 
Hence $f \in \mathcal{L}_y$ for each $y \in D$. Also note that the pointwise limit
$v_f =\lim_{n \rightarrow \infty} g_n$ is harmonic in $D$. \\

We now consider the general case where $f \in \mathcal{L}_x$. Let $n \in \mathbb{N}$. 
Let $I_{1,x}: T_{1,x} \rightarrow \mathbb{R} \cup \lbrace \infty \rbrace$ be as in 
Definition \ref{definition:seqi}. 
Recall that $f \in \mathcal{L}_x$ if and only if 
\[
    \overline{I}_x(f)=\underline{I}_x(f)=-\overline{I}_x(-f),   
\]
and that this common value is denoted $\int_x f$. Hence, for each $n \in \mathbb{N}$, there 
exists $g_n \in T_{1,x}$ with $f \leq g_n$ and
\begin{align}
    I_{1,x}(g_n)\leq \int_x f +\frac{1}{n}, \label{eqn:dirichlet1}
\end{align}
and there exists $h_n$ such that $-h_n \in T_1$ with $-f \leq -h_n$ and 
\begin{align}
    I_{1,x}(-h_n) \leq \int_x (-f) +\frac{1}{n}.  \label{eqn:dirichlet2}
\end{align}
Since $f$ is integrable, $I_{1,x}(g_n)$ and $I_{1,x}(-h_n)$ are both finite. 
Furthermore, since $g_n$ and $-h_n$ are both in $T_{1,x}$, we may conclude 
that both $h_n$ and $g_n$ are $I_x$-integrable. We may therefore write 
\[
    I_{1,x} (g_n) =\int_x g_n \text{ and } I_{1,x} (-h_n) = \int_x -h_n = -\int_x h_n.
\]
Since $T_{1,x}$ closed with respect to $\land$ and $\lor$, 
we may assume without loss of generality that the sequence $(g_n)$ is decreasing, 
and that the sequence $(h_n)$ is increasing.
Note that for each $n \in \mathbb{N}$ we have 
\[
    h_n \leq f \leq g_n,  
\]
and from what we have shown in the special case, that both
$v_{g_n}$ and $v_{h_n}$ are harmonic. Furthermore, it follows that for each $n \in \mathbb{N}$ 
and each $y \in D$ we have 
\[
    v_{h_n}(y)= \int_y h_n \leq \int_y g_n = v_{g_n}(y).
\]
Hence $v_{h_n}\leq v_{g_n}$.

Fix an open ball $B(x,a)$ in $D$ of radius $a>0$ around $x$ such that the closure 
$\overline{B(x,a)} \subset D$. 
Using Theorem \ref{thm:harnack_ineq} we conclude that 
for each $n \in \mathbb{N}$ and $y \in V$ we therefore have that $v_{g_n}-v_{h_n}$ is harmonic
on $D$ and that 
\begin{align}
    0 \leq v_{g_n}(y)-v_{h_n}(y) \leq \sup_{x \in B(x,a)} (v_{g_n}-v_{h_n}) 
    \leq C \inf _{x \in B(x,a)} (v_{g_n}-v_{h_n})
    \leq C (v_{g_n}(x)-v_{h_n}(x)). \label{eqn:dirichlet3}
\end{align}
From \eqref{eqn:dirichlet1} and \eqref{eqn:dirichlet2} we know that 
\[
    0 \leq v_{g_n}(x)-v_{h_n}(x) = \int_x g_n - \int_x h_n 
    \leq \int_x f + \int_x (-f) + \frac{2}{n} = \frac{2}{n}.
\]
Combining this with \eqref{eqn:dirichlet3} we get for all $n \in \mathbb{N}$ and all $y \in B(x,a)$
that 
\[
    0 \leq v_{g_n}(y)-v_{h_n}(y) \leq \frac{2C}{n}. 
\]
Therefore, there exists a $v : B(x,a) \rightarrow \mathbb{R}$ 
such that for each $y \in B(x,a)$ we have 
$$\lim_{n \rightarrow \infty} v_{h_n}(y)=v(y)$$
and 
$$\lim_{n \rightarrow \infty} v_{g_n}(y)=v(y).$$
Furthermore, on $B(x,a)$ this convergence 
is uniform. Therefore, we may conclude that $v$ is a harmonic function on $B(x,a)$. 
Also for each $n \in \mathbb{N}$, since $h_n \leq f \leq g_n$, for every $y \in B(x,a)$ we have 
\[
    v_{h_n}(y) \leq \underline{I}_y (f) \leq \overline{I}_y (f) \leq v_{g_n}(y).   
\]
Since both $v_{h_n}(y)$ and $v_{g_n}(y)$ converge to the same limit $v(y)$ as $n$ tends to infinity,
we conclude that $f$ is $I_y$-integrable and that $v_f(y)=\int_y f = v(y)$. Since $y\in B(x,a)$ 
is arbitrary we conclude that $v_f$ is defined and equal to $v$ on $B(x,a)$, and hence 
that $v_f$ is harmonic on $B(x,a)$. 

Let $z \in D$ be arbitrary. Since $D$, being a domain, is path connected, we pick a path from 
$x$ to $z$. A path in $\mathbb{R}^2$ is a compact set, and hence we may pick a finite string of overlapping open balls covering the path and repeat the argument
on each of them. As a consequence, we conclude that $v_f$ is harmonic on the neighbourhood around 
the path connecting
$x$ and $z$ and also, $f$ is (Daniell) integrable with respect to $I_z$. We may conclude that 
$f$ is (Daniell) integrable with respect to $I_z$ for every $z \in D$. Thus, we have proved that 
$f \in \mathcal{L}_x$ for some $x \in D$ implies $f \in \mathcal{L}_y$ for all $y \in D$. Furthermore, since
$v_f$ is harmonic in the neighbourhood of any path between two points in $D$,
we conclude that $v_f$ is harmonic in $D$. 
\end{proof}
This justifies our remark that the function $v_f$ can be seen as a solution to Problem \ref{prob:dirichlet}.
However, of note is that $v_f$ might not any longer be continuous on the boundary. 
This is discussed further in Wiener's paper \cite{wiener_dirichlet}.

\subsection{Stochastic processes}
Around the same time as his publication of \cite{wiener_dirichlet}, Wiener produced 
the first rigorous theory of Brownian motion. Today one defines Brownian 
motion using measure-theoretic tools on the space of continuous functions. 
However, at the time of Wiener, measure theory was still in its infancy. Furthermore, the 
probability axioms based on measure theory due to Andrey Kolmogorov was still a decade
away\footnote{
    Andrey Kolmogorov (1903-1987) presented the axioms for probability in the book 
    {\it Foundations of the Theory of Probability} which was 
    first published in German in 1933 \cite{kolmogorov1933grundlagen}.
}.
Consequently, Wiener developed the theory of Brownian motion with the use of the Daniell 
integral. \\

In his paper \cite{wiener_mean} Wiener mentioned that in Daniell's method,
defining the $I$-integral and a set $T_0$ of functions is left undetermined. 
To construct his Brownian motion, he therefore set out to choose a set $T_0$
and its corresponding $I$-integral. In this context, we take our set of elements
$X$ to be the set of continuous functions $f$ on the interval $[0,1]$ such that $f(0)=0$. 
Explicitly, we let
\[
    X=\set{f \in C([0,1])}{f(0)=0}.  
\]

Let $P_n= \lbrace t_0, t_1, \cdots, t_n \rbrace$ with 
\[
    0=t_0 <t_1< \cdots <t_n=1  
\]
be a partition of $[0,1]$ and let $\mathcal{B}_n= ( B_1, B_2, \cdots, B_n )$
be an ordered tuple of Borel sets of length $n$ in $\mathbb{R}$. 
Define $D(P_n, \mathcal{B}_n) \subset X$ 
by 
\[
    D(P_n, \mathcal{B}_n):= \set{f \in X}{f(t_i) \in B_i \text{ for each } i \in \lbrace 1,2,\cdots, n\rbrace }  
\]
Let $\mathcal{R}$ be 
\begin{align*}
    \mathcal{R}:= \lbrace D(P_n, \mathcal{B}_n) : \, &n \in \mathbb{N}, P_n \text{ is a partition of } [0,1], \\
    &\text{ and } \mathcal{B}_n \text{ is an ordered tuple of Borel sets.} \rbrace
\end{align*}

Define the function $\mu : \mathcal{R} \rightarrow [0,\infty]$ as follows: 
Let $D(P_n, \mathcal{B}_n) \in \mathcal{R}$ with $P_n= \lbrace t_0, t_1, \cdots t_n \rbrace$
a partition of $[0,1]$ and $\mathcal{B}_n= (B_1, B_2, \cdots B_n )$ a tuple 
of Borel sets. Define
\begin{align}   
    \mu D(P_n, \mathcal{B}_n):= \prod_{i=1}^n \frac{1}{\sqrt{2\pi (t_i-t_{i-1})}} 
    \int_{B_1} \int_{B_2} \cdots \int_{B_n} e^{-\frac{x_1^2}{t_1}}e^{-\frac{(x_2-x_1)^2}{t_2-t_1}}\cdots 
    e^{-\frac{(x_n-x_{n-1})^2}{t_n-t_{n-1}}} \, dx_1 dx_2 \cdots dx_n \label{eqn:wiener_measure}
\end{align}

We refer the readers to the Appendix (Definition \ref{def:sigma_ring}) for the definition of a ring.
\begin{proposition}
    The collection $\mathcal{R}$ is a ring of subsets of $X$. 
\end{proposition}
\begin{proof}
    We prove that $\mathcal{R}$ contains $X$ and that $\mathcal{R}$
     is closed under intersection and relative difference. 
     This would mean that $\mathcal{R}$ is an algebra of sets, 
     which is sufficient to conclude that $\mathcal{R}$ is a ring.

    To prove that $\mathcal{R}$ is closed under intersection, 
    let 
    $P_n=\lbrace t_0, \cdots, t_n \rbrace$ and $Q_m=\lbrace s_0, \cdots, s_m \rbrace$ be 
    partitions of $[0,1]$ and 
    $\mathcal{B}_n=( B_1, B_2, \cdots, B_n )$ and 
    $\mathcal{C}_m=( C_1, C_2, \cdots, C_m )$ 
    tuples of Borel sets and consider
    $D(P_n, \mathcal{B}_n), D(Q_m, \mathcal{C}_m) \in \mathcal{R}$. 
    Then $P_n\cup Q_m$ is also a partition of 
    $[0,1]$. Let the partition $P_n\cup Q_m$ be $\lbrace r_0, r_1, \cdots, r_l\rbrace$. 
    Define the tuple $\mathcal{D}_l= (D_1, D_2, \cdots, D_l)$ by 
    \[
        D_i= \begin{cases}
            B_j &\text{ if } r_i = t_j \in P_n \backslash Q_m\\
            C_j &\text{ if } r_i = s_j \in Q_m \backslash P_n\\
            B_j \cap C_k &\text{ if } r_i = t_j = s_k \in P_n \cap Q_m
        \end{cases}
    \]
    for $i = 1,2,\cdots l$. Then 
    \[
        D(P_n, \mathcal{B}_n) \cap D(Q_m, \mathcal{C}_m) = D(P_n \cup Q_m, \mathcal{D}_l) \in \mathcal{R}.
    \]

    To prove that $\mathcal{R}$ is closed under relative difference, 
    let $D(P_n, \mathcal{B}_n), D(Q_m, \mathcal{C}_m) \in \mathcal{R}$ be as in the previous case. 
    Define the tuple $\mathcal{D}_l= (D_1, D_2, \cdots, D_l)$ by 
    \[
        D_i= \begin{cases}
            B_j &\text{ if } r_i = t_j \in P_n \backslash Q_m\\
            \mathbb{R} \backslash C_j &\text{ if } r_i = s_j \in Q_m \backslash P_n\\
            B_j \backslash C_k &\text{ if } r_i = t_j = s_k \in P_n \cap Q_m
        \end{cases}
    \]
    for $i = 1,2,\cdots l$. Then 
    \[
        D(P_n, \mathcal{B}_n) \backslash D(Q_m, \mathcal{C}_m)
         = D(P_n \cup Q_m, \mathcal{D}_l) \in \mathcal{R}.
    \]

    Finally, note that $X=D(\lbrace 0,1 \rbrace, (\mathbb{R})) \in \mathcal{R}$. 
    Hence $\mathcal{R}$ is an algebra of sets.
    \end{proof}

    We refer the readers to the Appendix (Definition \ref{def:pre_measure}) for the definition of a 
    pre-measure.
    \begin{proposition}
        The function $\mu$ is a pre-measure.
    \end{proposition}
\begin{proof}
    From the definition, it is clear that $\mu (\emptyset)=0$, and that $\mu$ is nonnegative. Let 
    $D=(P_n, \mathcal{B}_n) \in \mathcal{R}$ be the disjoint union of the countable collection 
    $(D_k) \subset \mathcal{R}$. Since for all $k$, $D_k \subset D$ we may assume, without loss of generality, 
    that the partitions for the $D_k$'s are all the same. Hence we may write $D_k= D(P_n, \mathcal{A}_n^{(k)})$
    with $\mathcal{A}_n^{(k)}=(A_1^{k},A_2^{(k)},\cdots, A_n^{(k)})$ for $k \in \mathbb{N}$.
    The assumption that the $D_k$'s are disjoint means that the 
    sets $A_1^{(k)} \times A_2^{(k)} \times \cdots \times A_n^{(k)}$ are all disjoint and hence 
    that the iterated integral in \eqref{eqn:wiener_measure} may be written 
    \begin{align*}        
        & \int_{B_1} \int_{B_2} \cdots \int_{B_n} 
            e^{-\frac{x_1^2}{t_1}}e^{-\frac{(x_2-x_1)^2}{t_2-t_1}}\cdots 
    e^{-\frac{(x_n-x_{n-1})^2}{t_n-t_{n-1}}} \, dx_1 dx_2 \cdots dx_n \\
    &= \int_{\bigcup_{k}A_1^{(k)}} \int_{\bigcup_{k}A_2^{(k)}} \cdots \int_{\bigcup_{k}A_n^{(k)}} 
        e^{-\frac{x_1^2}{t_1}}e^{-\frac{(x_2-x_1)^2}{t_2-t_1}}\cdots 
    e^{-\frac{(x_n-x_{n-1})^2}{t_n-t_{n-1}}} \, dx_1 dx_2 \cdots dx_n \\
    &= \sum_{k} \int_{A_1^{(k)}} \int_{A_2^{(k)}} \cdots \int_{A_n^{(k)}} 
                e^{-\frac{x_1^2}{t_1}}e^{-\frac{(x_2-x_1)^2}{t_2-t_1}}\cdots 
                e^{-\frac{(x_n-x_{n-1})^2}{t_n-t_{n-1}}} \, dx_1 dx_2 \cdots dx_n.
    \end{align*}
    From this it follows that 
    \[
        \mu (D) = \mu \left( \bigcup_k D_k\right) = \sum_k \mu (D_k).  
    \]
    Hence we conclude that $\mu$ is a pre-measure.
        
\end{proof}
Up to this point, this is the same as the modern construction. What one would normally 
do from here is to apply the Carath\'eodory Extension theorem to get a measure 
$\mu_W$ on the $\sigma$-algebra generated by $\mathcal{R}$.
However, Wiener originally used Daniell's method to construct his 
integral of a functional \cite[Chapter 10]{wiener_differential}. We also use Daniell’s theory.
Let $T_0$ be 
\[
    T_0 = \text{span}\set{\chi_R}{R \in \mathcal{R} \text{ and } \mu (R) < \infty},
\]
and define $I: T_0 \rightarrow \mathbb{R}$ as: 
for each $ \sum_{i=1}^n c_i \chi_{R_i} \in T_0$ let 
\[
    I\left(\sum_{i=1}^n c_i \chi_{R_i} \right) := \sum_{i=1}^n c_i \mu (R_i). 
\]
The fact that $T_0$ is an extended vector lattice and $I$ an $I$-integral follows from 
Lemma \ref{lemma:simple}. We may therefore extend $T_0$ and $I$ to an extended vector lattice 
$\mathcal{L}_W$ 
and a Daniell Integral $\int_W :\mathcal{L}_W \rightarrow \mathbb{R}$, respectively. The measure obtained using 
Theorem \ref{thm:DS} is the Wiener measure.

\section{Appendix}
Here we list some standard results of measure theory. 
A complete treatment of these results can be found in \cite{de_barra}.

\begin{definition}\label{def:sigma_ring}
    Let $X$ be a set. A class $\mathcal{R}$ of subsets of $X$ is called a {\it ring} if 
     $E,F \in \mathcal{R}$ implies that 
     \begin{enumerate}[(i)]
         \item $E \cup F \in \mathcal{R}$; and 
         \item $E \backslash F \in \mathcal{R}$.
     \end{enumerate}
     Additionally, if $\mathcal{R}$ satisfies the property %
     \begin{enumerate}[(i)] \setcounter{enumi}{2}
        \item if $(R_n)$ is a sequence of elements of $\mathcal{R}$ 
            then $\bigcup_n R_n \in \mathcal{R}$,
     \end{enumerate}
     then we call $\mathcal{R}$ a {\it $\sigma$-ring}.
\end{definition}
\begin{remark}
    Let $\mathcal{R}$ be a $\sigma$-ring. If $E,F \in \mathcal{R}$, then it also follows that $E \cap F \in \mathcal{R}$ since we have the 
    identity $$E \cap F= E\backslash (E \backslash F).$$ 
\end{remark}
\begin{definition}\label{def:pre_measure}
    Let $\mathcal{R}$ be a $\sigma$-ring.
    A function $\mu : \mathcal{R} \rightarrow \mathbb{R} \cup \lbrace \infty \rbrace$ 
    is called a {\it pre-measure} if 
    \begin{enumerate} [(i)]
        \item $\mu$ is nonnegative;
        \item $\mu (\emptyset) =0$; and 
        \item if $(E_n)$ is a disjoint sequence of sets in $\mathcal{R}$ such that 
            $\bigcup_n E_n \in \mathcal{R}$, then 
            \[
                \mu \left( \bigcup_n E_n \right)= \sum_n \mu (E_n).
            \]
    \end{enumerate}
    If $\mu (X) < \infty$ or $X$ is the union of a sequence of sets $(E_n) \subset \mathcal{R}$ such that 
    $\mu (E_n) < \infty$ for each $n \in \mathbb{N}$ then we call the pre-measure $\sigma$-finite.
\end{definition}
\begin{definition}\label{def:sigma_alg}
    Let $X$ be a set. A class $\mathcal{A}$ of subsets of $X$ is called a {\it $\sigma$-algebra} if 
     \begin{enumerate}[(i)]
         \item $X \in \mathcal{A}$;
         \item if $E \in \mathcal{A}$ then $X \backslash E \in \mathcal{A}$; and
         \item if $(A_n)$ is a sequence of elements of $\mathcal{A}$ then $\bigcup_n A_n \in \mathcal{A}$.
     \end{enumerate} 
\end{definition}
\begin{definition}
    Let $\mathcal{A}$ be a $\sigma$-algebra.
    A function $\mu : \mathcal{A} \rightarrow \mathbb{R} \cup \lbrace \infty \rbrace$ 
    is called a {\it measure} if it is a pre-measure when $\mathcal{A}$ is viewed as a ring.
\end{definition}
\begin{definition} \label{def:complete_measure}
    Let $X$ be a set and  $\mathcal{A}$ a $\sigma$-algebra on $X$. 
    If $\mu: \mathcal{A} \rightarrow \mathbb{R} \cup \lbrace \infty \rbrace$
    is a measure on $X$ then we say it is {\it complete} if 
    for all sets $E \in \mathcal{A}$ such that 
    $\mu E=0$ we have $F \subset E$ implies $F \in \mathcal{A}$.
\end{definition}
\begin{definition}\label{def:sigma_gen} Let $X$ be a set.
    We call the smallest $\sigma$-algebra containing a 
    specific class $\mathcal{C}$ of subsets of $X$ 
    the {\it $\sigma$-algebra generated by $\mathcal{C}$} and
     denote it by $\sigma(\mathcal{C})$. If $X= \mathbb{R}$ and $\mathcal{C}$ is the set 
    of all finite open intervals, we call the $\sigma$-algebra 
    generated by $\mathcal{C}$ the {\it Borel
    $\sigma$-algebra}.
\end{definition}
\begin{theorem}[Carath\'eodory Extension Theorem]
    Let  $X$ be a set and $\mathcal{R}$ be a ring of subsets of $X$. 
    Let $\mu : \mathcal{R} \rightarrow \mathbb{R} \cup \lbrace \infty \rbrace$ be a pre-measure. 
    There exists an extension of $\mu$ to the $\sigma$-algebra generated by $\mathcal{R}$. Furthermore,
    if $\mu$ is $\sigma$-finite, then this extension is unique. 
\end{theorem}
\begin{remark}
     Recall that the Lebesgue $\sigma$-algebra $\mathcal{M}$ and Lebesgue measure $m$ on the set $\mathbb{R}$ 
     is obtained by applying the 
    Carath\'eodory extension theorem to the ring of intervals. Hence the Lebesgue measure 
    is a complete measure and $\mathcal{M}$ contains the Borel $\sigma$-algebra.
\end{remark}
\begin{definition}
    If $G \subset \mathbb{R}$, we say $G$ is a $G_\delta$-set if it is the intersection of countably many 
    open sets in $\mathbb{R}$. 
    If $F \subset \mathbb{R}$, we say $F$ is a $F_\sigma$-set if it is the union of countably many 
    closed sets in $\mathbb{R}$.
\end{definition}
\begin{theorem}[Regularity of the Lebesgue Measure] \label{thm:regularity}
    Let $\mathcal{M}$ be the $\sigma$-algebra of Lebesgue measurable subsets of $\mathbb{R}$. 
    Let $m: \mathcal{M}\rightarrow \mathbb{R} \cup \lbrace \infty \rbrace$ be the Lebesgue measure,
    and denote the Lebesgue outer measure by $m^*$.
    The following are equivalent for a set $E \subset \mathbb{R}$. 
    \begin{enumerate}[(i)]
        \item $E \in \mathcal{M}$.
        \item For all $\varepsilon >0$ there exists an open set $O \subset \mathbb{R}$ such that 
            $E \subset O$ and $m^*(O \backslash E) \leq \varepsilon$.
        \item There exists a $G_\delta$-set $G$ such that $E \subset G$ and $m^*(G\backslash E)=0$.
        \item For all $\varepsilon >0$ there exists a closed set $F \subset \mathbb{R}$ such that 
        $F \subset E$ and $m^*(E \backslash F) \leq \varepsilon$.
        \item There exists a $F_\sigma$-set $F$ such that $F \subset E$ and $m^*(E\backslash F)=0$.
    \end{enumerate}
\end{theorem}

\begin{theorem}\label{thm:appendix_completion}
    Let $\mathcal{M}$ be the Lebesgue $\sigma$-algebra on $\mathbb{R}$ and let $\mathcal{B}$ be the Borel $\sigma$-algebra. 
    The Lebesgue measure $m: \mathcal{M}\rightarrow \mathbb{R} \cup \lbrace \infty \rbrace$ 
    is complete. Furthermore, if $\mathcal{A}$ is another $\sigma$-algebra on $\mathbb{R}$ and 
    if $\mu : \mathcal{A}\rightarrow \mathbb{R} \cup \lbrace \infty \rbrace$ 
    is another complete measure such that $\mu = m$ on $\mathcal{B}$, then $\mathcal{M} \subset \mathcal{A}$ 
    and $\mu=m$ on $\mathcal{M}$.
\end{theorem}
\begin{proof}
    We defer the completeness of $m$ to \cite[Example 4, p. 30]{de_barra}. \\

    Let $E \in \mathcal{M}$. Then there exists a $G_\delta$-set $G$ such that $E \subset G$ and $m(G \backslash E)=0$.
    There also exists a $G_\delta$-set $D$ such that $G\backslash E \subset D$ and $m(D)=0$.
    Note that since $G\backslash E$ and $E$ are disjoint,
    \[
        m(G)= m(G \backslash E) + m(E) =m(E).
    \] 
    Since $D$ is a $G_\delta$-set, $D \in \mathcal{B}$, and we have that $m(D)=\mu (D)=0$. 
    By completeness of $\mu$ we conclude that $G \backslash E \in \mathcal{A}$. 
    Since $G$ is a $G_\delta$-set, we also 
    conclude that $G \in \mathcal{B}$ and that $m (G) = \mu (G)$. Hence,
     $E = G\backslash (G\backslash E) \in \mathcal{A}$ and 
     \[
        \mu (E)= \mu (G)- \mu(G \backslash E) = m (G) -0 = m(E).   
     \]
     Thus we conclude that $\mathcal{M} \subset \mathcal{A}$ and that the measures coincide on $\mathcal{M}$.
\end{proof}

\begin{theorem}[Dini's Theorem] \label{thm:dini}
    If $X$ is a compact topological space and $(f_n)$ is a monotone sequence of continuous 
    functions on $X$ that converges pointwise to a continuous limit $f$, then the convergence is uniform.
\end{theorem}

\begin{lemma} \label{lemma:disjoint_indicator}
    Let $\mathcal{R}$ be a ring of subsets of $X$.
    Let $\mu : \mathcal{R} \rightarrow \mathbb{R}\cup \lbrace \infty \rbrace$ be a pre-measure.
    Let 
    \[
        T := \text{span}\set{\chi_R}{R \in \mathcal{R} \text{ and } \mu (R) < \infty}.
    \]
    Define $I: T \rightarrow \mathbb{R}$ as follows: 
    For each $ \sum_{i=1}^n c_i \chi_{R_i} \in T$ let 
    \[
        I\left(\sum_{i=1}^n c_i \chi_{R_i} \right) := \sum_{i=1}^n c_i \mu (R_i). 
    \]
    If $x= \sum_{i=1}^n c_i \chi_{R_i} \in T$, then there exists 
        a collection of pairwise disjoint sets 
        $\lbrace B_1, B_2, \cdots, B_n \rbrace$ in $\mathcal{R}$ and real numbers 
        $\lbrace b_1, b_2,\cdots, b_n \rbrace$ such that $x= \sum_{k=1}^l c_k \chi_{B_k}$.
\end{lemma}
\begin{proof}
    The proof is by induction.
    If $n=1$, then we $x = c\chi_R$ for some $R \in \mathcal{R}$ and 
    some scalar $c$ and we are done. 
    Hence assume the claim is valid for an integer $n\geq 1$. 
    Let $x= \sum_{i=1}^{n+1} c_i \chi_{R_i} \in T$. Note that 
    \[
        \sum_{i=1}^{n+1} c_i \chi_{R_i}=\sum_{i=1}^{n} c_i \chi_{R_i} + c_{n+1}\chi_{R_{n+1}},
        \]
    and hence we apply the inductive hypothesis on 
    $\sum_{i=1}^{n} c_i \chi_{R_i}$ to get disjoint 
    sets $\lbrace B_1, B_2, \cdots, B_l \rbrace$ in $\mathcal{R}$ 
    and real numbers $\lbrace b_1, b_2, \cdots, b_l \rbrace$ such that 
    \[
        x=\sum_{i=1}^{l} b_k \chi_{B_k} + c_{n+1}\chi_{R_{n+1}}.
        \]
        Now consider the collection of sets 
    \[
        \left\lbrace B_1\backslash R_{n+1}, B_2\backslash R_{n+1}, \cdots, 
            B_l \backslash R_{n+1}\right\rbrace \bigcup
        \left\lbrace B_1\cap R_{n+1}, B_2\cap R_{n+1}, \cdots, B_l \cap R_{n+1} \right\rbrace \bigcup
        \left\lbrace R_{n+1}\backslash \bigcup_{k=1}^l B_k
        \right\rbrace.
    \]
    Note that the sets in this collection are pairwise disjoint, 
    and that by the properties of a ring, all of these 
    are in $\mathcal{R}$. For convenience, denote for 
    $k = 1,2, \cdots, l$ the set $D_k=B_k\backslash R_{n+1}$, 
    the set $E_k=B_k\cap R_{n+1}$, and denote by 
    $F=R_{n+1}\backslash \bigcup_{k=1}^l B_k$. Then we have the representation 
    \[
        x=\sum_{k=1}^l b_k \chi_{D_k} +\sum_{k=1}^l(b_k+c_{n+1})\chi_{E_k}+c_{n+1}\chi_F,  
        \]
    as a linear combination over disjoint sets. 
    We conclude that the claim holds for all $n \in \mathbb{N}$.
\end{proof}

\bibliographystyle{unsrt}
\bibliography{Amalgamation}

\end{document}